\documentclass[11pt]{amsart}
\usepackage{amssymb, amsfonts, amsthm, mathrsfs, hyperref}
\usepackage{graphicx}
\usepackage[dvipsnames]{xcolor}
\usepackage[normalem]{ulem} %
\usepackage{mathtools,subcaption,bbm} %
\usepackage[all]{xy}
\usepackage{enumerate}
\usepackage[leqno]{amsmath}
\usepackage{comment}  %
\usepackage{stmaryrd}

\usepackage{tikz}
\usetikzlibrary{decorations.markings}
\usetikzlibrary{decorations.pathmorphing}
\usetikzlibrary{arrows, snakes}
\usetikzlibrary{arrows.meta} %
\usetikzlibrary{shapes.geometric, calc,decorations.pathreplacing}
\usetikzlibrary{backgrounds}
\usetikzlibrary{fit}
\usetikzlibrary{decorations.pathreplacing}
\input xy
\xyoption{all}

\tikzset{middlearrow/.style={
		decoration={markings,
			mark= at position 0.5 with {\arrow{#1}} ,
		},
		postaction={decorate}
	}
}

\numberwithin{equation}{section}
\newtheorem{theorem}{Theorem}[section]
\newtheorem{lemma}[theorem]{Lemma}
\newtheorem{proposition}[theorem]{Proposition}

\newtheorem{definition}[theorem]{Definition}

\newtheorem{remark}[theorem]{Remark}
\newtheorem{example}[theorem]{Example}

\DeclareMathOperator*{\divg}{div}
\DeclareMathOperator*{\grad}{grad}

\DeclareMathOperator*{\argmin}{arg\,min}
\newcommand{\nnabla}{\nabla\!\!\!\!\nabla}
\newcommand{\llangle}{\langle\!\!\!\langle}
\newcommand{\rrangle}{\rangle\!\!\!\rangle}
\newcommand{\GGamma}{\Gamma \!\!\! \Gamma}

\newcommand{\dd}{\mathrm{d}}
\newcommand{\vol}{\mathrm{vol}}
\newcommand{\cexp}{c\text{-}\!\exp}
\newcommand{\R}{\mathbb{R}}
\newcommand{\X}{\mathcal{X}}
\newcommand{\Y}{\mathcal{Y}}
\newcommand{\M}{\mathcal{M}}
\newcommand{\cP}{\mathcal{P}}

\newcounter{sarrow}

\newcommand{\defeq}{\vcentcolon=}

\newcommand{\sop}{\ensuremath{\mathcal{P}}}
\def\indic#1{\ensuremath{\mathbbm{I}\left[{#1}\right]}} %
\definecolor{C0}{HTML}{1F77B4}
\definecolor{C1}{HTML}{FF7F0E}
\definecolor{C2}{HTML}{2CA02C}
\definecolor{C3}{HTML}{D62728}
\definecolor{C4}{HTML}{9467BD}
\definecolor{C5}{HTML}{8C564B}
\newcommand{\blue}[1]{\textcolor{C0}{#1}}

\newcommand{\Wt}{\mathscr{W}_2}
\newcommand{\BWD}{\mathscr{B}}
\newcommand{\Ms}{\mathbbm{M}} %
\newcommand{\D}{\mathcal{D}} %
\renewcommand{\L}{\mathcal{L}} %
\newcommand{\risk}{\mathcal{R}} %

\begin{document}

\title{Bregman-Wasserstein divergence: geometry and applications}

\author{Amanjit Singh Kainth}
\address{Department of Computer Science, University of Toronto}
\email{amanjitsk@cs.toronto.edu}

\author{Cale Rankin}
\address{School of Mathematics, Monash University}
\email{cale.rankin@monash.edu}

\author{Ting-Kam Leonard Wong}
\address{Department of Statistical Sciences, University of Toronto}
\email{tkl.wong@utoronto.ca}


\begin{abstract}
The Bregman-Wasserstein divergence is the optimal transport cost when the underlying cost function is given by a Bregman divergence, and arises naturally in fields such as statistics and machine learning. We establish fundamental properties of the Bregman-Wasserstein divergence and propose a novel generalized transport geometry that promotes the Bregman geometry to the space of probability distributions. We provide a probabilistic interpretation involving exponential families and define generalized displacement interpolations compatible with the Bregman geometry. These interpolations are used to derive a generalized Pythagorean inequality, which is of independent interest. Furthermore, we construct a generalized dualistic geometry that lifts the differential geometry of the Bregman divergence to an infinite-dimensional statistical manifold. On the computational side, we demonstrate how Bregman-Wasserstein optimal transport maps can be estimated using neural approaches, establish the well-posedness of Bregman-Wasserstein barycenters, and relate them to Bayesian learning. Finally, we introduce the Bregman-Wasserstein JKO scheme for discretizing Riemannian Wasserstein gradient flows.%
\end{abstract}

\maketitle{}

\section{Introduction} \label{sec:intro}
{\it Statistical divergences} define distance-like quantities between probability distributions \cite[Chapter 8]{PC19} that are not necessarily symmetric. Statistical divergences based on  {\it optimal transport} (OT) \cite{V03, V08} are widely applied in statistics, machine learning, finance, among other fields:
\begin{itemize}
\item[(i)] as {\it loss functions} for estimation or training, as in the {\it minimum Kantorovich estimator} \cite{BBR06} and the {\it Wasserstein generative adversarial network} \cite{ACB17};
\item[(ii)] to quantify model uncertainty, e.g.~in {\it distributionally robust optimization} \cite{DSW24} and stochastic analysis \cite{BBP24};
\item[(iii)] to provide geometric structures for {\it distribution-valued data} or {\it processes} \cite{CLM23, PZ20}.
\end{itemize}
Recently, OT has also been applied in information theory to study network information theory, rate distortion and quantum information, see for e.g., \cite{bai2023information, chuang2023infoot, de2021quantum}. Currently, the most popular OT-based divergence is the {\it Euclidean $2$-Wasserstein distance} $\mathscr{W}_2$ (or its square $\mathscr{W}_2^2$), defined for probability distributions $\mu_0$ and $\mu_1$ on $\mathbb{R}^d$ with finite second moment by
\begin{equation} \label{eqn:2.Wasserstein}
\mathscr{W}_2^2(\mu_0, \mu_1) := \inf_{\pi \in \Pi(\mu_0, \mu_1)} \int_{\mathbb{R}^d \times \mathbb{R}^d} |x_0 - x_1|^2 \dd \pi(x_0, x_1),
\end{equation}
where $\Pi(\mu_0, \mu_1)$ is the set of couplings of $\mu_0$ and $\mu_1$. By definition, a {\it coupling} of $\mu_0$ and $\mu_1$ is a joint probability distribution whose first and second marginals are $\mu_0$ and $\mu_1$ respectively. The $2$-Wasserstein distance enjoys many remarkable properties, making it a natural or convenient choice in diverse applications. These include capturing the underlying Euclidean geometry by the quadratic cost function in \eqref{eqn:2.Wasserstein}, deep connections with convex analysis via {\it Brenier's theorem} \cite{B91}, explicit geodesics given by McCann's {\it displacement interpolation} \cite{M97}, {\it Otto's calculus} that reformulates various partial differential equations as {\it Wasserstein gradient flows} \cite{O01}, as well as readily available numerical algorithms \cite{PC19, S15}. However, the Euclidean $2$-Wasserstein distance has limitations. In particular, as a metric, it does not reflect possible asymmetry in $\mu_0$ and $\mu_1$ (note that the {\it Kullback-Leibler (KL) divergence} ${\bf H}(\mu_0||\mu_1)$ is generally asymmetric). We remark that non-Euclidean OT costs are now actively investigated by the machine learning community; see for example \cite{UC23} and the references therein.

\begin{table}[t!]
	\begin{tabular}{|c|c|} \hline
		{\bf Euclidean} & {\bf Bregman}  \\ \hline
		$\frac{1}{2}|x_0 - x_1|^2$ & ${\bf B}_{\Omega}(x_0, x_1)$\\ \hline
		Euclidean geometry: $g_{ij} = \delta_{ij}$ &  Dualistic geometry: $g_{ij} = D_{ij}\Omega$ \\ %
		   & $y = D \Omega(x)$, $x = D \Omega^*(y)$ \\
          & (Sections \ref{sec:Bregman.geometry} and \ref{sec:Bregman.dualistic}) \\
           \hline

		 normal location family & exponential family \\
         & (Section \ref{sec:exp.family}) \\
         \hline
         Euclidean barycenter & Bregman barycenter \\
         & (Section \ref{sec:barycenters}) \\ \hline \hline
         {\bf Euclidean $2$-Wasserstein} & {\bf Bregman-Wasserstein (new)} \\
         \hline
		 $\frac{1}{2}\mathscr{W}_2^2(\mu_0, \mu_1)$ & $\mathscr{B}(\mu_0, \mu_1)$ \\
         & (Section \ref{sec:BW.definition}) \\ \hline
		 McCann's & Primal and dual\\
		 displacement interpolation & displacement interpolations \\
         & (Section \ref{sec:interpolations}) \\ \hline
		 Second order Otto calculus  & Generalized dualistic geometry \\
         (Section \ref{sec:prelim-otto-lott}) & (Sections \ref{sec:BW.Otto}--\ref{sec:parallel.transport}) \\ \hline
     Wasserstein barycenter & Bregman-Wasserstein barycenter \\
     & (Section \ref{sec:barycenters}) \\ \hline
	\end{tabular}
    \caption{\footnotesize The Bregman-Wasserstein divergence lifts the geometry of Bregman divergence to the space of probability distributions, in such a way that many properties of $\mathscr{W}_2$ have natural analogues. For example, we extend McCann's displacement interpolation to define primal and dual displacement interpolations which are geodesics under the generalized dualistic geometry. We also provide pointers to the relevant sections.} \label{tab:analogies}
\end{table}

In this paper, we undertake an in-depth study of the {\it Bregman-Wasserstein divergence} which is the optimal transport cost when the (ground) cost function is a {\it Bregman divergence} \cite{B67}; in Section \ref{sec:literature} we provide a detailed discussion of the related literature. Let $\Omega$ be a differentiable, strictly convex function defined on a convex domain $\X \subset \R^d$. The Bregman divergence ${\bf B}_{\Omega}(\cdot, \cdot)$ of $\Omega$ is defined as the error term of its linear approximation:
\begin{equation} \label{eqn:Bregman.divergence}
\mathbf{B}_{\Omega}(x_0, x_1) := \Omega(x_0) - \Omega(x_1) - D \Omega(x_1) \cdot (x_0 - x_1) \geq 0, \quad x_0, x_1
\in \X,
\end{equation}
where $D \Omega$ and $\cdot$ denote, respectively, the Euclidean gradient and dot product. Given probability distributions $\mu_0^{\X}$ and $\mu_1^{\X}$ on $\X$,\footnote{The use of the superscript $\X$ will be explained in Section \ref{sec:BW.definition}.}, the {\it Bregman-Wasserstein divergence} $\mathscr{B}_{\Omega}(\mu_0^{\X}, \mu_1^{\X})$ is defined by
\begin{equation} \label{eqn:BW.div.intro}
\mathscr{B}_{\Omega}(\mu_0^{\X}, \mu_1^{\X}) := \inf_{\pi^{\X} \in \Pi(\mu_0^{\X}, \mu_1^{\X})} \int_{\X \times \X} \mathbf{B}_{\Omega}(x_0, x_1) \dd \pi^{\X}(x_0, x_1).
\end{equation}
When $\Omega$ is the squared Euclidean norm $|x|^2 := x \cdot x$, $\mathbf{B}_{\Omega}$ recovers the squared Euclidean distance $|x_0 - x_1|^2$, and hence $\mathscr{B}_{\Omega}$ reduces to $\mathscr{W}_2^2$.
Bregman divergences have been applied extensively in probability, information theory and geometry, statistics, machine learning, and optimization, among other fields, see for example \cite{A16, AN00, BMDG05, AM03, blondel2020learning, CDS02, EK22, RM15} and the references therein. The Bregman-Wasserstein divergence \eqref{eqn:BW.div.intro} lifts the Bregman divergence to the space of probability distributions. Our main contribution in this paper is to show that the Bregman-Wasserstein divergence provides a natural and tractable generalization of the $2$-Wasserstein distance $\mathscr{W}_2$, in the sense that many useful properties of $\mathscr{W}_2$, some mentioned above, can be extended to our Bregman-Wasserstein setting. See Table \ref{tab:analogies} for an overview of some of our theoretical results and extensions.

\subsection{Outline} \label{sec:outline}
The purpose of this paper is two-fold. First, we establish fundamental properties of the Bregman-Wasserstein divergence, including a novel geometric structure on the space of probability distributions. Second, we implement the Bregman-Wasserstein divergence using state-of-the-art algorithms and present applications in statistics and machine learning.\footnote{Code available at \url{https://github.com/amanjitsk/bregman-wasserstein}.}

In Section \ref{sec:BW.divergence}, we define the Bregman-Wasserstein divergence after reviewing the duality and geometry of the Bregman divergence. We characterize the optimal transport map using Brenier's theorem (Proposition \ref{prop:solving.BW.transport}); this result shows that existing algorithms for the $2$-Wasserstsein distance can be readily adapted to the Bregman-Wasserstein divergence (more in Section \ref{sec:neural.OT}). We also give a {\it probabilistic interpretation} of the Bregman-Wasserstein divergence, inspired by the {\it Schr\"{o}dinger bridge problem} \cite{leonard2012schrodinger}, when the underlying Bregman generator originates from an exponential family (Section \ref{sec:exp.family}).

In Section \ref{sec:interpolations}, we define {\it primal} and {\it dual displacement interpolations} (Definition \ref{def:primal.dual.displacement.interpolation}) which extends the primal and dual geodesics of Bregman geometry to distribution-valued curves. We relate them to the {\it $c$-convex geometry} in optimal transport (Section \ref{sec:c.convex}) and prove a {\it generalized Pythagorean inequality} (Theorem \ref{thm:BW.Pyth}), which is of independent interest.

In Section \ref{sec:gen-dual} we give an in-depth study of the dualistic geometry of Bregman divergence -- in the sense of {\it information geometry} \cite{A16, AN00} -- lifted to the space of probability distributions. Building on the approach of Otto \cite{O01} and Lott \cite{L08} reviewed in Section \ref{sec:prelim-otto-lott}, we define on a suitable space $\mathcal{P}^{\infty}(\M)$ of probability distributions on a Bregman manifold $\M$ a triple $(\mathfrak{g}, \nnabla, \nnabla^*)$, where $\mathfrak{g}$ is Otto's Riemannian metric and $(\nnabla, \nnabla^*)$ is a pair of affine connections on $\mathcal{P}^{\infty}(\M)$ (Definition \ref{defn:conjugate-connections-PM}) which are conjugate with respect to $\mathfrak{g}$. Thus we may regard $(\mathcal{P}^{\infty}(\M), \mathfrak{g}, \nnabla, \nnabla^*)$ as an infinite-dimensional {\it statistical manifold} \cite{lauritzen1987statistical}. We also study the parallel transports and sectional curvatures under our generalized dualistic geometry, and show that our primal and dual displacement interpolations can be regarded as geodesics with respect to $\nnabla$ and $\nnabla^*$.

In Section \ref{sec:applications} we turn our attention to implementation and applications of the Bregman-Wasserstein divergence. We include several numerical experiments to showcase the utility of the Bregman-Wasserstein divergence. First, we adapt the {\it Neural Optimal Transport (NOT)} approach of \cite{{buzun2024erf}} to estimate optimal transport maps and primal/dual
displacement interpolations trained with neural networks (Section \ref{sec:neural.OT}). Second, we introduce and study \textit{Bregman-Wasserstein barycenters}: we prove their well-posedness, demonstrate their use in Bayesian statistics following the approach of \cite{backhoff-veraguas2022blw}, and use entropically regularized transport to compute example barycenters (Section \ref{sec:barycenters}). Third, we propose the {\it Bregman-Wasserstein JKO scheme} as a feasible discretization of Riemannian Wasserstein gradient flows (Section \ref{sec:Wasserstein.gradient.flow}), and state a convergence result from our follow-up work \cite{RW24}. Finally, we conclude and discuss promising future directions in Section \ref{sec:conclusion}. Further details of our numerical implementations can be found in \ref{sec:implementation.detail}. Some technical proofs are collected in Appendices \ref{sec:BW.barycenter.proof} and \ref{sec:appendix.proofs}.

\subsection{Related literature} \label{sec:literature}
There is a growing literature that explores optimal transport with Bregman-type costs, as well as Bregman-type divergences between probability distributions. It is clear from \eqref{eqn:BW.div.intro} that the Bregman-Wasserstein divergence is different from the integration of a univariate Bregman divergence between probability densities, as in the {\it $U$-divergence} introduced in \cite{murata2004information}. In optimal transport, the first use of the Bregman-Wasserstein divergence we are aware of is \cite{carlier2007monge} whose authors studied the symmetrized Bregman divergence ${\bf B}_{\Omega}(x_0, x_1) + {\bf B}_{\Omega}(x_1, x_0)$ and noted that Brenier's theorem \cite{B91} for the quadratic transport can be applied to the Bregman cost after a change of variables (see Proposition \ref{prop:solving.BW.transport} for a precise statement). In \cite{GHY17}, the Bregman-Wasserstein divergence was introduced to construct ambiguity sets for distributionally robust optimization;\footnote{In \cite{GHY17} the term Wasserstein-Bregman divergence is used, while in some other papers it is called the Bregman transport cost. Here we call it the {\it Bregman-Wasserstein divergence} to emphasize that it lifts the Bregman divergence to probability distributions.} another application in the training of generative adversarial networks (GANs) was given in \cite{guo2021relaxed}. In \cite{CE17}, transport inequalities in relation to log-concave distributions were studied. A Bregman-Wasserstein divergence was used to analyze a barotropic Navier-Stokes system in \cite{fanelli2020statistical}. More recently, the Bregman-Wasserstein divergence was used in \cite{AC21} to quantify the convergence of the mirror Langevin algorithm for sampling from probability distributions, to construct ambiguity sets in risk management~\cite{PV23, PVYY24} and derive optimization algorithms on the Wasserstein space \cite{bonet2024mirror, DKPS23, RW24}. We also note that a different approach was adopted in \cite{L21} to construct Bregman-type divergences on the (usual) Wasserstein space via displacement convex functionals.

\medskip
\noindent
{\it Notations.} We list in Table \ref{ta:notation} some notations used throughout the paper. We adopt the Einstein summation convention unless otherwise stated. That is, a sum over an index is implied if it appears both as a subscript and a superscript, e.g.~$g_{ij}v^iv^j = \sum_{i, j} g_{ij}v^iv^j$.

\begin{table}[h!]
\begin{center}
\begin{tabular}{l|l}

  Symbol & Meaning\\ \hline
  $D, D^2$ & Euclidean gradient and Hessian \\
  $D_i, D_{ij}$ & Euclidean partial derivatives \\
  $a \cdot b$ & Euclidean dot product \\
  $\mathcal{M}$ & (Riemannian or Bregman) manifold\\
${\bf B}$ & Bregman divergence on $\mathcal{M}$\\
  $g \text{ or } \langle\cdot,\cdot\rangle$ & Riemannian metric on $\mathcal{M}$\\
$\text{grad}$ & Riemannian gradient \\
$\text{div}$ & Riemannian divergence \\
  $\overline{\nabla}$ & Levi-Civita connection on $\mathcal{M}$ \\
  $\nabla, \nabla^*$ & Primal and dual connection on $\mathcal{M}$ \\
$\mathscr{B}$ & Bregman-Wasserstein divergence on $\mathcal{P}(\mathcal{M})$\\
$\mathcal{P}^{\infty}(\mathcal{M})$ & Space of probability measures on $\mathcal{M}$\\
& with suitable regularity conditions\\
  $\mathfrak{g} \text{ or } \llangle \cdot , \cdot \rrangle$ & Otto's Riemannian metric on $\mathcal{P}^{\infty}(\mathcal{M})$\\
  $\overline{\nnabla}$ & Levi-Civita connection on $\mathcal{P}^{\infty}(\mathcal{M})$\\
  $\nnabla, \nnabla^*$ & Primal and dual connections on $\mathcal{P}^{\infty}(\mathcal{M})$
\end{tabular}
\end{center}
\caption{\footnotesize List of notations.}\label{ta:notation}
\end{table}

\section{Bregman-Wasserstein divergence} \label{sec:BW.divergence}
After reviewing the Bregman divergence and its geometry, we define the Bregman-Wasserstein divergence and provide two comparisons with $2$-Wasserstein distances: the Euclidean one and a {\it Riemannian} one. Both viewpoints will be helpful in our later development. Then, we give a probabilistic interpretation when the underlying Bregman divergence originates from an exponential family.

\subsection{Bregman geometry} \label{sec:Bregman.geometry}
We refer the reader to \cite{A01} where further details about the Bregman divergence can be found. We take for granted some results from convex analysis \cite{R70} and definitions (e.g.~manifold, tangent vector and Riemannian metric) from differential geometry \cite{L18}.

\medskip

Let $\X$ be an open convex subset of $\mathbb{R}^d$, $d \geq 1$. We will refer to $\X$ as the {\it primal domain}.

\begin{definition} [Regular Bregman generator] \label{def:Bregman.generator}
A convex function $\Omega: \X \rightarrow \mathbb{R}$ is said to be a regular Bregman generator on $\X$ if (i) $\Omega$ is smooth,\footnote{The smoothness condition is used mainly to simplify the exposition. For most results it is sufficient to assume that $\Omega$ is $C^2$.} (ii) the Hessian $D^2 \Omega$ is strictly positive definite everywhere on $\X$, and (iii) the range $\Y := D \Omega(\X) \subset \mathbb{R}^d$ of the gradient map is convex.
\end{definition}

Henceforth we let $\Omega$ be a regular Bregman generator on $\X$. Then, the {\it mirror map} $D \Omega$ is a diffeomorphism from $\X$ onto its range $\mathcal{Y}$; the set $\mathcal{Y}$, which we refer to as the {\it dual domain}, is also an open set which is convex by (iii). The inverse of $D \Omega$ is given on $\Y$ by $(D \Omega)^{-1} = D \Omega^*$, where
\[
\Omega^*(y) = x \cdot y - \Omega(x), \quad y = D\Omega(x) \in \Y,
\]
is the {\it Legendre transformation} of $\Omega$. One can verify that $\Omega^*$ is a regular Bregman generator on $\Y$. Thus, $(\X, \Omega)$ and $(\Y, \Omega^*)$ play symmetric roles in our theory. For $x \in \X$ and $y \in \Y$, we have the {\it Fenchel-Young inequality}
\begin{equation*}
\Omega(x) + \Omega^*(y) - x \cdot y \geq 0,
\end{equation*}
with equality if and only if $y = D \Omega(x)$ or, equivalently, $x = D\Omega^*(y)$.

Consider the Bregman divergence ${\bf B}_{\Omega}: \X \times \X \rightarrow \mathbb{R}_+$ of $\Omega$ defined by \eqref{eqn:Bregman.divergence}. From the strict convexity of $\Omega$ we see that ${\bf B}_{\Omega}(x, x') \geq 0$ with equality if and only if $x = x'$. However, ${\bf B}_{\Omega}$ is generally asymmetric and hence is not (a transformation of) a metric. Using the identity $\Omega(x) + \Omega^*(y) \equiv x \cdot y$ which holds for $x \in \X$ and $y = D\Omega (x) \in \Y$, we may express
\begin{equation} \label{eqn:Bregman.self.dual}
{\bf B}_{\Omega}(x, x') = \Omega(x) + \Omega^*(y') - x \cdot y' = {\bf B}_{\Omega^*}(y', y),
\end{equation}
where $x, x' \in \X$ and $y = D\Omega(x), y' = D \Omega(x') \in \Y$. We call \eqref{eqn:Bregman.self.dual} the {\it self-dual representation} of Bregman divergence \cite[Theorem 1.1]{A16}.

\begin{figure}[t!]
	\begin{tikzpicture}[scale = 0.65]
		\draw (0,0) ellipse (1.7 and 1.25);
		\node [above] at (0, 1.31) {\footnotesize Bregman manifold $\M$};

		\node[circle, draw=black, fill = black, inner sep=0pt, minimum size=3pt, label = right: {\footnotesize $p$}] at (-0.7, 0.1)  {};
		\draw (-4,-3.7) ellipse (1.7 and 1.25);

		\node [below] at (-4, -5) {\footnotesize primal domain $\X$};

		\node[circle, draw=black, fill = black, inner sep=0pt, minimum size=3pt, label = below: {\footnotesize $x_p$}] at (-4.5, -3.5)  {};
		\draw (4, -3.7) ellipse (1.7 and 1.25);
		\node [below] at (4, -5) {\footnotesize dual domain $\Y$};
		\node[circle, draw=black, fill = black, inner sep=0pt, minimum size=3pt, label = right: {\footnotesize $y_p$}] at (4.1, -3.6)  {};
		\draw[->] (-2.1, -3.6) -- (2.1, -3.6);
		\node[above] at (0, -3.6) {\footnotesize $D\Omega$};
		\draw[->] (2.1, -3.9) -- (-2.1, -3.9);
		\node[below] at (0.12, -3.9) {\footnotesize $D\Omega^*$};

		\draw[->] (-1.6, -1.2) -- (-3.3, -2.3);
		\node[left] at (-2.2, -1.4) {\footnotesize $\iota$};
		\draw[->] (1.6, -1.2) -- (3.3, -2.3);
		\node[right] at (2.2, -1.4) {\footnotesize $\iota^*$};

	\end{tikzpicture}
	\caption{\footnotesize A Bregman generator $\Omega$ induces two coordinate systems related by the mirror map $D\Omega$. Here $\iota: \M\rightarrow \X$ and $\iota^* : \M \rightarrow \Y$ are, respectively, the primal and dual coordinate maps, i.e., $\iota(p) = x_p$ and $\iota^*(p) = y_p$. They are related by $\iota^* = D\Omega \circ \iota$ or equivalently $\iota = D\Omega^* \circ \iota^*$.}%
 \label{fig:Bregman.manifold}
\end{figure}

Motivated by the diffeomorphism $D\Omega : \X \rightarrow \Y$ and the identity \eqref{eqn:Bregman.self.dual}, it is helpful to think of $x$ and $y = D \Omega(x)$ as {\it alternative representations of the same point}. Explicitly, we introduce a $d$-dimensional smooth manifold $\M$ which as a set is equal to $\X$ or $\Y$, on which $x \in \X$ (the {\it primal variable}) and $y := D\Omega(x) \in \Y$ (the {\it dual variable}) serve as global coordinates; see Figure \ref{fig:Bregman.manifold}. We let $\iota: \M \rightarrow \X$ and $\iota^* : \M \rightarrow \Y$ be the {\it primal} and {\it dual coordinate maps}. For $p \in \M$, let $x_p := \iota(p) \in \X$ and $y_p := \iota^*(p) = D\Omega(x_p) \in \Y$ be respectively the {\it primal} and {\it dual coordinates} of $p$. Using \eqref{eqn:Bregman.self.dual}, we define a divergence ${\bf B}: \M \times \M \rightarrow \mathbb{R}_+$ by
\begin{equation} \label{eqn:M.divergence}
{\bf B}(p, q) := \Omega(x_p) + \Omega^*(y_q) - x_p \cdot y_q.
\end{equation}
From \eqref{eqn:Bregman.self.dual}, ${\bf B}_{\Omega}$ and ${\bf B}_{\Omega^*}$ become coordinate representations of ${\bf B}$, which we call the {\it Bregman divergence} on $\M$. We equip $\M$ with the primal and dual coordinates and the divergence ${\bf B}$; we call it the {\it Bregman manifold}.

\begin{remark}
Although it is common to think of $\X$ and $\Y$ as separate spaces as in Figure \ref{fig:Bregman.manifold}, it is also helpful to regard both as subsets of a {\it single} ambient Euclidean space $\R^d$. Completing the square in \eqref{eqn:M.divergence}, we may express
\begin{equation} \label{eqn:Bregman.as.Euclidean.shifted}
\begin{split}
{\bf B}(p, q) &= \frac{1}{2}|x_p - y_q|^2 + \left(\Omega(x_p) - \frac{1}{2}|x_p|^2\right) + \left( \Omega^*(y_q) - \frac{1}{2}|y_q|^2\right) \\
&=: \frac{1}{2}|x_p - y_q|^2 + \widetilde{\Omega}(x_p) + \widetilde{\Omega^*}(y_q).
\end{split}
\end{equation}
That is, the Bregman divergence between $p$ and $q$ is equal to half of the {\it Euclidean} squared distance between $x_p$ and $y_q$ (in so-called {\it mixed coordinates}), shifted by a term involving only $p$ and a term involving only $q$. These shifts are needed since $p = q$ does not imply $|x_p - y_q| = 0$. In Section \ref{sec:BW.definition}, we will see that this identity allows us to relate the Bregman-Wasserstein divergence with the Euclidean $2$-Wasserstein distance.
\end{remark}

Next we describe the geometry of Bregman divergence which consists of a Riemannian metric and two notions of geodesics. A complete differential-geometric description, which allows us to generalize to the space of probability distributions, involves the {\it dualistic structure} in information geometry and will be reviewed in Section \ref{sec:Bregman.dualistic}.

Although the divergence  ${\bf B}$ is generally non-Euclidean, it is ``locally quadratic''. Specifically, given a tangent vector $v \in T_p\M$, the expression
\begin{equation} \label{eqn:M.metric}
\|v\|_p^2 = \left. \frac{\dd^2 }{\dd t^2} \right|_{t = 0} {\bf B}(\gamma_0, \gamma_t) = \left. \frac{\dd^2 }{\dd t^2} \right|_{t = 0} {\bf B}(\gamma_t, \gamma_0),
\end{equation}
where $\gamma$ is a curve with $\gamma(0) = p$ and $\dot{\gamma}(0) = v$, defines via the polarization identity a Riemannian metric $g$ on $\M$. Thus $\textbf{B}$ is not a metric but is locally approximated by a squared Riemannian distance. In coordinates, we have
\begin{equation} \label{eqn:Bregman.metric}
g\left( \left. \frac{\partial}{\partial x^i} \right|_p, \left. \frac{\partial}{\partial x^j} \right|_p\right) = D_{ij} \Omega(x_p), \quad g\left( \left. \frac{\partial}{\partial y^i} \right|_p, \left. \frac{\partial}{\partial y^j} \right|_p\right) = D_{ij} \Omega^*(y_p).
\end{equation}
Riemannian metrics of the form \eqref{eqn:Bregman.metric} are known as {\it Hessian} \cite{SY97}. In fact, the mirror descent algorithm can be regarded as a discretized Riemannian gradient flow with respect to $g$, see \cite{RM15}. We also note that
\begin{equation} \label{eqn:metric.inverse}
D^2 \Omega^*(y_p) = \left( D^2 \Omega(x_p) \right)^{-1}, \quad p \in \M.
\end{equation}

By a {\it primal geodesic} we mean a curve $(\gamma_t)_{0 \leq t \leq 1}$ on $\M$ which is a straight line under the primal coordinates, i.e., $x_{\gamma_t} = \iota(\gamma_t) = (1 - t)x_{\gamma_0} + tx_{\gamma_1} \in \X$. Using the dual coordinate system, we define analogously the {\it dual geodesic} given by $y_{\gamma_t} = \iota^*(\gamma_t) = (1 - t) y_{\gamma_0} + t y_{\gamma_1} \in \Y$. Intuitively, two kinds of geodesics are needed to capture the asymmetry of the Bregman divergence. An important consequence of \eqref{eqn:metric.inverse} is that the primal and dual coordinate vector fields are {\it biorthogonal} with respect to $g$:
\begin{equation} \label{eqn:biorthogonal}
g\left( \frac{\partial}{\partial x^i}\Big|_p, \frac{\partial}{\partial y^j}\Big|_p \right) = \left\{\begin{array}{lr}
        1, & \text{if } i = j;\\
        0, & \text{if } i \neq j.
\end{array}\right.
\end{equation}
It follows that if $u = a^i \frac{\partial}{\partial x^i}|_p$ and $v = b^j \frac{\partial}{\partial y^j}|_p$ are tangent vectors at $p$, then $g(u, v) = \sum_i a^ib^i = a \cdot b$. Thus, Riemannian computations are substantially simplified by using both the primal and dual coordinate systems.

\begin{remark}
Unless the Bregman generator $\Omega$ is quadratic, the primal and dual geodesics are {\it not} Riemannian geodesics (with respect to the Levi-Civita connection $\overline{\nabla}$ induced by the metric $g$), and are not distance-minimizing. The primal and dual geodesics are motivated by their naturalness in applications. For example, they show up naturally in the generalized Pythagorean theorem (see Theorem \ref{thm:Bregman.Pyth} below) which characterizes projections with respect to the Bregman divergence. Also, in the context of exponential family (see Section \ref{sec:exp.family}), the primal and dual geodesics correspond respectively to the exponential and mixture interpolations.
\end{remark}

The {\it generalized Pythagorean theorem} is fundamental in applications of the Bregman divergence. In Section \ref{sec:Pyth} we will obtain a generalized Pythagorean {\it inequality} for the Bregman-Wasserstein divergence.

\begin{theorem}[Generalized Pythagorean theorem for Bregman divergence] \label{thm:Bregman.Pyth}
Let $p \in \M$. Let $(\gamma_t)_{0 \leq t \leq 1}$ be a primal geodesic with $\gamma_0 = p$ and $(\sigma_t)_{0 \leq t \leq 1}$ be a dual geodesic with $\sigma_0 = p$. Then for $0 \leq t \leq 1$ we have
\begin{equation} \label{eqn:Bregman.Pyth}
{\bf B}(p, \sigma_t) + {\bf B}(\gamma_t, p) - {\bf B}(\gamma_t, \sigma_t) = t^2 g(\dot{\gamma}_0, \dot{\sigma}_0).
\end{equation}
In particular, the generalized Pythagorean relation ${\bf B}(p, \sigma_t) + {\bf B}(\gamma_t, p) = {\bf B}(\gamma_t, \sigma_t)$ holds if and only if $\gamma$ and $\sigma$ are orthogonal at $p$.
\end{theorem}
\begin{proof}
We recall the proof which is simple (once formulated) and instructive. By construction of $\gamma$ and $\sigma$, we may write in coordinates
\[
x_{\gamma_t} = x_p + ta \in \X, \quad y_{\sigma_t} = y_p + tb \in \Y,
\]
where $\dot{\gamma}_0 = a^i \frac{\partial}{\partial x^i}|_p$ and $\dot{\sigma}_0 = b^j \frac{\partial}{\partial y^j}|_p$ are tangent vectors at $p$.

By the self-dual representation \eqref{eqn:Bregman.self.dual} of ${\bf B}$ and the Fenchel-Young identity $\Omega(x) + \Omega^*(y) \equiv x \cdot y$, we obtain after some algebraic simplifications
\begin{equation} \label{eqn:Bregman.Pyth.identity}
\begin{split}
&{\bf B}(p, \sigma_t) + {\bf B}(\gamma_t, p) - {\bf B}(\gamma_t, \sigma_t)\\
&= \left( \Omega(x_p) + \Omega^*(y_{\sigma_t}) - x_p \cdot y_{\sigma_t} \right) + \left( \Omega(x_{\gamma_t}) + \Omega^*(y_p) - x_{\gamma_t} \cdot y_p \right) \\
&\quad - \left( \Omega(x_{\gamma_t}) + \Omega^*(y_{\sigma_t}) - x_{\gamma_t} \cdot y_{\sigma_t} \right) \\
&= t^2 a \cdot b = g(\dot{\gamma}_0, \dot{\sigma}_0),
\end{split}
\end{equation}
where the last equality follows from \eqref{eqn:biorthogonal}.
\end{proof}

We end this subsection with two examples of Bregman generators. More examples can be found in \cite{A16, BMDG05} and the references therein.

\begin{example}[Finite simplex] \label{eg:simplex.cgf}
Let $\X = \mathbb{R}^d$. The function
\begin{equation} \label{eqn:simplex.cgf}
\Omega(x) = \log \left(1 + \sum_{i = 1}^d e^{x^i} \right)
\end{equation}
defines a regular Bregman generator on $\X$. The dual variable is given by
\[
y = D \Omega(x) = \left( \frac{e^{x^1}}{1 + \sum_{j = 1}^d e^{x^j}}, \ldots, \frac{e^{x^d}}{1 + \sum_{j = 1}^d e^{x^j}} \right),
\]
and takes values in $\Y = D \Omega(\X) = \left\{ y \in (0, 1)^d: \sum_{i = 1}^d y^i < 1 \right\}$. Note that if $y^0 = 1 - \sum_{i = 1}^d y^i$ then $\{(y^0, \ldots, y^d): y \in \Y\}$ is the open unit simplex in $\mathbb{R}^{1 + d}$. In fact, $x$ is the natural parameter of the simplex regarded as an exponential family (see Section \ref{sec:exp.family}), and is given by $x^i = \log \frac{y^i}{y^0}$. The dual function $\Omega^*$ is the negative Shannon entropy given by
\[
\Omega^*(y) = \sum_{i = 0}^d y^i \log y^i, \quad y \in \Y.
\]
By a straightforward computation, we see that
\begin{equation} \label{eqn:simplex.KL}
{\bf B}(p, q) = {\bf B}_{\Omega^*}(y_q, y_p) = \sum_{i = 0}^d y_q^i \log \frac{y_q^i}{y_p^i} = {\bf H}(y_q || y_p),
\end{equation}
which is the (reverse) KL-divergence on the simplex. The induced Riemannian metric is the Fisher-Rao metric on the simplex.
\end{example}

\begin{example}[Hopfield neural network] \label{eg:Hopfield}
Let $\X = \mathbb{R}^d$ and let $\sigma_1, \ldots, \sigma_d : \mathbb{R} \rightarrow (0, 1)$ be increasing diffeomorphisms. By direct integration, we can construct a Bregman generator $\Omega$ on $\X$ whose mirror map is given by $D \Omega(x) = (\sigma_1(x^1), \ldots, \sigma_d(x^d))$. In \cite{HCTM20}, the transformations $\sigma_i$ are regarded as the activation functions of a continuous-time continuous-state Hopfield neural network. Under the primal coordinate system, the coefficients of the induced metric $g$ has diagonal form and is given by $g_{ij}(x) = \delta_{ij} \sigma_i'(x^i)$.
\end{example}

\subsection{Bregman-Wasserstein divergence: definition and first
properties} \label{sec:BW.definition}
Let $\M$ be a Bregman manifold with Bregman divergence ${\bf B}$. Given a topological space $E$ we let $\cP(E)$ be the space of Borel probability measures on $E$. Analogous to points, each probability measure on $\M$ has a primal representation
and a dual representation. Given $\mu \in \cP(\M)$, we let the pushforwards\footnote{If $\mu$ is a measure and $T$ is a measurable mapping defined on the state space of $\mu$, the {\it pushforward} $T_{\#}\mu$ is defined by $T_{\#} \mu = \mu \circ T^{-1}$.} $\mu^{\X} = \iota_{\#} \mu \in \cP(\X)$ and $\mu^{\Y} = \iota_{\#}^* \mu \in \cP(\Y)$ be respectively the {\it primal} and {\it dual representations} of $\mu$. Since $\iota^* = D \Omega \circ \iota$, we have
\begin{equation} \label{eqn:measure.mirror.map}
\mu^{\Y} = (D \Omega)_{\#} \mu^{\X}.
\end{equation}
We may regard this as the {\it mirror map} for distributions. It is used for example in the definition of the {\it mirrored Langevin dynamics} \cite{hsieh2018mirrored}.

We first give the coordinate-free definition of the Bregman-Wasserstein divergence, then state its primal and dual representations.
Given $\mu_0, \mu_1 \in \cP(\M)$, consider the Monge-Kantorovich optimal transport problem with the non-negative and smooth cost function $c(p, q) = {\bf B}(p, q)$. (For introductions to optimal transport see \cite{S15, V03, V08}.) Recall that given probability measures $\mu_0$ and $\mu_1$ on possibly different measurable spaces, the set of {\it couplings} $\Pi(\mu_0, \mu_1)$ of $(\mu_0, \mu_1)$ is defined by the set of all probability measures on the product space with marginals $\mu_0$ and $\mu_1$.

\begin{definition}[Bregman-Wasserstein divergence] \label{def:BW.div}
For $\mu_0, \mu_1 \in \cP(\M)$, we define the Bregman-Wasserstein divergence by
\begin{equation} \label{eqn:Bregman.Wasserstein.divergence}
\mathscr{B}(\mu_0, \mu_1) := \inf_{\pi \in \Pi(\mu_0, \mu_1)} \int_{\M \times \M} {\bf B}(p, q) \mathrm{d} \pi(p, q).
\end{equation}
\end{definition}

From the self-dual representation \eqref{eqn:Bregman.self.dual}, we have
\begin{equation} \label{eqn:BW.divergence.coordinates}
\begin{split}
&\mathscr{B}(\mu_0, \mu_1)\\
&=\inf_{\pi^{\X} \in \Pi(\mu_0^{\X}, \mu_1^{\X})} \int_{\X \times \X} {\bf B}_{\Omega}(x, x') \dd \pi^{\X}(x, x') =: \mathscr{B}_{\Omega}(\mu_0^{\X}, \mu_1^{\X}) \\
&= \inf_{\pi^{\Y} \in \Pi(\mu_1^{\Y}, \mu_0^{\Y})} \int_{\Y \times \Y} {\bf B}_{\Omega^*}(y', y) \dd \pi^{\Y}(y', y) =: \mathscr{B}_{\Omega^*}(\mu_1^{\Y}, \mu_0^{\Y}),
\end{split}
\end{equation}
where $\mathscr{B}_{\Omega}(\mu_0^{\X}, \mu_1^{\X})$ is the {\it primal representation} and $\mathscr{B}_{\Omega^*}(\mu_1^{\Y}, \mu_0^{\Y})$ is the {\it dual representation}. The identity $\mathscr{B}_{\Omega}(\mu_0^{\X}, \mu_1^{\X}) \equiv \mathscr{B}_{\Omega^*}(\mu_1^{\Y}, \mu_0^{\Y})$ captures the asymmetry of the Bregman-Wasserstein divergence. The coordinate-free expression \eqref{eqn:Bregman.Wasserstein.divergence} is handy in our theoretical development. The primal and dual representations \eqref{eqn:BW.divergence.coordinates} are indispensable for computation and implementation (see Section \ref{sec:applications}).

\begin{example}[Self-dual case]\label{eg:reduces.to.Euclidean.W2}
Let $\X = \mathbb{R}^d$ and $\Omega(x) = \frac{1}{2}|x|^2$. Since $\Omega^* = \Omega$ and $D\Omega = \mathrm{Id}$, we have $\X = \Y$ and
\begin{equation} \label{eqn:BW.quadratic.case}
\mathscr{B}(\mu_0, \mu_1) = \frac{1}{2} \mathscr{W}_2^2(\mu_0^{\X}, \mu_1^{\X}) = \frac{1}{2}\mathscr{W}_2^2(\mu_0^{\Y}, \mu_1^{\Y}),
\end{equation}
which is one half of the squared $2$-Wasserstein distance.
\end{example}

\begin{example}
Consider the KL-divergence on the simplex as in Example \ref{eg:simplex.cgf}. The induced Bregman-Wasserstein divergence may be regarded as a ``transport KL-divergence'' between probability measures on the simplex, and will be used as a key example throughout the paper.
\end{example}

\subsubsection{Comparison with the Euclidean $2$-Wasserstein distance}
Consider the optimal transport problem \eqref{eqn:Bregman.Wasserstein.divergence} whose cost function is the Bregman divergence ${\bf B}$. Recall from \eqref{eqn:Bregman.as.Euclidean.shifted} the identity
\begin{equation} \label{eqn:Bregman.as.Euclidean.shifted2}
{\bf B}(p, q) = \frac{1}{2}|x_p - y_q|^2 + \widetilde{\Omega}(x_p) + \widetilde{\Omega^*}(y_q),
\end{equation}
where $\widetilde{\Omega}(x_p) = \Omega(x_p) - \frac{1}{2}|x_p|^2$ depends only on $p$ and $\widetilde{\Omega^*}(y_q) = \Omega^*(y_q) - \frac{1}{2}|y_q|^2$ depends only on $q$. For any coupling $\pi \in \Pi(\mu_0, \mu_1)$, we have
\begin{equation} \label{eqn:BW.Euclidean.computation}
\begin{split}
&\int_{\M \times \M} {\bf B}(p, q) \dd \pi(p, q) \\
&= \frac{1}{2} \int_{\M \times \M} |x_p - y_q|^2 \dd \pi(p, q) + \int_{\X} \widetilde{\Omega} \dd \mu_0^{\mathcal{X}} + \int_{\Y} \widetilde{\Omega^*} \dd \mu_1^{\mathcal{Y}} \\
&= \frac{1}{2} \int_{\R^d \times \R^d} |x - y'|^2 \dd \pi^{\X, \Y}(x, y') + \int_{\X} \widetilde{\Omega} \dd \mu_0^{\mathcal{X}} + \int_{\Y} \widetilde{\Omega^*} \dd \mu_1^{\mathcal{Y}},
\end{split}
\end{equation}
where $\pi^{\X, \Y} := (\iota, \iota^*)_{\#} \pi \in \Pi(\mu_0^{\X}, \mu_1^{\Y})$ couples $\mu_0$ and $\mu_1$ in {\it mixed coordinates}. As observed in \cite{carlier2007monge} and \cite{GHY17}, assuming that the integrals $\int_{\X} \widetilde{\Omega}(x) \dd \mu_0^{\mathcal{X}}(x)$ and $ \int_{\Y} \widetilde{\Omega^*}(y) \dd \mu_1^{\mathcal{Y}}(y)$ exist, the OT problem \eqref{eqn:Bregman.Wasserstein.divergence} is equivalent to the $2$-Wasserstein transport (with the quadratic cost function $|x - y'|^2$) between $\mu_0^{\X}$ and $\mu_1^{\Y}$ on $\R^d$, and hence can be solved in terms of Brenier's theorem \cite{B91}. We make the above ideas precise in the following proposition.

\begin{proposition}[Characterization of optimal transport map] \label{prop:solving.BW.transport}
Let $\mu_0, \mu_1 \in \mathcal{P}(\M)$. Assume:
\begin{itemize}
    \item[(i)] $\mu_0^{\X}$ is absolutely continuous with respect to the Lebesgue measure on $\X$ (or equivalently $\mu_0$ is absolutely continuous with respect to the Riemannian volume measure on $\M$);
    \item[(ii)] $\Omega \in L^1(\mu_0^{\X})$ and $\Omega^* \in L^1(\mu_1^{\Y})$.
\end{itemize}
Let $D f$ be the ($\mu_0^{\X}$-a.e.) unique convex gradient (Brenier map) which is $\mathscr{W}_2$ optimal for the pair $(\mu_0^{\X}, \mu_1^{\Y})$, that is
\begin{equation} \label{eqn:BW.Brenier}
\int_{\X} |x - Df(x)|^2 \dd \mu_0^{\X}(x) = \mathscr{W}_2^2(\mu_0^{\X}, \mu_1^{\Y}).
\end{equation}
Note that $D f(x) \in \Y$ for $\mu_0^{\X}$-almost all $x \in \X$. Then the optimal transport problem \eqref{eqn:Bregman.Wasserstein.divergence} is solved by the deterministic coupling $\pi = (\mathrm{Id}, T)_{\#} \mu_0$ induced by the transport map $T = (\iota^*)^{-1} \circ Df \circ \iota : \M \rightarrow \M$. In particular, we have
\begin{equation} \label{eqn:BW.value}
\begin{split}
\mathscr{B}(\mu_0, \mu_1) &=  \int_{\M} {\bf B}(p, T(p)) \dd \mu_0(p) \\
  &= \frac{1}{2} \mathscr{W}_2^2(\mu_0^{\X}, \mu_1^{\Y}) + \int_{\X} \widetilde{\Omega} \dd \mu_0^{\X} + \int_{\Y} \widetilde{\Omega^*} \dd \mu_1^{\Y},
  \end{split}
\end{equation}
\end{proposition}
\begin{proof}
Consider \eqref{eqn:BW.Euclidean.computation} which is valid under condition (ii). Taking infimum over $\pi \in \Pi(\mu, \nu)$, we have
\begin{equation} \label{eqn:BW.Euclidean.computation2}
\begin{split}
\mathscr{B}(\mu_0, \mu_1) &= \inf_{\pi^{\X, \Y} \in \Pi(\mu_0^{\X}, \mu_1^{\Y})} \frac{1}{2} \int_{\R^d \times \R^d} |x - y'|^2 \dd \pi^{\X, \Y}(x, y') + C(\mu_0,\mu_1),%
\end{split}
\end{equation}
where the term $C(\mu_0,\mu_1):= \int_{\X} \widetilde{\Omega} \dd \mu_0^{\X} + \int_{\Y} \widetilde{\Omega^*} \dd \mu_1^{\Y}$ is constant for given $\mu_0$ and $\mu_1$. By a version of Brenier's theorem (see e.g., \cite[Theorem 1.1]{mccann2011five}), the first term is minimized by $\pi^{\X, \Y} = (\mathrm{Id} \times Df)_{\#} \mu_0^{\X}$, where $D f$ is the $\mu_0^{\X}$-a.e.~unique convex gradient which pushforwards $\mu_0^{\X}$ to $\mu_1^{\Y}$. We obtain \eqref{eqn:BW.value} by plugging this coupling into \eqref{eqn:BW.Euclidean.computation2}.
\end{proof}

Proposition \ref{prop:solving.BW.transport} is a direct consequence of Brenier's theorem but has significant implications in applications; it states that computation of the Bregman-Wasserstein divergence between $\mu_0$ and $\mu_1$ is equivalent to that of the Euclidean $2$-Wasserstein distance $\mathscr{W}_2(\mu_0^{\X}, \mu_1^{\Y})$ between $\mu_0^{\X}$ and $\mu_1^{\Y}$, up to transformations via the mirror map and its inverse. For example, the  $\mathscr{B}_{\Omega}(\mu_0^{\X}, \mu_1^{\X})$-optimal transport map which pushforwards $\mu_0^{\X}$ to $\mu_1^{\X}$ under the primal reprsentation is
\begin{equation} \label{eqn:primal.OT.map}
T^{\X} =  D\Omega^* \circ Df : \X \rightarrow \X,
\end{equation}
where $Df$ is the Brenier map given by \eqref{eqn:BW.Brenier}. In particular, \eqref{eqn:BW.value} is the distributional analogue of \eqref{eqn:Bregman.as.Euclidean.shifted2}. Similarly, under the dual representation we have (assuming $\mu_1^{\Y}$ is absolutely continuous)
\begin{equation}
\mathscr{B}(\mu_0, \mu_1) = \mathscr{B}_{\Omega^*}(\mu_1^{\Y}, \mu_0^{\Y}) = \int_{\Y} {\bf B}_{\Omega^*}(y, T^{\Y}(y)) \dd \mu_1^{\Y}(y),
\end{equation}
where
\begin{equation} \label{eqn:dual.OT.map}
T^{\Y} =  D\Omega \circ Dh : \Y \rightarrow \Y
\end{equation}
and $Dh = (Df)^{-1}$ is the Brenier map which pushforwards $\mu_1^{\Y}$ to $\mu_0^{\X}$. In Section \ref{sec:neural.OT}, we estimate these transport maps using state-of-the-art OT solvers based on neural networks.

\begin{example}[Univariate case] \label{def:univariate.OT.map}
Suppose $d = 1$, so that both $\X$ and $\Y$ are open intervals of $\R$. Then the mirror map $D\Omega: \X \rightarrow \Y$ is an increasing diffeomorphism. The transport map $T^{\X}$ in \eqref{eqn:primal.OT.map} is non-decreasing since both $Df$ and $D\Omega^*$ are non-decreasing. In particular, the optimal coupling between $\mu_0^{\X}$ and $\mu_1^{\X}$ is the so-called {\it comonotonic coupling} \cite[Chapter 2]{S15}. In this case, the optimal coupling and transport map (which is $\mu_0^{\X}$-unique when $\mu_0^{\X}$ is absolutely continuous) does not depend on the choice of the Bregman generator $\Omega$ on $\X$, but the value of $\mathscr{B}_{\Omega}(\mu_0^{\X}, \mu_1^{\X})$ still does.
\end{example}

\subsubsection{Comparison with a Riemannian $2$-Wasserstein distance}
On $\M$, the Bregman divergence induces a Hessian Riemannian metric $g$ defined by \eqref{eqn:Bregman.metric}. That is, $(\M, g)$ is a Riemannian manifold. Let $\mathsf{d}_g(p, q)$ be the induced Riemannian distance and consider the corresponding {\it Riemannian $2$-Wasserstein distance} $\mathscr{W}_{2, g}$, defined by
\begin{equation} \label{eqn:Riemannian.2.Wasserstein}
\mathscr{W}_{2, g}^2(\mu_0, \mu_1) := \inf_{\pi \in \Pi(\mu_0, \mu_1)} \int_{\M \times \M} \mathsf{d}_g^2(p, q) \dd \pi(p, q).
\end{equation}
Since the Riemannian distance $\mathsf{d}_g(p, q)$ does not admit closed-form expressions except in simple cases, the Riemannian $2$-Wasserstein distance is generally not easy to implement. On the other hand, recall from \eqref{eqn:M.metric} that ${\bf B}$ and $\frac{1}{2} \mathsf{d}_g^2$ are equal up to second order along the diagonal. This suggests that $\mathscr{B}(\mu_0, \mu_1) \approx \frac{1}{2}\mathscr{W}_{2, g}^2(\mu_0, \mu_1)$ when $\mu_0 \approx \mu_1$; see Proposition \ref{prop:BW.metric} for a precise statement. Since the Bregman-Wasserstein divergence is straightforward to compute (in the sense that the computational cost is similar to that of a Euclidean quadratic transport) once $D\Omega$ and $D\Omega^*$ are known, it provides a tractable local approximation of $\mathscr{W}_{2, g}$. In Section \ref{sec:Wasserstein.gradient.flow}, we use this idea to propose the {\it Bregman-Wasserstein JKO scheme} to discretize Riemannian Wasserstein gradient flows.

\subsection{Probabilistic interpretation in terms of exponential family} \label{sec:exp.family}

It is well known that Bregman divergences are closely related to exponential families of probability distributions and the KL-divergence \cite{AN00, BMDG05}. Using this relation, we give a direct probabilistic interpretation of the Bregman-Wasserstein divergence when the underlying Bregman divergence originates from an exponential family.

Let $\mathbb{P}_0$ be a probability measure on $\mathbb{R}^d$ and let $\Omega: \mathbb{R}^d \rightarrow \mathbb{R} \cup \{+\infty\}$ be its cumulant generating function or equivalently the logarithm of the Laplace transform:
\[
\Omega(\theta) = \log \int_{\mathbb{R}^d} e^{\theta \cdot y} \dd \mathbb{P}_0(y).
\]
Let $\X = \mathrm{dom}(\Omega) = \{\theta \in \mathbb{R}^d : \Omega(\theta) < \infty\}$ which can be shown to be convex. Assuming $\X$ is open, $\Omega$ is a regular Bregman generator on $\X$; in fact $(\X, \Omega)$ is of Legendre type \cite[Theorem 9.3]{B14}. Here we use $\theta$,  rather than $x$, to denote the primal variable since $\theta$ is interpreted as the natural parameter of the exponential family. For $\theta \in \X$, define $\mathbb{P}_{\theta} \in \mathcal{P}(\mathbb{R}^d)$ by
\begin{equation} \label{eqn:exp.family}
p_{\theta}(y) = \frac{\dd \mathbb{P}_{\theta}}{\dd \mathbb{P}_0}(y) = e^{\theta \cdot y - \Omega(\theta)}.
\end{equation}
Then $\M = \{ \mathbb{P}_{\theta} : \theta \in \X \} \subset \mathcal{P}(\mathbb{R}^d)$ is the $d$-dimensional {\it exponential family} induced by $\mathbb{P}_0$. The resulting Bregman manifold has a rich theory and its numerous statistical applications have been well studied \cite{A16, AN00}. We mention a few required properties. First, we have
\begin{equation} \label{eqn:Bregman.as.KL}
{\bf B}( \mathbb{P}_{\theta}, \mathbb{P}_{\theta'}) = {\bf B}_{\Omega}(\theta, \theta') = {\bf H}(\mathbb{P}_{\theta'} || \mathbb{P}_{\theta} ),
\end{equation}
which is the KL-divergence (see e.g.~\cite[(2.16)]{A16}), and the induced Riemannian metric $g$ is the Fisher-Rao metric on $\M$. Thus, analogous to Example \ref{eg:simplex.cgf}, we may interpret the Bregman-Wasserstein divergence as a transport KL-divergence. Moreover, the dual variable $\eta = D\Omega(\theta)$ is the {\it expectation parameter} given by $\eta = \int_{\mathbb{R}^d} y \dd \mathbb{P}_{\theta}(y)$, and the dual function $\Omega^*$ is the negative Shannon entropy:
\begin{equation} \label{eqn:exp.family.dual.function}
\Omega^*(\eta) = - \int_{\mathbb{R}^d} p_{\theta} \log p_{\theta} \dd \mathbb{P}_0, \quad \eta = D\Omega(\theta) \in \Y.
\end{equation}

Now further suppose that $\mathbb{P}_0$ is either discrete or is absolutely continuous with respect to the Lebesgue measure. Then, by  \cite[Theorem 4]{BMDG05}, the density can be expressed, on the common support $\mathcal{S} \subset \mathbb{R}^d$ of elements $\mathbb{P}_{\theta}$ of $\M$ and in terms of the Bregman divergence ${\bf B}_{\Omega^*}$, by
\begin{equation} \label{eqn:exp.family.Bregman}
p_{\theta}(y) = e^{-{\bf B}_{\Omega^*}(y, \eta) + \Omega^*(y)}, \quad y \in \mathcal{S}.
\end{equation}
The main technicality in this result is to show that $\mathcal{S} \subset \mathrm{dom}(\Omega^*)$, where $\Omega^*$ is the convex conjugate of $\Omega$, so that the Bregman divergence ${\bf B}_{\Omega^*}(y, \eta)$ is defined. For the purpose of this discussion, we assume $\mathcal{S} \subset \Y = D\Omega(\X)$.

We are ready to discuss the probabilistic interpretation of the Bregman-Wasserstein divergence. Suppose $\mu_0 \in \mathcal{P}(\M)$ is a discrete distribution of the form $\mu_0^{\Y} = \frac{1}{N} \sum_{i = 1}^N \delta_{\eta_i}$. Now, suppose that for each parameter $\theta_i = D\Omega^*(\eta_i) \in \X$, we sample one value from the distribution $\mathbb{P}_{\theta_i}$. Equivalently, we may consider a system of $N$ independent particles where each particle $Z^i_t$ is a Markov chain $(Z_0, Z_1)$ with transition kernel $Q(y, \cdot) = \mathbb{P}_{D \Omega^*(y)}(\cdot)$. Let $\mu_1^{\Y} = \frac{1}{N} \sum_{j = 1}^N \delta_{y_j}$ be the observed empirical distribution of the particles $\{Z^1_1, \ldots, Z^N_1\}$ at time $1$. Note that we only observe the initial and final empirical distributions of the particles, but do {\it not} know the identities of the particles.

Consider the problem of finding an optimal matching $i \mapsto j  = \sigma(i)$, between the initial and final positions, so as to maximize the log likelihood $\mathrm{loglik}(\sigma)$. This is similar in spirit to the {\it Schr\"{o}dinger bridge problem} \cite{leonard2012schrodinger}, but in our setting there is no additional noise parameter and we maximize directly the likelihood rather than minimizing a relative entropy. For a given permutation $\sigma$, we have%
\begin{equation*}%
\begin{split}
\mathrm{loglik}(\sigma) &:= \sum_{i = 1}^N \log p_{\theta_i}(y_{\sigma(i)}) = \sum_{i = 1}^N \left(-{\bf B}_{\Omega^*} (y_{\sigma(i)}, \eta_i) + \Omega^*(y_{\sigma(i)}) \right).
\end{split}
\end{equation*}
Rearranging, we get
\begin{equation} \label{eqn:exp.family.BW}
\int_{\Y \times \Y} {\bf B}_{\Omega^*}(y, \eta) \dd \pi_{\sigma}^{\Y}(y, \eta) = -\frac{1}{N} \mathrm{loglik}(\sigma) + \int_{\Y} \Omega^* \dd \mu_1^{\Y},
\end{equation}
where $\pi_{\sigma}^{\Y}$ is the coupling on $\Y \times \Y$ induced by $\sigma$.
Thus, the Bregman-Wasserstein divergence $\mathscr{B}(\mu_0, \mu_1)$ is, up to additive and multiplicative constants, the negative (maximized) log likelihood of the optimal matching. One may easily extend this interpretation to arbitrary $\mu_0, \mu_1 \in \mathcal{P}(\M)$. We believe that the relation with exponential family will yield new applications of the Bregman-Wasserstein divergence.

\begin{example} \label{eg:optimal.matching}
Let $\mathbb{P}_0$ be the uniform distribution on the open hypercube $\Y = (-1, 1)^d$, and consider the corresponding exponential family $\M = \{\mathbb{P}_{\theta}\}_{\theta \in \mathbb{R}^d}$. The corresponding Bregman generator is
\[
\Omega(\theta) = \sum_{i = 1}^d \log \frac{\sinh \theta^i}{\theta^i}, \quad \theta \in \X = \mathbb{R}^d,
\]
and the mirror map $D\Omega: \X \rightarrow \Y$ is given by
\[
\eta = D\Omega(\theta) = \left( \coth(\theta^i) - \frac{1}{\theta^i} \right)_{1 \leq i \leq d}.
\]
The inverse mirror map, which acts component-wise, is not explicit but can be implemented efficiently. In Figure \ref{fig:optimal_matching} we illustrate the solution of the optimal matching problem \eqref{eqn:exp.family.BW}, where $d = 2$ and the empirical distributions $\mu_0^{\Y}, \mu_1^{\Y}$ are simulated.\footnote{The optimal matching was computed in \texttt{R} using the \texttt{transport} package \cite{Rtransport}.}

\begin{figure}[t!]
	\centering
	\hspace{-0.55cm}
	\includegraphics[scale = 0.5]{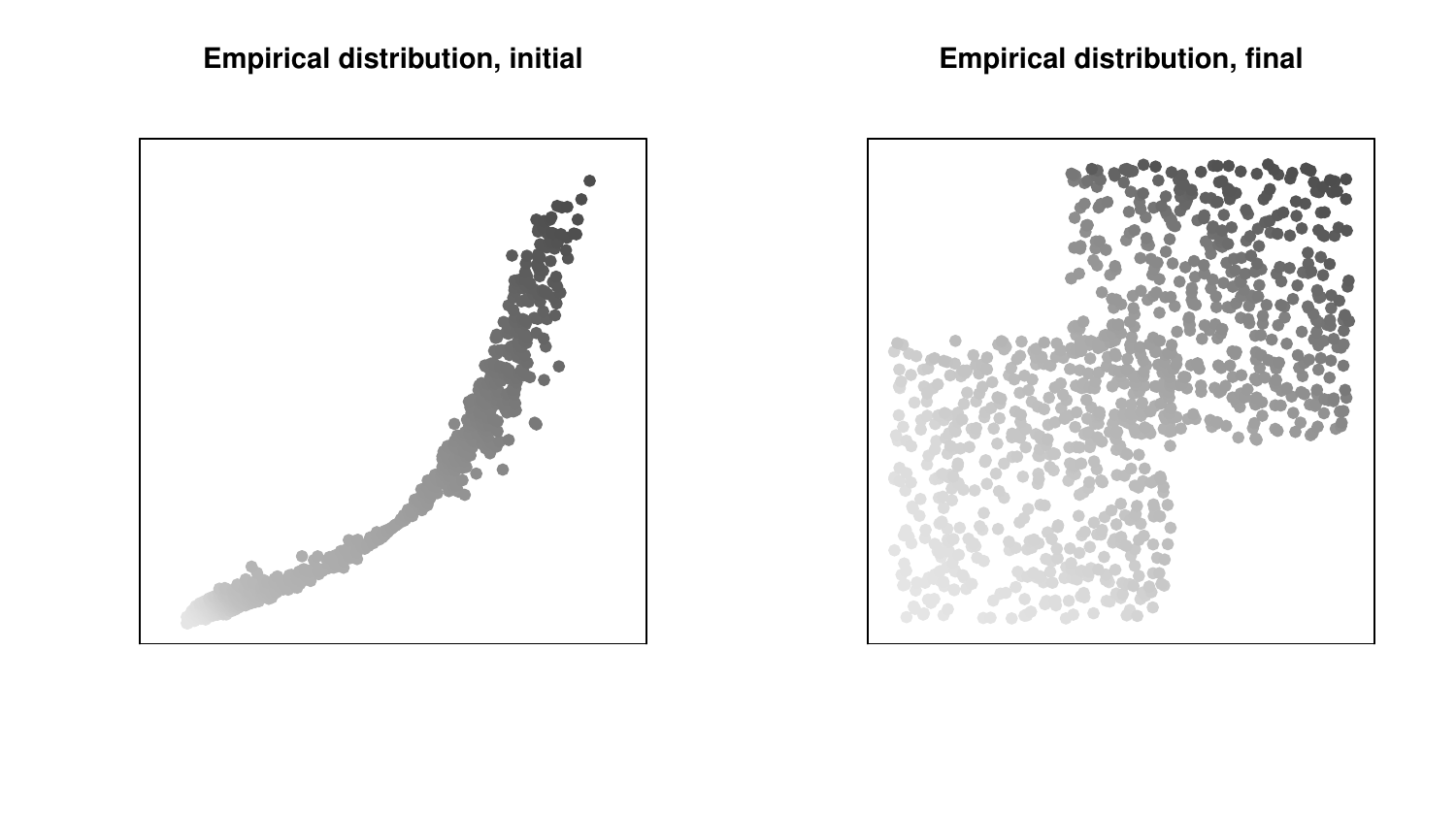}
    \vspace{-1.25cm}
	\caption{\footnotesize Optimal matching in the context of Example \ref{eg:optimal.matching} where $\Y = (-1, 1)^d$. Left: Empirical distribution $\mu_0^{\Y}$ at time $0$. Right: Empirical distribution $\mu_1^{\Y}$ at time $1$. Here $d = 2$ and each distribution has $N = 1000$ data points. The optimal matching is indicated by the coloring of the points.} \label{fig:optimal_matching}
\end{figure}
\end{example}

\begin{example} \label{eg:BW.exp.family}
In the context of this subsection, for any $\mu \in \mathcal{P}(M)$, define $\mathbb{Q}_{\mu} \in \mathcal{P}(\mathcal{S})$ (recall that $\mathcal{S}$ is the support of the exponential family) by
\[
\mathbb{Q}_{\mu}(\cdot) = \int_{\X} \mathbb{P}_{\theta}(\cdot) \dd \mu^{\X}(\theta).
\]
Thus, each $\mathbb{Q}_{\mu}$ is an {\it exponential mixture distribution} \cite{lindsay1983geometry}. Given two such mixture distributions, say $\mathbb{Q}_{\mu_0}$ and $\mathbb{Q}_{\mu_1}$, we may define a divergence between them (or, rather, between the weight distributions $\mu_0$ and $\mu_1$) as the Bregman-Wasserstein divergence between the mixture weights $\mu_0$ and $\mu_1$:
\begin{equation}\label{eqn:BW.mixture.exp.family}
\mathscr{D}(\mathbb{Q}_{\mu_0}, \mathbb{Q}_{\mu_1}) := \mathscr{B}(\mu_0, \mu_1) = \inf_{\pi^{\X} \in \Pi(\mu_0^{\X}, \mu_1^{\X})} \int_{\X \times \X} {\bf H}(\mathbb{P}_{\theta'}|| \mathbb{P}_{\theta}) \dd \pi(\theta, \theta'),
\end{equation}
where the last equality follows from \eqref{eqn:Bregman.as.KL}.
Note that this is different from the KL-divergence ${\bf H}(\mathbb{Q}_{\mu_1}||\mathbb{Q}_{\mu_0})$ since $\mathscr{D}(\mathbb{Q}_{\mu_0}, \mathbb{Q}_{\mu_1})$ takes into account the difference of the mixture weights. In Example \ref{eg:GM.barycenters} we apply this divergence in a Bayesian setting.
\end{example}

\section{Primal and dual displacement interpolations} \label{sec:interpolations}
In this section, we define the primal and dual displacement interpolations between a pair of probability distributions and study their properties.

\subsection{Motivation and definition}
When considering the $2$-Wasserstein distance, a key role is played by McCann's {\it displacement interpolation} \cite{M97} which is defined as follows. Let $\mu_0, \mu_1 \in \mathcal{P}_{2}(\mathbb{R}^d)$ (probability measures on $\R^d$ with finite second moment) be, say, absolutely continuous, and let $Df$ be the convex gradient satisfying $(Df)_{\#}\mu_0 = \mu_1$. Then McCann's displacement interpolation is defined by
\begin{equation} \label{eqn:McCann.interpolation}
\mu_t = ((1 - t) \mathrm{Id} + t Df)_{\#} \mu_0, \quad 0 \leq t \leq 1.
\end{equation}
It is well known that displacement interpolations are geodesics in the {\it Wasserstein space} $(\mathcal{P}_2(\mathbb{R}^d), \mathscr{W}_2)$, that is $\mathscr{W}_2(\mu_s, \mu_t) = |s - t| \mathscr{W}_2(\mu_0, \mu_1)$, for any $s, t \in [0, 1]$. Under the displacement interpolation \eqref{eqn:McCann.interpolation}, each ``particle'' travels along a Euclidean geodesic (i.e.~constant-velocity straight line) on $\mathbb{R}^d$. See \cite{AGS08, Gigli12} for the corresponding theory on a Riemannian manifold.

When we consider a Bregman manifold $\M$ it is the Bregman geometry, not the Riemannian, that proves most useful and tractable in applications. Thus, we wish to define generalized displacement interpolations on $\mathcal{P}(\M)$ which are compatible with the primal and dual geodesics. Here we state the definition of primal and dual displacement interpolations, a graphical illustration of which is given in Figure \ref{fig:interpolations}.

 \begin{definition}[Primal and dual displacement interpolations] \label{def:primal.dual.displacement.interpolation}
 Consider a curve $(\mu_t)_{0 \leq t \leq 1} \subset \cP(\M)$.
 \begin{enumerate}
     \item[(i)] We call $(\mu_t)$ a primal displacement interpolation if there exists a convex function $h$ on $\Y$ such that $Dh : \Y \rightarrow \X$ and
     \begin{equation}  \label{eqn:primal.generalized.geodesic}
     \mu_t^{\X} = ((1 - t) D \Omega^* + t D h)_{\#}  \mu_0^{\Y}.
     \end{equation}
     \item[(ii)] We call $(\mu_t)$ a dual displacement interpolation if there exists a convex function $f$ on $\X$ such that $D f : \X \rightarrow \Y$ and
     \begin{equation} \label{eqn:dual.generalized.geodesic}
     \mu_t^{\Y} =  ((1 - t) D \Omega + t Df))_{\#} \mu_0^{\X}.
     \end{equation}
 \end{enumerate}
 \end{definition}

 Definition \ref{def:primal.dual.displacement.interpolation} is closely related to the {\it generalized geodesic} introduced in \cite[Definition 9.2.2]{AGS08} (also see \cite[Definition 7.31]{S15}). A justification using the theory of {\it $c$-convexity} is given in Section \ref{sec:c.convex}. In Section \ref{sec:parallel.transport} we will see that these primal and dual displacement interpolations turn out to be parallel transports with respect to the primal and dual connections on the space of probability measures.

	\begin{figure}[t!]
    \scalebox{0.75}{
\begin{tikzpicture}
    \tikzset{
        dist/.style={ellipse, draw, minimum width=1.5cm, minimum height=0.8cm, thick, fill opacity=0.5},
        irreg/.style={
            ellipse,
            draw,
            minimum width=1.5cm,
            minimum height=0.8cm,
            thick,
            fill opacity=0.5,
        },
        interp/.style={ellipse, draw, dashed, minimum width=1.5cm, minimum height=0.8cm, thin, fill opacity=0.5},
        irreg_curve_interp/.style={
            ellipse,
            draw,
            dashed,
            minimum width=1.5cm,
            minimum height=0.8cm,
            thin,
            fill opacity=0.5,
        }
    }

    \node[dist, fill=blue!30, rotate=20] at (-1.5,-3.5) (A1) {};
    \node[dist, fill=red!30, rotate=-15] at (-2.5,0.5) (A2) {};

    \node at (-1.4,0.5) {$\textcolor{red!90}{\mu_0^{\mathcal{X}}}$};
    \node at (-3.075,-1.5) {$\textcolor{blue!50!red!50}{\mu_t^{\mathcal{X}}}$};
    \node at (-2.5,-3.5) {$\textcolor{blue!90}{\mu_1^{\mathcal{X}}}$};
    \draw[gray] plot[mark=*,dashed] coordinates {(-1.5,-3.5) (-2,-1.5) (-2.5,0.5)};
    \node at (-2.75,0.65) {\footnotesize$x_0$};
    \node at (-2.3,-1.5) {\footnotesize$x_t$};
    \node at (-1.75,-3.65) {\footnotesize$x_1$};

    \foreach \i in {1,...,3} {
        \pgfmathsetmacro{\t}{\i/4}
        \pgfmathsetmacro{\x}{(-1.5)*(1-\t) + (-2.5)*\t}
        \pgfmathsetmacro{\y}{(-3.5)*(1-\t) + (0.5)*\t}
        \pgfmathsetmacro{\rot}{20*(1-\t) + (-15)*\t}
        \pgfmathsetmacro{\blue}{100*(1-\t)} 
        \pgfmathsetmacro{\red}{100*\t} 
        \node[interp, fill=blue!\blue!red!\red, rotate=\rot] at (\x,\y) {};
    }

    \draw[thick, gray!70, dashed] (-2,-1.5) ellipse (2.25cm and 3.5cm);

    \node[above] at (-2,1.25) {$\mathcal{X}$};

    \node[irreg, fill=blue!50, rotate=25] at (3.0,-3.75) (B1) {};
    \node[irreg, fill=red!50, rotate=-20] at (2.5,0.5) (B2) {};

    \node at (3.6,0.5) {$\textcolor{red!90}{\mu_0^{\mathcal{Y}}}$};
    \node at (3.45,-1.75) {$\textcolor{blue!50!red!50}{\mu_t^{\mathcal{Y}}}$};
    \node at (2.5,-3.2) {$\textcolor{blue!90}{\mu_1^{\mathcal{Y}}}$};
    \draw[gray] plot[mark=*,smooth,tension=1] coordinates {(3.0,-3.75) (4.5,-1.75) (2.5,0.5)};
    \node at (2.25,0.65) {\footnotesize$y_0$};
    \node at (4.25,-1.85) {\footnotesize$y_t$};
    \node at (2.75,-3.85) {\footnotesize$y_1$};

    \foreach \i in {1,...,3} {
        \pgfmathsetmacro{\t}{\i/4}
        \pgfmathsetmacro{\x}{(3.5)*(1-\t)*(1-\t) + (6)*2*\t*(1-\t) + (2.5)*\t*\t}
        \pgfmathsetmacro{\y}{(-3.5)*(1-\t)*(1-\t) + (-2)*2*\t*(1-\t) + (0.5)*\t*\t}
        \pgfmathsetmacro{\rot}{25*(1-\t) + (-20)*\t}
        \pgfmathsetmacro{\blue}{100*(1-\t)} 
        \pgfmathsetmacro{\red}{100*\t} 
        \node[irreg_curve_interp, fill=blue!\blue!red!\red, rotate=\rot] at (\x,\y) {};
    }

    \draw[thick, gray!70, dashed] (3.4,-1.5) ellipse (2.25cm and 3.5cm);

    \node[above] at (3.25,1.25) {$\mathcal{Y}$};

    \draw[->, dashed, ultra thick, -Stealth]
        (-1.25,1.2) to[out=45, in=135] node[midway, below, sloped] {$D\Omega$} (2.0,1.2);

    \draw[->, dashed, ultra thick, -Stealth]
        (2.0,-4.5) to[out=-135, in=-45] node[midway, above, sloped] {$D\Omega^*$} (-1.25,-4.5);

    \draw[->, thick, -Stealth]
        (1.9,0.3) to node[midway, below, sloped] {$Dh$} (-0.75,-3.2);

    \draw[->, thick, -Stealth]
        (1.9,0.3) to node[midway, above, sloped] {$(1-t)D\Omega^* + t Dh$} (-1.25,-1.25);
  \end{tikzpicture}
}
        \caption{\footnotesize Illustration of primal displacement interpolation. We first find a convex gradient $Dh$ which pushforwards $\mu_0^{\Y}$ to $\mu_1^{\X}$. Under the primal representation, the primal displacement interpolation $\mu_t^{\X}$ is defined by the pushforward of $\mu_0^{\Y}$ under the convex combination $Dh_t := (1 - t) D\Omega^* + t Dh$. This guarantees that each particle moves along a primal geodesic $(x_t)$.} \label{fig:interpolations}
	\end{figure}

Letting $t = 1$ in \eqref{eqn:dual.generalized.geodesic}, we have that $Df$ is a convex gradient with $(Df)_{\#} \mu_0^{\X} = \mu_1^{\Y}$. Thus, $Df$ is the Brenier map which appears in the optimal transport map $T^{\X}$ (with respect to $\mathscr{B}_{\Omega}(\mu_0^{\X}, \mu_1^{\X})$) given by \eqref{eqn:primal.OT.map}. Analogously, $Dh$ in \eqref{eqn:primal.generalized.geodesic} is the Brenier map in the optimal transport map $T^{\Y}$ (with respect to $\mathscr{B}_{\Omega^*}(\mu_0^{\Y}, \mu_1^{\Y})$) as in \eqref{eqn:dual.OT.map}, after swapping $\mu_0$ and $\mu_1$.

\begin{example}
Let $(\gamma_t)_{0 \leq t \leq 1}$ be a curve in $\M$ and, for each $t$, let $\mu_t = \delta_{\gamma_t}$ be the point mass at $\gamma_t$. Then $(\mu_t)$ is a primal (resp.~dual) displacement interpolation if and only if the curve $(\gamma_t)$ is a primal (resp.~dual) geodesic. More generally, each ``particle'' in the primal (resp.~dual) displacement interpolation travels along a primal (resp.~dual) geodesic.
\end{example}

\begin{figure}[t!]
	\centering
	\includegraphics[scale = 0.65]{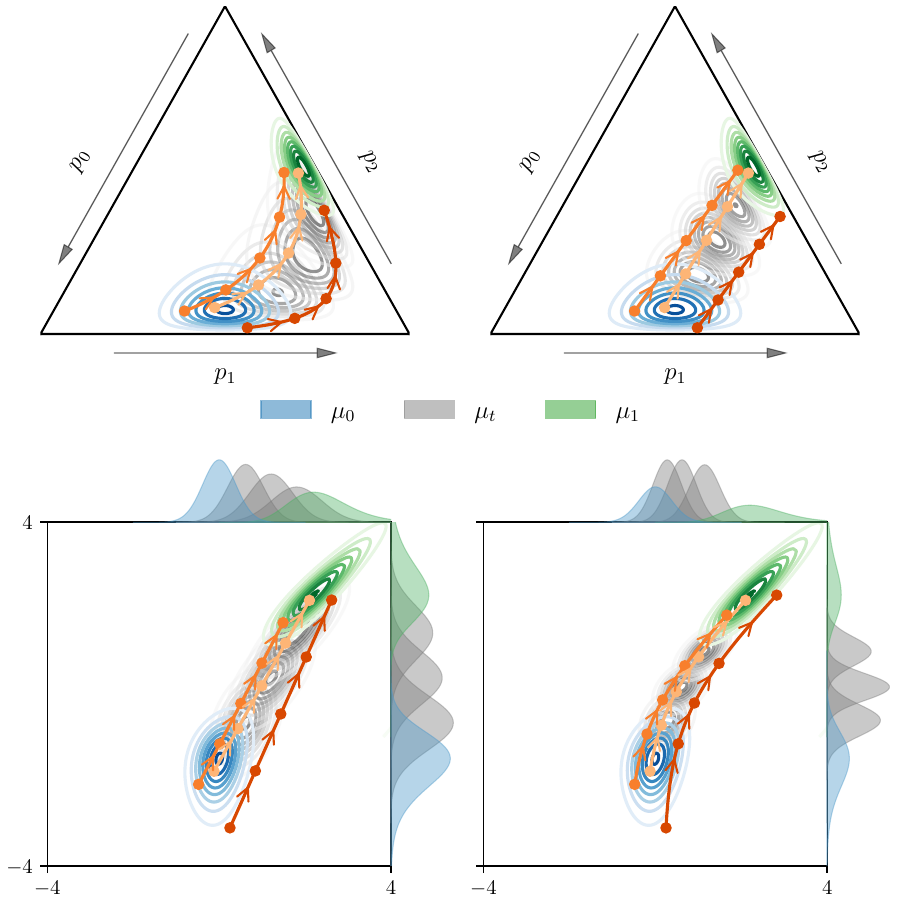}

	\caption{\footnotesize Illustration of primal and dual displacement interpolations in the context of Example \ref{eg:generalized.geodesics}. The source distribution $\mu_0$ is shown in blue and the target $\mu_1$ is shown in green. The first row shows the dual domain (simplex) and the second row shows the primal domain ($\mathbb{R}^2$). The first column illustrates the primal displacement interpolation, and the second column the dual displacement interpolation. We also plot the primal and dual geodesics traced by three particles (shades of orange).}
	\label{fig:geodesics}
\end{figure}

\begin{example} \label{eg:generalized.geodesics}
In Figure \ref{fig:geodesics} we illustrate graphically the primal and dual displacement interpolations using the Bregman generator in Example \ref{eg:simplex.cgf}, where $d = 2$. The primal domain $\X$ is $\mathbb{R}^2$ and the dual domain $\Y$ can be identified with the open unit simplex. The Bregman divergence ${\bf B}$ is the (reverse) KL-divergence (see \eqref{eqn:simplex.KL}). Here $\mu_0$ and $\mu_1$ are simulated point clouds and all distributions are visualized using contours of density estimates. We visualize both the primal and dual displacement interpolations under both the primal and dual representations. Under the primal displacement interpolation (left column), each particle travels along a straight line in $\X$. These primal geodesics appear to be curved when plotted under the dual coordinate system. Using the notation in Example \ref{eg:simplex.cgf}, a primal geodesic (also called an $e$-geodesic or exponential interpolation in this context) has the form
\[
y^i_t = \frac{1}{Z_t} (y^i_0)^{1 - t} (y^i_1)^t, \quad 0 \leq i \leq d,
\]
where $Z_t$ is the normalizing constant such that $\sum_{i = 0}^d y^i_t = 1$. Similarly, under the dual displacement interpolation (right column) particles move along straight lines in $\Y$. Here, the optimal transports between point clouds are implemented in terms of optimal matching. In Section \ref{sec:neural.OT} we adopt another approach using neural networks.
\end{example}

Displacement interpolations were first introduced by McCann \cite{M94,M97} to study optimization on the space of probability measures via a new notion of convexity. With primal and dual displacement interpolations, we define two notions of displacement convexity compatible with the Bregman geometry.

\begin{definition}[Primal and dual displacement convexities]
A functional $\mathcal{F}: \mathcal{P}(\M) \rightarrow \mathbb{R} \cup \{+\infty\}$ is said to be primal displacement convex if $\mathcal{F}(\mu_t)$ is convex in $t$ for any primal displacement interpolation $(\mu_t)_{0 \leq t \leq 1}$. Similarly, $\mathcal{F}$ is said to be dual displacement convex if $\mathcal{F}(\mu_t)$ is convex in $t$ for any dual displacement interpolation $(\mu_t)_{0 \leq t \leq 1}$.
\end{definition}

\begin{example}[Conditional entropy]
Consider an exponential family $\M = \{ \mathbb{P}_{\theta}\}_{\theta \in \X}$ in the context of Section \ref{sec:exp.family}, where $\Omega = \Omega(\theta)$ is the cumulant generating function. Let $\mu \in \mathcal{P}(\M)$ and let $\tilde{\theta}$ be a random vector distributed as $\mu^{\X}$. Given $\tilde{\theta} = \theta$, let $Y \sim \mathbb{P}_{\theta}$. Consider the negative {\it conditional (differential) entropy} of $Y$ given $\tilde{\theta}$ defined by
\[
\mathcal{H}(\mu) = -{\bf H}(Y|\tilde{\theta}) :=  \int_{\X} \left( \int p_{\theta} \log p_{\theta} \dd \mathbb{P}_{\theta} \right) \dd \mu^{\X}(\theta).
\]
By \eqref{eqn:exp.family.dual.function}, we may write $\mathcal{H}(\mu) = \int_{\X} \Omega^*(\eta) \dd \mu^{\Y}(\eta)$, where $\eta = D\Omega(\theta)$ is the expectation parameter. Since $\Omega^*$ is convex on $\Y$, it is straightforward to extend the argument in \cite[Theoreme 5.15]{V03} to see that $\mathcal{H}$ is dual displacement convex.
\end{example}

\subsection{Generalized Pythagorean inequality}  \label{sec:Pyth}
The Bregman divergence ${\bf B}$ on $\M$ satisfies a generalized Pythagorean theorem (Theorem \ref{thm:Bregman.Pyth}). We show that it implies a generalized Pythagorean {\it inequality} for the Bregman-Wasserstein divergence $\mathscr{B}$ on $\mathcal{P}(\M)$. The statement requires a Riemannian metric between velocities of curves on $\mathcal{P}(\M)$. In Section \ref{sec:prelim-otto-lott} we will see that the special case \eqref{eqn:Otto.inner.product.special.case} used is consistent with {\it  Otto's metric} in the general theory (also see Lemma \ref{lem:displacement.interpolation.fields}).

\begin{theorem}[Generalized Pythagorean inequality] \label{thm:BW.Pyth}
Let $\rho \in \mathcal{P}(\M)$ and smooth convex gradients $Dh: \Y \rightarrow \X$ and $Df : \X \rightarrow \Y$, let $(\mu_t)_{0 \leq t \leq 1}$ be the primal displacement interpolation started at $\mu_0 = \rho$ and induced by $Dh$, and let $(\nu_t)_{0 \leq t \leq 1}$ be the dual displacement interpolation started at $\nu_0 = \rho$ and induced by $Df$, i.e.,
\[
\mu_t^{\X} = ((1 - t) D\Omega^* + tDh)_{\#} \mu_0^{\Y}, \quad \nu_t^{\Y} = ((1 - t) D \Omega + t Df)_{\#} \nu_0^{\X}.
\]
Assume that the pairs $(\rho, \mu_1)$ and $(\rho, \nu_1)$ satisfy the conditions in Proposition \ref{prop:solving.BW.transport}. For $0 \leq t < 1$, it holds that
\begin{equation} \label{eqn:BW.Pyth}
\mathscr{B}(\rho, \nu_t) + \mathscr{B}(\mu_t, \rho
) - \mathscr{B}(\mu_t, \nu_t) \geq t^2 \mathfrak{g}( \dot{\mu}_0 , \dot{\nu}_0),
\end{equation}
where
\begin{equation} \label{eqn:Otto.inner.product.special.case}
\mathfrak{g}( \dot{\mu}_0 , \dot{\nu}_0) := \int_{\M} g(\grad (f - \Omega) \circ \iota  ,\grad (h - \Omega^*)\circ \iota^*) \dd \rho(p)
\end{equation}
is Otto's Riemannian inner product between the primal and dual geodesics at time zero. Furthermore, when $\dim \M = 1$, then equality holds in~\eqref{eqn:BW.Pyth}.
\end{theorem}
\begin{proof}
For $0 \leq t < 1$, the convex function $f_t = (1 - t) \Omega + t f$ satisfies conditions (i) and (ii) of Definition \ref{def:Bregman.generator}, so $Df_t$ is a diffeomorphism from $X$ onto its image. Moreover, the mapping $S_t = (\iota^*)^{-1} \circ Df_t \circ \iota$ satisfies the conditions of Proposition \ref{prop:solving.BW.transport} and is the optimal transport map which attains $\mathscr{B}(\rho, \nu_t)$. %
Similarly, for $h_t = (1 - t) \Omega^* + t h$ the mapping $T_t = \iota^{-1} \circ Dh_t \circ \iota^*$ is invertible and is optimal for $\mathscr{B}(\mu_t, \rho)$. It follows that
\begin{equation*}
\mathscr{B}(\mu_t, \rho) = \int_{\M} {\bf B}(T_t(p), p) \dd \rho(p), \quad
\mathscr{B}(\rho, \nu_t) = \int_{\M} {\bf B}(p, S_t(p)) \dd \rho(p).
\end{equation*}
Since $T_t$ is invertible, $U_t = S_t \circ T_t^{-1}$ pushforwards $\mu_t$ to $\nu_t$ but is generally suboptimal as a
transport map, i.e.,
\begin{equation} \label{eqn:BW.triangle.proof}
\mathscr{B}(\mu_t, \nu_t) \leq \int_{\M} {\bf B}(p, U_t(p)) \dd \mu_t(p) = \int_{\M} {\bf B}(T_t(p), S_t(p)) \dd \rho(p),
\end{equation}
where the last inequality follows from the fact that $(\mathrm{Id} \times U_t)_{\#} \mu_t = (T_t \times S_t)_{\#} \rho$.

By Theorem \ref{thm:Bregman.Pyth}, we have
\begin{equation*}
\begin{split}
& \mathscr{B}(\rho, \nu_t) + \mathscr{B}(\mu_t, \rho) - \mathscr{B}(\mu_t, \nu_t) \\
&\geq \int_{\M} \left({\bf  B}(p, S_t(p)) + {\bf B}(T_t(p), p) - {\bf B}(T_t(p), S_t(p))\right) \dd \rho(p) \\
&= t^2 \int_{\M} g\left( \frac{\dd}{\dd t} \Big|_{t = 0} T_t(p), \frac{\dd}{\dd t} \Big|_{t = 0} S_t(p)\right) \dd \rho(p) \\
&= t^2 \int_{\M} g(\grad (f - \Omega) \circ \iota  ,\grad (h - \Omega^*)\circ \iota^*) \dd \rho(p) \\
&= t^2 \mathfrak{g}(\dot{\mu}_0, \dot{\nu}_0),
\end{split}
\end{equation*}
where the second last equality follows from \eqref{eqn:Bregman.Pyth.identity} and the definition of Riemannian gradient.

It remains to show that equality holds in $d = 1$. Indeed, when $d = 1$ (so that both $\X$ and $\Y$ are intervals of $\R$) we have
\[
U_t = S_t \circ T_t^{-1} = (\iota^*)^{-1} \circ (Df_t \circ Dh_t^{-1}) \circ \iota,
\]
where $Df_t \circ Dh_t^{-1} : \X \rightarrow \Y$ is increasing since both $Df$ and $Dh_t$ are increasing. From Example \ref{def:univariate.OT.map}, $U_t$ is an optimal transport map from $\mu_t$ to $\nu_t$. Hence equality holds in \eqref{eqn:BW.triangle.proof} and \eqref{eqn:BW.Pyth} is identically $0$.

\end{proof}

\begin{remark} \label{rmk:Pyth.curvature}
It is well known that the Wasserstein space $(\mathcal{P}_2(\mathbb{R}^d), \mathscr{W}_2)$ for $d \geq 2$, regarded as an infinite dimensional Riemannian manifold, has non-negative sectional curvature even though $\mathbb{R}^d$ (which is a self-dual Bregman manifold) is flat \cite{L08}. More generally, if $(\mathcal{M}, g)$ is a Riemannian manifold with Ricci curvature bounded below by $K$, then the (Riemannian) Wasserstein space $(\mathcal{P}_2(\mathcal{M}), \mathscr{W}_2)$ also has curvature bounded below by $K$ (in the sense of Alexandrov) \cite{V08}. As seen from the proof of Theorem \ref{thm:Bregman.Pyth}, in general (when $d \geq 2$) one cannot expect that equality holds in \eqref{eqn:BW.Pyth}.
\end{remark}

\subsection{Relation with $c$-convex geometry} \label{sec:c.convex}
We give an interpretation of Definition \ref{def:primal.dual.displacement.interpolation} using the {\it $c$-convex geometry} in optimal transport. For convenience, in this discussion we adopt the general notations of optimal transport as in \cite{V08}. Consider the Bregman cost
\[
c(x, y) = \Omega(x) - \Omega(y) - D\Omega(y) \cdot (x - y)
\]
on the primal domain $\X$ (here $x$ and $y$ are independent variables). Let $c_x$ be the (Euclidean) gradient of $c$ with respect to $x$. Then  for each $x$
\[
c_x(x, y) = D\Omega(x) - D\Omega(y)
\]
is injective in $y \in \X$ since $\Omega$ is a regular Bregman generator. In particular, $c$ satisfies the {\it twist condition} (which is not surprising since $c$ is equivalent to the quadratic cost upon a change of variables). Consider the {\it $c$-exponential map} $\cexp$ \cite[Definition 12.29]{V08} defined by the identity
\[
c_x(x, \cexp_x(p)) = -p.
\]
By a direct computation, we see that
\begin{equation} \label{eqn:cexp}
\cexp_x(p) = D\Omega^* ( D \Omega (x) + p),
\end{equation}
which is well-defined when $x \in \X$ and $D \Omega(x) + p \in \Y$. By definition, a {\it $c$-segment} with base $\overline{x}$ \cite[Definition 12.10]{V08} is a curve of the form
\[
y_t = \cexp_{\overline{x}}((1 - t)p_0 + t p_1), \quad 0 \leq t \leq 1.
\]
From \eqref{eqn:cexp}, $(y_t) \subset \Y$ is a $c$-segment (based at any $\overline{x}$) if and only if it is a dual geodesic. On the other hand, it is straightforward to see that a function $u(x)$ is $c$-convex if and only if $f = u + \Omega$ is a convex function of the form
\begin{equation} \label{eqn:Bregman.c.convex}
f(x) = u(x) + \Omega(x) = \sup_{z \in \Y} \left\{ x \cdot z - v(z) \right\}, \quad x \in X,
\end{equation}
for some function $v: \Y \rightarrow \mathbf{R}$. When $v$ is $C^1$ the $c$-gradient, when exists, is given by
\[
T(x) = D^c u(x) = \cexp_x(Du(x)) = D\Omega^*(Df(x)).
\]
So this $f$ corresponds to the Brenier potential $f$ in \eqref{eqn:BW.Brenier}. In particular, letting $v(z) = \Omega^*(z)$ (or $f = \Omega$) shows that $u(x) \equiv 0$ is $c$-convex and corresponds to the identity transport $T(x) \equiv x$. Suppose $u$ is a differentiable $c$-convex function such that $f$ is convex and $Df(\X) \subset \Y$. Then it can be shown using \eqref{eqn:Bregman.c.convex} that the convex combination $u_t = tu = (1 - t) \cdot 0 + t u$ is $c$-convex for $0 \leq t \leq 1$ and $f_t = (1 - t) \Omega + t f$. The $c$-gradient is given by
\begin{equation} \label{eqn:ut.cgradient}
T_t(x) = D \Omega^* ( (1 - t) D\Omega(x) + t Df(x)).
\end{equation}
Comparing \eqref{eqn:ut.cgradient} and \eqref{eqn:dual.generalized.geodesic}, we see that for a given $\mu_0^{\X} \in \mathcal{P}(\X)$, the pushforward $\mu_t^{\X} = (T_t)_{\#} \mu_0^{\X}$, $0 \leq t \leq 1$, is a dual displacement interpolation under the primal representation. Thus, the dual displacement interpolation corresponds to interpolation of $c$-convex potentials. Similarly, by considering the dual cost function $c^*(x, y) = {\bf B}_{\Omega}(y, x)$, we see that the primal displacement interpolation corresponds to interpolations of $c^*$-convex potentials.

\section{Bregman-Wasserstein geometry} \label{sec:gen-dual}
In this section, we use the Bregman-Wasserstein divergence to develop a generalized dualistic geometry on the space of probability measures. This section is primarily of theoretical interest, and the reader more interested in applications may proceed to Section \ref{sec:applications}.\footnote{Nevertheless, we believe that the results presented here will become more relevant to application with the advance of computational methods, such as {\it flow matching} \cite{haviv2024wasserstein,lipman2023fmg,lipman2024fmg} on families of probability distributions.} We first recall the {\it dualistic structure} of Bregman divergence in Section \ref{sec:Bregman.dualistic} and the Riemannian structure of the $2$-Wasserstein space in Section \ref{sec:prelim-otto-lott}. In Section \ref{sec:BW.Otto}, we construct a {\it generalized dualistic structure} on the space of probability measures by combing the Bregman and Wasserstein geometries. Further properties of this structure are studied in Section \ref{sec:parallel.transport}.

\subsection{Dualistic geometry of Bregman divergence} \label{sec:Bregman.dualistic}
In Section \ref{sec:Bregman.geometry} we showed that a Bregman divergence induces a Hessian Riemannian metric $g$. We also defined the primal and dual geodesics. Mathematically, these geodesics correspond to a pair of torsion-free affine connections (covariant derivatives), see \cite[Chapters 5--6]{A16}.

We define the {\it primal connection} $\nabla$ by the (torsion-free) Euclidean flat affine connection with respect to the primal coordinate system. This means the Christoffel symbols of $\nabla$ in the primal coordinates vanish and thus for any vector field $V$ we have $\nabla_V (u^i \frac{\partial}{\partial x^i}) = (Vu^i) \frac{\partial}{\partial x^i}$. Using the dual coordinate system, we define analogously the {\it dual connection $\nabla^*$}. Equivalently, these connections can be defined in terms of the Bregman divergence ${\bf B}$ via
\begin{equation} \label{eqn:dualistric.geometry.divergence}
\begin{split}
\langle \nabla_U V, W \rangle &= - U_p V_p W_{q} {\bf B}(p, q)|_{p = q}, \\
\langle \nabla_U^* V, W \rangle &= - U_{q} V_{q} W_{p} {\bf B}(p, q)|_{p = q}. \\
\end{split}
\end{equation}
The details of this construction (introduced in \cite{E92}) can be found in \cite[Sections 11.5--6]{CU14}.

That the geometric structure $(g, \nabla, \nabla^*)$ is termed {\it dualistic} stems from the following property first formalized in \cite{NA82}: $\nabla$ and $\nabla^*$ are  and are {\it conjugate} with respect to $g$ in the sense that
\begin{equation} \label{eqn:conjugate.connections}
U g(V, W) = g(\nabla_UV, W) + g(V, \nabla_U^*W),
\end{equation}
for vector fields $U, V, W$ on $\M$. By \eqref{eqn:conjugate.connections} the average $\overline{\nabla} = \frac{1}{2}(\nabla + \nabla^*)$ is the Levi-Civita connection induced by $g$. The triple $(g, \nabla, \nabla^*)$ satisfying \eqref{eqn:conjugate.connections} is called a {\it dualistic structure}. Since both $\nabla$ and $\nabla^*$ are flat (their curvature tensors vanish), we say that $(g, \nabla, \nabla^*)$ is {\it dually flat}. We also note that the difference $\nabla^* - \nabla$ can be identified with a symmetric $3$-tensor (called the {\it skewness tensor})
\begin{equation} \label{eqn:cubic.tensor}
T(U, V, W)  := g(\nabla_U^* V - \nabla_U V, W).
\end{equation}

\begin{proposition}[Third order expansion of Bregman divergence] \label{prop:divergence.expansion}
Let $\gamma_t$ be a dual geodesic on $M$ with $\dot{\gamma}_0 = v$. Then
\begin{equation*} %
{\bf B}(\gamma_0, \gamma_t) = \frac{t^2}{2}g(v, v) + \frac{t^3}{3!} T(v, v, v) + o(t^3),
\end{equation*}
where $T$ is the skewness tensor defined by \eqref{eqn:cubic.tensor}.

\end{proposition}
\begin{proof}
This is similar to \cite[Proposition 4.1]{AJVS17}, so we omit the details.
\end{proof}

\subsection{Riemannian structure of the $2$-Wasserstein space} \label{sec:prelim-otto-lott}
Let $(\M, g)$ be a Riemannian manifold with Riemannian distance $\mathsf{d}_g$ and volume measure $\dd \vol$. Equip $\mathcal{P}_2(\M)$ (the space of Borel probability measures on $\M$ with finite second moment) with the $2$-Wasserstein metric $\mathscr{W}_2 = \mathscr{W}_{2, g}$ defined by \eqref{eqn:Riemannian.2.Wasserstein}. In his seminal paper \cite{O01} Otto introduced a formal Riemannian structure on $(\mathcal{P}_2(\M), \mathscr{W}_2)$ and used it to interpret some evolutionary partial differential equations as gradient flows in the Wasserstein space. Lott \cite{L08} further developed the Riemannian geometry of this space, notably computing the Levi-Civita connection, the associated geodesic equation, and the curvature tensor. Further works of Ambrosio, Gigli, and Savare \cite{AGS08} and also Gigli \cite{Gigli12} removed smoothness assumptions and developed other aspects of the geometry.

To focus on the geometric ideas, we follow the approach of Otto and Lott and restrict our attention to formal computations on the space of absolutely continuous probability measures with smooth densities. If $\mu \in \mathcal{P}(\M)$ is absolutely continuous with respect to $\dd \vol$, we identify $\mu$ with its density $\rho = \frac{\dd \mu}{\dd \vol}$. Formally, we define $\mathcal{P}^{\infty}(\M)$ as the space of probability measures $\dd \mu = \rho \dd \vol \in \mathcal{P}(\M)$ whose densities $\rho$ are positive and smooth, and such that (i) $\rho$ decays sufficiently fast such that all integrals involved (such as $\int_{\M} \mathsf{d}_g^2(p_0, p) \rho(p) \dd \vol(p)$) converge, and (ii) all required operations, such as differentiation under the integral sign and integration by parts, can be carried out. The space $\mathcal{P}^{\infty}(\M)$ represents our ``manifold of probability measures''

Motivated by \cite[Theorem 8.3.1]{AGS08}, we associate with any (smooth) vector field $v$ on $M$ the tangent vector $\mathcal{V}_v$ at $\rho \in \mathcal{P}^{\infty}(M)$ via its action
\begin{equation}
  \label{eq:ags-tangent}
   (\mathcal{V}_v F)(\rho) :=  \frac{\dd}{\dd\epsilon}\Big|_{\epsilon=0} F(\rho  - \epsilon \divg(\rho \Pi_\rho(v))),
\end{equation}
where $\Pi_\rho(v)$ is the orthogonal projection of $v$ onto the space
\[
T_{\rho}\mathcal{P}^\infty(M) := \overline{\{\grad \phi ; \phi \in C^\infty_c(M)\}}^{L^2_\rho},
\]
which we regard as the {\it tangent space} to $\mathcal{P}^{\infty}(M)$ at $\rho \dd \vol$. The tangent vector \eqref{eq:ags-tangent} can be interpreted intuitively in terms of the {\it continuity equation}
\begin{equation} \label{eqn:continuity.equation}
\frac{\partial \rho_t}{\partial t} = - \divg (\rho_t w_t),
\end{equation}
where $w_t$ (corresponding to $\Pi_\rho(v)$ in \eqref{eq:ags-tangent}) is a time-dependent vector field. We give an example which will be used in several results below.

\begin{lemma} \label{lem:displacement.interpolation.fields}
On a Bregman manifold $\M$ (where $g$ is the Hessian metric), consider a curve $(\mu_t)_{0 \leq t \leq 1} \subset \mathcal{P}^{\infty}(\M)$. If $(\mu_t)$ is a primal displacement interpolation induced by the convex gradient $Dh : \Y \rightarrow \X$, then $(\mu_t)$ satisfies the continuity equation  \eqref{eqn:continuity.equation} at time $t = 0$ with
\[
v_0 = \grad \phi, \quad \phi = (h - \Omega^*) \circ \iota^*.
\]
Similarly, if $(\mu_t)$ is a dual displacement interpolation induced by the convex gradient $Df : \X \rightarrow \Y$, then the continuity equation is satisfied at $t = 0$ with
\[
v_0 = \grad \phi, \quad \phi = (f - \Omega) \circ \iota.
\]
In both cases we let $\dot{\mu}_0 := \mathcal{V}_{\grad \phi}$ be the initial velocity vector of the curve.
\end{lemma}
\begin{proof}
We consider here the dual displacement interpolation; the primal displacement interpolation can be handled similarly. For $0 \leq t \leq 1$, let
\[
T_t = (\iota^*)^{-1} \circ ((1 - t) D \Omega + t D f) \circ \iota : \M \rightarrow \M.
\]
Then $\mu_t = (T_t)_{\#} \mu_0$, from which it is well established that $\mu_t$ solves the continuity equation \ref{eqn:continuity.equation} at time $0$ with $v_0 = \frac{\dd}{\dd t}T_t\big\vert_{t=0}$ \cite[\S 4.1]{S15}. Note that for each $p$, $t \mapsto (\iota^* \circ T_t)(p) = (1 - t) D\Omega(p) + tDf(p)$ is a constant velocity straight line in $\Y$.

Consider on $M$ the function $\phi = (f - \Omega) \circ \iota$. Using \eqref{eqn:metric.inverse}, the Riemannian gradient of $\phi$ can be expressed under the primal coordinate system by
\[
\grad \phi(p) = \sum_{j = 1}^d D_{ij} \Omega^*(y_p) (D_i f(x_p) - D_i\Omega(x_p)) \left. \frac{\partial}{\partial x^j} \right|_p.
\]
On the other hand, since $y_p = D\Omega(x_p)$ and $\frac{\partial y}{\partial x}(p) = D^2 \Omega(x_p) = (D^2 \Omega(x_p))^{-1}$, for $p \in M$ we have
\[
\frac{\dd}{\dd t} \Big|_{t = 0} \big(T_t (p)\big) = (\grad \phi)(p).
\]
Thus $v_0 = \grad \phi$ is the initial velocity field of the flow associated with the dual displacement interpolation. It follows that $(\mu_t)$ satisfies the continuity equation at $t = 0$ with $v_0 = \grad \phi_0$.
\end{proof}

For $\rho \in \mathcal{P}^{\infty}(\M)$, {\it Otto's Riemannian metric} is defined by
\begin{equation} \label{eq:extended-otto-metric}
\mathfrak{g}_{\rho}(\mathcal{V}_v, \mathcal{V}_w) = \llangle \mathcal{V}_v,\mathcal{V}_w \rrangle_{\rho} := \int_{\M} \langle \Pi_\rho(v), \Pi_\rho(w)\rangle \rho\dd \vol.
\end{equation}
Let $\overline{\nabla}$ be the Levi-Civita connection on $(\M, g)$. Following Lott, the {\it Levi-Civita connection} $\overline{\nnabla}$ on $\mathcal{P}^{\infty}(\M)$ is defined for vector fields $\mathcal{V}_v, \mathcal{V}_w$ on $\mathcal{P}^{\infty}(M)$ (corresponding to vector fields $v, w$ on $\M$) by
\begin{equation} \label{eq:key-ident}
  \overline{\nnabla}_{\mathcal{V}_v}\mathcal{V}_w := \mathcal{V}_{\overline{\nabla}_v w}.
\end{equation}
The connection is used to define the geodesic equation and curvature tensor. What will, for our purposes, prove essential about this definition is that it is given entirely in terms of the Levi-Civita connection $\overline{\nabla}$ on $\M$.

\subsection{Generalized dualistic structure} \label{sec:BW.Otto}
Let $\M$ be a Bregman manifold. Using the dualistic structure $(g, \nabla, \nabla^*)$ on $\M$, we construct a {\it generalized dualistic structure} $(\mathfrak{g}, \nnabla, \nnabla^*)$ on $\mathcal{P}^{\infty}(\M)$. Some of the proofs are deferred to Appendix \ref{sec:appendix.proofs}.

First, we let $\mathfrak{g}$ be Otto's metric on $\mathcal{P}^{\infty}(\M)$ defined by \eqref{eq:extended-otto-metric}. We observe that $\mathfrak{g}$ arises in the second order expansion of the Bregman-Wasserstein divergence. This is precisely how the Hessian metric $g$ is constructed from the Bregman divergence, see \eqref{eqn:M.metric}.

\begin{proposition}[Second order expansion of Bregman-Wasserstein divergence] \label{prop:BW.metric}
Let $(\mu_t)_{0 \leq t \leq 1} \subset \mathcal{P}^{\infty}(\M)$ be a primal displacement interpolation. Then
\begin{equation} \label{eq:metric-compatibility2}
\mathscr{B}(\mu_t,\mu_0) = \frac{t^2}{2} \mathfrak{g}( \dot{\mu}_0, \dot{\mu}_0 ) + o(t^2).
\end{equation}
Similarly, if $(\mu_t)$ is a dual displacement interpolation then
\begin{equation} \label{eq:metric-compatibility}
\mathscr{B}(\mu_0,\mu_t) = \frac{t^2}{2} \mathfrak{g}( \dot{\mu}_0, \dot{\mu}_0) + o(t^2).
\end{equation}
\end{proposition}
\begin{proof}
We prove \eqref{eq:metric-compatibility} and leave \eqref{eq:metric-compatibility2} to the reader. We assume $\mu_t$ is a dual displacement interpolation, so is necessarily given by $\mu_t^{\Y} = ((1 - t)D \Omega + tDf)_{\#} \mu_0^{\X}$ for $Df : \X \rightarrow \Y$ the gradient of a suitable convex function. Let
\[
T_t = (\iota^*)^{-1} \circ ((1 - t) D\Omega + tDf) \circ \iota.
\]
By Proposition \ref{prop:solving.BW.transport}, we have
\begin{equation} \label{eqn:BW.metric.proof}
\mathscr{B}(\mu_0, \mu_t) = \int_{\M} {\bf B}(p, T_t(p)) \dd \mu_0(p).
\end{equation}

As in the proof of Lemma \ref{lem:displacement.interpolation.fields}, we have $\frac{\dd}{\dd t} |_{t = 0} T_t(p) = \grad \phi (p)$, where $\phi = (f - \Omega) \circ \iota$. It follows that $\frac{\dd}{\dd t} \big|_{t = 0} {\bf B}(p, T_t(p)) = 0$ and, by \eqref{eqn:M.metric},
\[
\frac{\dd^2}{\dd t^2} \Big|_{t = 0} {\bf B}(p, T_t(p)) = g(\grad \phi(p), \grad \phi(p)).
\]
Differentiating \eqref{eqn:BW.metric.proof} under the integral sign, which is possible since we are operating in $\mathcal{P}^{\infty}(\M)$, we have $\frac{\dd}{\dd t} \big|_{t = 0} \mathscr{B}(\mu_0, \mu_t) = 0$ and
\begin{equation*}
\begin{split}
\frac{\dd^2}{\dd t^2} \Big|_{t = 0} \mathscr{B}(\mu_0, \mu_t) &= \int_{\M} g(\grad \phi(p), \grad \phi(p)) \dd \mu_0(p)= \mathfrak{g}(\dot{\mu}_0, \dot{\mu}_0).
\end{split}
\end{equation*}
This completes the proof of \eqref{eq:metric-compatibility}.
\end{proof}

Next we define a pair of connections $(\nnabla, \nnabla^*)$ on $\mathcal{P}^{\infty}(\M)$, and show that they are torsion-free and conjugate with respect to $\mathfrak{g}$. Our definitions are natural extensions of Lott's (see \eqref{eq:key-ident}).

\begin{definition}[Primal and dual connections on $\mathcal{P}^{\infty}(\M)$] \label{defn:conjugate-connections-PM}
For vector fields $v,w$ on $\M$, we define the primal covariant derivative $\nnabla_{\mathcal{V}_v} \mathcal{V}_w$ and dual covariant derivative $\nnabla_{\mathcal{V}_v}^* \mathcal{V}_w$ by
\begin{equation} \label{eq:pm-primal-connection}
\nnabla_{\mathcal{V}_v} \mathcal{V}_{w} := \mathcal{V}_{\nabla_v w}, \quad
\nnabla_{\mathcal{V}_{v}}^* \mathcal{V}_{w} := \mathcal{V}_{\nabla_v^*w}.
\end{equation}
\end{definition}

Formally, we simply replace the Levi-Civita connection $\overline{\nabla}$ in \eqref{eq:key-ident} by the primal or dual connection to obtain \eqref{eq:pm-primal-connection}. The following theorem is the main result of this subsection. It shows that the triple $(\mathfrak{g}, \nnabla, \nnabla^*)$ can be regarded as a {\it generalized dualistic structure} on $\mathcal{P}^{\infty}(M)$.

\begin{theorem}[Generalized dualistic structure on $\mathcal{P}^{\infty}(M)$] \label{thm:BW.conjugate.connections} { \ }
\begin{itemize}
\item[(i)] Both $\nnabla$ and $\nnabla^*$ are torsion-free in the sense that
\begin{equation} \label{eqn:BW.torsion.free}
\begin{split}
 \nnabla_{\mathcal{V}_{v}} \mathcal{V}_{w} - \nnabla_{\mathcal{V}_{w}} \mathcal{V}_{v} = \nnabla_{\mathcal{V}_{v}}^* \mathcal{V}_{w} - \nnabla_{\mathcal{V}_{w}}^* \mathcal{V}_{v} = \llbracket \mathcal{V}_{v}, \mathcal{V}_{w} \rrbracket := \mathcal{V}_{v}\mathcal{V}_{w} - \mathcal{V}_{w} \mathcal{V}_{v},
\end{split}
\end{equation}
where $\llbracket \cdot, \cdot \rrbracket$ is the Lie bracket on $\mathcal{P}^{\infty}(M)$.
\item[(ii)] The two connections $\nnabla$ and $\nnabla^*$ are conjugate with respect to $\mathfrak{g}$ in the sense that for vector fields $\mathcal{V}_{u}, \mathcal{V}_{v}, \mathcal{V}_{w}$ we have
\begin{equation}\label{eq:wass-conjugate}
\mathcal{V}_{u}\mathfrak{g}(\mathcal{V}_{v}, \mathcal{V}_{w})= \mathfrak{g}(\nnabla_{\mathcal{V}_{u}} \mathcal{V}_{v}, \mathcal{V}_{w}) + \mathfrak{g}(\mathcal{V}_{v}, \nnabla_{\mathcal{V}_{u}}^* \mathcal{V}_{w}).
\end{equation}
\item[(iii)] The average of the two connections is $\frac{1}{2}(\nnabla + \nnabla^*) = \overline{\nnabla}$, the Levi-Civita connection on $\mathcal{P}^{\infty}(M)$ defined by \eqref{eq:key-ident}.
\end{itemize}
\end{theorem}

\begin{remark}[Infinite dimensional statistical manifolds]
Using functional analytic techniques, infinite dimensional manifolds of probability distributions have been constructed and studied by Pistone and others, see e.g., \cite{ay2015information, gibilisco1998connections, pistone2022affine, pistone1995infinite} and their references. Connections (such as the Amari-Chentsov connection) can also be constructed on these manifolds. Conceptually, these manifolds are related to infinite dimensional extensions of exponential families. Our approach, which is based on optimal transport, is quite different. Also see \cite{ay2024information} for another recent approach based on OT.
\end{remark}

We now relate the connections $\nnabla, \nnabla^*$ with the Bregman-Wasserstein divergence $\mathscr{B}$. On a Bregman manifold, the connections $\nabla, \nabla^*$ can be defined in terms of ${\bf B}$ by \eqref{eqn:dualistric.geometry.divergence}. Similarly to the Riemannian case \cite[Lemma 3]{L08}, the Christoffel symbols of $\nnabla$ and $\nnabla^*$ have simple expressions.

\begin{proposition}[Christoffel symbols of $\nnabla$ and $\nnabla^*$] \label{cor:christoffel}
  For vector fields $\mathcal{V}_u,\mathcal{V}_v,\mathcal{V}_w \in T_\rho \mathcal{P}^\infty(M)$ the generalized Christoffel symbols for $\nnabla$ and $\nnabla^*$ are given as follows:
  \begin{align*}
    \GGamma_{\mathcal{V}_v,\mathcal{V}_w,\mathcal{V}_u}(\rho) &:= \llangle \nnabla_{\mathcal{V}_{v}}\mathcal{V}_{w} , \mathcal{V}_{u} \rrangle_{\rho} =   \int_{\M}\langle\nabla_{v}w , u\rangle \rho \dd \vol, \\
    \GGamma^*_{\mathcal{V}_v,\mathcal{V}_w,\mathcal{V}_u}(\rho) &:= \llangle \nnabla^*_{\mathcal{V}_{v}}\mathcal{V}_{w} , \mathcal{V}_{u} \rrangle_{\rho} =   \int_{\M}\langle\nabla^*_{v}w , u\rangle \rho\dd \vol.
  \end{align*}
\end{proposition}

Alternatively, the skewness tensor $T = \nabla^* - \nabla$ (see \eqref{eqn:cubic.tensor}) arises in the third order expansion of the Bregman divergence. We state the analogous result for a dual displacement interpolation. Clearly, a similar expansion exists along a primal displacement interpolation.

\begin{theorem}[Third order expansion of Bregman-Wasserstein divergence]  \label{thm:cubic.expansion}
  Let $\mu_t$ be a dual displacement interpolation with $\dot{\mu_0} = \mathcal{V}_{\grad \phi}$ (see Lemma \ref{lem:displacement.interpolation.fields}). Then
  \[
  \mathscr{B}(\mu_0,\mu_t) = \frac{t^2}{2}\mathfrak{g}(\mathcal{V}_\phi,\mathcal{V}_\phi)+\frac{t^3}{3!}\mathfrak{T}_{\mathcal{V}_\phi,\mathcal{V}_\phi,\mathcal{V}_\phi} + o(t^3),
  \]
  where $\mathfrak{T}_{\mathcal{V}_\phi,\mathcal{V}_\phi,\mathcal{V}_\phi} =  \GGamma^*_{\mathcal{V}_\phi,\mathcal{V}_\phi,\mathcal{V}_\phi} - \GGamma_{\mathcal{V}_\phi,\mathcal{V}_\phi,\mathcal{V}_\phi}$.%
\end{theorem}

\subsection{Parallel transport and curvature} \label{sec:parallel.transport}
Let $\mu_t$ be a curve in $\mathcal{P}^\infty(\M)$ with $\mu_t = \rho_t \dd \vol$ and velocity $\dot{\mu}_t = \mathcal{V}_{v_t}$. Let $\mathcal{V}_{w_t}$ be a vector field (on the Wasserstein space) along $\mu_t.$ For each $t$ we have that $v_t,w_t$ are vector fields on $\M$ satisfying $\Pi_{\rho_t}(v_t) = v_t$ and similarly for $w_t$. In \cite[eq.~4.15]{L08} Lott introduced a notion of {\it parallel transport} in the Wasserstein space and used it to relate geodesics with displacement interpolation. The associated equation of parallel transport is $\mathcal{V}_{\partial_t w_t} + \overline{\nnabla}_{\mathcal{V}_{v_t}}\mathcal{V}_{w_t} = 0$, and in this case $\mathcal{V}_{w_t}$ is said to be {\it parallel} along the curve $\mu_t$. The natural extension to our setting is the following:
\begin{itemize}
\item $\mathcal{V}_{w_t}$ is a {\it primal parallel transport} along $\mu_t$ if $\mathcal{V}_{\partial_t w_t} + \nnabla_{\mathcal{V}_{v_t}}\mathcal{V}_{w_t} = 0$.%
  \item $\mathcal{V}_{w_t}$ is a {\it dual parallel transport} along $\mu_t$ if $\mathcal{V}_{\partial_t w_t} + \nnabla_{\mathcal{V}_{v_t}}^*\mathcal{V}_{w_t} = 0$.
\end{itemize}
The following proposition is a straightforward consequence of  the same reasoning employed by Lott \cite[Proposition 3]{L08}.
\begin{proposition}[Parallel transport equations] { \ } \label{prop:par-transport}
  \begin{enumerate}
  \item[(i)] $\mathcal{V}_{w_t}$ is a primal parallel transport along $\mu_t$ provided
    \[ \divg\left(\rho_t\left(\partial_tw_t + \nabla _{v_t}w_t\right)\right) = 0. \]

  \item[(ii)] $\mathcal{V}_{w_t}$ is a dual parallel transport along $\mu_t$ provided
    \[ \divg\left(\rho_t\left(\partial_tw_t + \nabla^*_{v_t}w_t\right)\right) = 0.\]
  \end{enumerate}
\end{proposition}

We prove two results which exhibit that the theory introduced so far is consistent. First, if $\mathcal{V}_{w^{(1)}_t}$ and $\mathcal{V}_{w^{(2)}_t}$ are, respectively, primal and dual parallel transports then their inner product is preserved -- a direct analogue of a characteristic property of the dualistic geometry on $\M$ (see \cite[Chapter 6]{A16}. Second, we show that the primal and dual displacement interpolations are, respectively, primal and dual geodesics in the sense of parallel transport.

\begin{proposition}
  Let $\mu_t$ be as above with $\dot{\mu}_t = \mathcal{V}_{v_t}$. Let $\mathcal{V}_{w^{(1)}_t},\mathcal{V}_{w^{(2)}_t}$ be, respectively, primal and dual parallel transports along $\mu_t$.
 Then
\[ \frac{\dd}{\dd t}\llangle \mathcal{V}_{w^{(1)}_t}, \mathcal{V}_{w^{(2)}_t} \rrangle_{\mu_t} = 0.\]
\end{proposition}
\begin{proof}
  To begin, recall from \eqref{eq:extended-otto-metric} that
  \[ \llangle \mathcal{V}_{w^{(1)}_t}, \mathcal{V}_{w^{(2)}_t} \rrangle_{\mu_t} = \int_{\M}\langle w^{(1)}_t,w^{(2)}_t\rangle\rho_t \dd \vol. \]
  Thus on differentiating, we have
  \begin{equation*}
  \begin{split}
&\frac{\dd}{\dd t}\llangle \mathcal{V}_{w^{(1)}_t}, \mathcal{V}_{w^{(2)}_t} \rrangle_{\mu_t} \\
&= \int_{\M}[\langle\partial_tw^{(1)}_t,w^{(2)}_t\rangle + \langle w^{(1)}_t,\partial_tw^{(2)}_t\rangle]\rho(t) \dd \vol - \int_{\M}\langle w^{(1)}_t,w^{(2)}_t\rangle \divg (\rho_t v_t) \dd \vol\\
&= \int_{\M} [\langle\partial_tw^{(1)}_t,w^{(2)}_t\rangle + \langle w^{(1)}_t,\partial_tw^{(2)}_t\rangle]\rho_t \dd \vol \int_{\M} v_t(\langle w^{(1)}_t,w^{(2)}_t\rangle) \rho_t \dd \vol.
  \end{split}
  \end{equation*}
  Now using that the primal and dual connections are conjugate on the base space (see \eqref{eqn:conjugate.connections}) this becomes
\begin{equation*}
\begin{split}
 &   \frac{\dd}{\dd t}\llangle \mathcal{V}_{w^{(1)}_t}, \mathcal{V}_{w^{(2)}_t} \rrangle_{\mu_t} \\
 &= \int_{\M}[\langle\partial_tw^{(1)}_t+\nabla_{v_t}w^{(1)}_t,w^{(2)}_t \rangle+\langle w^{(1)}_t,\partial_tw^{(2)}_t+\nabla^*_{v_t}w^{(2)}_t\rangle]\rho \dd \vol.
\end{split}
\end{equation*}
This expression is $0$ by our assumption that $w^{(1)}_t$ and $w^{(2)}_t$ are, respectively, primal and dual parallel transports. To see, for example, that the integral of the first term in the integrand is $0$ approximate $w^{(2)}_t$ by a sequence of smooth gradients and integrate by parts. This approximation is possible by our assumption that $\mathcal{V}_{w^{(2)}_t}$ is a tangent vector to $\mathcal{P}^\infty(\M)$ at $\mu_t$.
\end{proof}

\begin{theorem} \label{thm:geodesic.equations}
Assume $(\mu_t)_{0 \leq t \leq 1} \subset \mathcal{P}^{\infty}(\M)$ is a primal displacement interpolation. Then $\dot{\mu}_t$ is a primal parallel transport along $\mu_t$. If, alternatively, $(\mu_t)$ is a dual displacement interpolation then $\dot{\mu}_t$ is a dual parallel transport along $\mu_t$.
\end{theorem}
\begin{proof}
  We just prove the latter; as usual the proof is similar in both cases. Working in the dual coordinates, because $\mu_t$ is a dual displacement interpolation we obtain
  \[ \mu_t^{\Y} = [(1-t) \mathrm{Id} +tDf \circ D\Omega^*]_{\#}\mu_0^{\Y}, \]
  for an appropriate convex function $f$. We will work in the dual coordinates. We recall that $\mu_t$ solves the continuity equation $\dot{\mu}_t = -\text{div}(v_t \mu_t)$, for $v_t$ such that $Y(t,y):=(1-t)y+tDf(D\Omega^*(y))$ is the flow of $v_t$. That is $v_t$, necessarily satisfies
  \begin{equation}
   \frac{\dd}{\dd t}Y(t,y) = v_t(Y(t,y)).\label{eq:ddt-v-def}
 \end{equation}
  Differentiating this expression yields
  \[ 0 = \frac{\dd^2}{\dd t^2}Y(t,y) = \frac{\partial v_t}{\partial t} + \frac{\partial v_t}{\partial y^i}\frac{\dd}{\dd t}Y^i.\]
  Here the first equality is because $Y(t,y)$ is linear in $t$. Using once again \eqref{eq:ddt-v-def} we have $0 = \frac{\partial v_t}{\partial t} + \frac{\partial v_t}{\partial y^i}v_t^i$.  
  Finally, because we are working in the dual coordinates and thus the Christoffel symbols for the dual connection vanish we've in fact computed $0 =  \frac{\partial v_t}{\partial t} + \nabla^*_{v_t}v_t$, which implies $\mathcal{V}_{v_t}$ is dually parallel transported along $\mu_t$.
\end{proof}

\begin{remark}
The equations in Proposition \ref{prop:par-transport} are not in, general, invariant under reversal of time. This reflects that in accordance with our definition of primal and dual displacement interpolation, the primal displacement inerpolation between $\mu_1$ and $\mu_0$ is not, in general, the time reversal of the primal displacement inerpolation between $\mu_0$ and $\mu_1$. This arises because the composition of $D\Omega$ or $D\Omega^*$ with the gradient of a convex function is not, in general, the gradient of a convex function.
\end{remark}

Next we study the curvature tensors of $\nnabla$ and $\nnabla^*$. For simplicity, we consider here only gradient vector fields on $\M$. To simplify the notation in this section we set $\mathcal{V}_i = \mathcal{V}_{\grad \phi_i}$ and $\Phi_i = \grad \phi_i$, $i = 1, \ldots, 4$.  Following the standard definition in differential geometry (see \cite{L18}), the {\it primal curvature tensor} $\mathfrak{R}$ is defined by
\begin{equation}
\begin{split}
 &\llangle \mathfrak{R}(\mathcal{V}_1, \mathcal{V}_2)\mathcal{V}_3, \mathcal{V}_4 \rrangle := \llangle\nnabla_{\mathcal{V}_1}\nnabla_{\mathcal{V}_2}\mathcal{V}_3-\nnabla_{\mathcal{V}_2}\nnabla_{\mathcal{V}_1}\mathcal{V}_3 - \nnabla _{\llbracket \mathcal{V}_1,\mathcal{V}_2\rrbracket}\mathcal{V}_3,\mathcal{V}_4
 \rrangle.\label{eq:curv-def}
 \end{split}
\end{equation}
The {\it dual curvature tensor} $\mathfrak{R}^*$ is defined similarly.

The sectional curvature is more subtle. For concreteness, in this paper we consider the direct extension of the Riemannian case. Namely, we define the {\it primal sectional curvature} in the plane generated by $\mathcal{V}_1,\mathcal{V}_2$ by
  \begin{align}\label{eq:sec-def}
    \mathfrak{K}(\mathcal{V}_1,\mathcal{V}_2) &:= \frac{\llangle \mathfrak{R}(\mathcal{V}_1,\mathcal{V}_2)\mathcal{V}_2,\mathcal{V}_1 \rrangle}{\llangle \mathcal{V}_1,\mathcal{V}_1 \rrangle  \ \llangle \mathcal{V}_2,\mathcal{V}_2\rrangle  - \llangle \mathcal{V}_1,\mathcal{V}_2 \rrangle^2} ,              \end{align}
  and similarly for the {\it dual sectional curvature} $\mathfrak{K}^*$. Note that since $\nnabla$ is generally not the Levi-Civita connection $\overline{\nnabla}$, some symmetries of the curvature tensor are lost and we may have $\mathfrak{K}(\mathcal{V}_1,\mathcal{V}_2) \neq \mathfrak{K}(\mathcal{V}_2,\mathcal{V}_1)$. See Opozda's paper \cite{opozda2015bochner} for a modified definition of sectional curvature which keeps the symmetry; its implications are discussed in Remark \ref{rmk:sectional.curvature.def} below.

  On the Bregman manifold which is dually flat, the primal and dual curvature tensors $R$, $R^*$ coincide as they both vanish. More generally (see e.g.~\cite[Proposition 8.1.4]{CU14}), on a statistical manifold the primal and dual curvature tensors coincide when one has constant sectional curvature. In \cite[Theorem 3]{nielsen20}, this result is even called the {\it Fundamental Theorem of Information Geometry}. While the generalized Pythagorean inequality (Theorem \ref{thm:BW.Pyth}) suggests that $(\mathfrak{g}, \nnabla, \nnabla^*)$ is not dually flat (when $d \geq 2$), we show that the sectional curvatures $\mathfrak{K}$ and $\mathfrak{K}^*$ are still ``essentially the same'' in the sense that $\mathfrak{K}(\mathcal{V}_1, \mathcal{V}_2) = \mathfrak{K}^*(\mathcal{V}_2, \mathcal{V}_1)$. On the other hand, the when $d=1$ the generalized Pythagorean inequality suggests that $(\mathfrak{g}, \nnabla, \nnabla^*)$ is dually flat, and here we demonstrate that rigorously. These results indicate the significance of the Bregman-Wasserstein geometry as an OT-based divergence.

\begin{theorem}[Primal and dual sectional curvatures] \label{thm:sectional.curvature}
  Let vector fields $\mathcal{V}_1,\mathcal{V}_2$ be as above. Recall that $d = \dim \M$. Then:
\begin{itemize}
\item[(i)] When $d = 1$ both the primal and dual curvatures $\mathfrak{R}, \mathfrak{R}^*$ and the primal and dual sectional curvatures $\mathfrak{K},\mathfrak{K}^{*}$ are identically $0$.%
\item[(ii)] When $d \geq 2$ we have the identity $\mathfrak{K}(\mathcal{V}_1,\mathcal{V}_2) =\mathfrak{K}^*(\mathcal{V}_2,\mathcal{V}_1)$.
\end{itemize}
\end{theorem}

\begin{remark}[Sign of the sectional curvature] \label{rmK;curvature.sign}
As mentioned in Section \ref{sec:Pyth}, the Levi-Civita connection $\overline{\nnabla}$ induced by Otto's metric on $\mathcal{P}^{\infty}(\mathbb{R}^d)$ is flat when $d = 1$ but has nonnegative sectional curvature when $d \geq 2$. A natural question is whether the primal and dual sectional curvatures are also nonnegative when $d \geq 2$. While the generalized Pythagorean theorem (Theorem \ref{thm:BW.Pyth}) suggests that this might be the case (see Remark \ref{rmk:Pyth.curvature}), we have been unable to prove this rigorously.
\end{remark}

\begin{remark}[Equality of primal and dual sectional curvature] \label{rmk:sectional.curvature.def}
If we use Opozda's definition of the curvature tensor and the sectional curvature \cite[Section 12.2]{opozda2015bochner}, then the sectional curvature is symmetric in $\mathcal{V}_1, \mathcal{V}_2$ and we obtain exact equality of the primal and dual sectional curvatures. However, as in Remark \ref{rmK;curvature.sign}, even with this modified definition we have not been able to determine the sign of the sectional curvature.
\end{remark}

\section{Experiments and Applications} \label{sec:applications}
\subsection{Implementation with neural OT} \label{sec:neural.OT}

Learning transport maps between distributions has been fruitfully
applied to many problems in science, and has garnered interest
from the machine learning community. Due to the explosion of OT applications in machine learning~\cite{peyre2019cot}, we provide
a brief overview of the techniques developed for the computation
of optimal transport maps and costs, then specialize to the Bregman-Wasserstein divergence. When the measures are discrete the OT problem is equivalent to the linear assignment problem and the {\it Hungarian algorithm}~\cite{kuhn1955hma} is well studied in this context. However, its cubic runtime is prohibitively expensive.

{\it Learning via regularization} yields improved learning efficiency. Cuturi~\cite{cuturi2013sdl}  popularized the use of {\it entropic regularization} to estimate optimal
couplings with the {\it Sinkhorn algorithm}, which scales quadratically in
the number of samples. Subsequently, it has been applied
to compute barycenters~\cite{cuturi2014fcw}, entropic gradient flows in the
Wasserstein space~\cite{peyre2015eaw} and more; see~\cite{peyre2019cot} for
a summary of recent progress. Recently, neural approaches have become
popular for their improved performance, owing to their universal
approximation guarantees and the development of large scale stochastic
optimization techniques. Such learning paradigms when applied to the OT problem
can be categorized into two main types: primal and dual solvers. A {\it primal solver}
learns the OT map directly with a loss function which relies on the primal
Monge-Kantorovich problem. The Monge
gap~\cite{UC23} is a recent example of this approach. {\it Dual solvers} instead exploit Kantorovich's duality to
estimate Kantorovich potentials which can be used to construct
the OT map.
Empirical evidence suggests dual approaches often outperform primal
ones, possibly due to relatively stable optimization landscapes afforded by
Kantorovich's duality~\cite{korotin2021dno}, and as a result, often have faster
training times.

\begin{figure}[t!]
  \begin{subfigure}[t]{0.45\textwidth}
  \begin{center}
    \includegraphics[width=\linewidth]{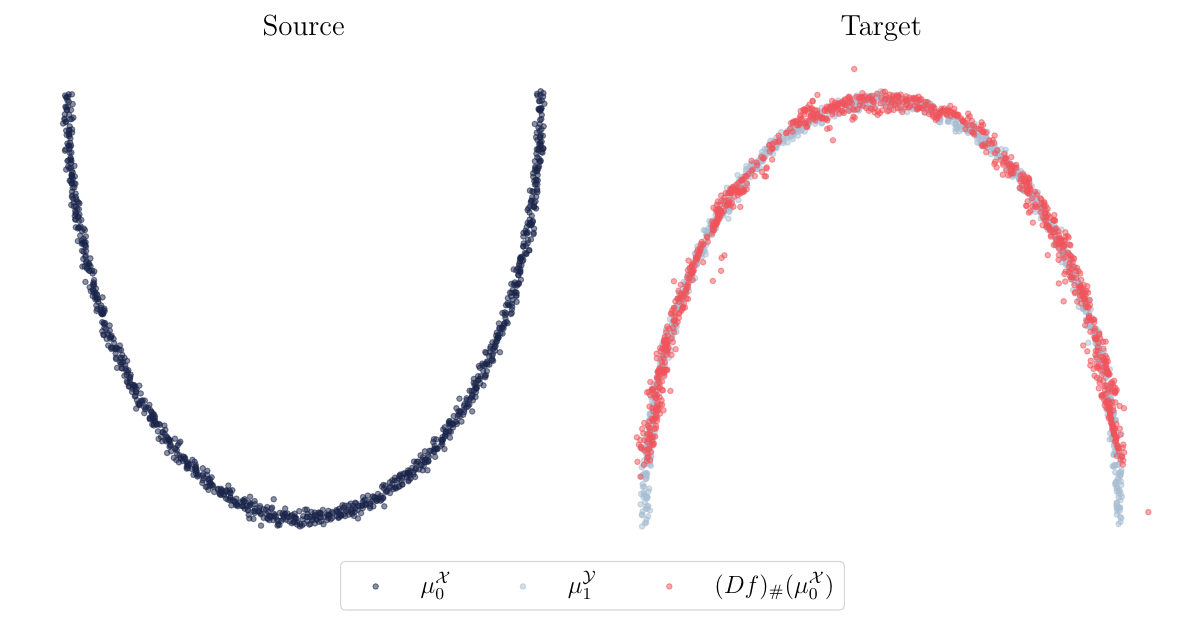}
  \end{center}
  \caption{\footnotesize $\Omega_i(x^i) = \frac{1}{2}(x^i)^2$}
  \label{fig:ddi_rout_2_Euclidean}
  \end{subfigure}
  \begin{subfigure}[t]{0.45\textwidth}
  \begin{center}
  \includegraphics[width=\linewidth]{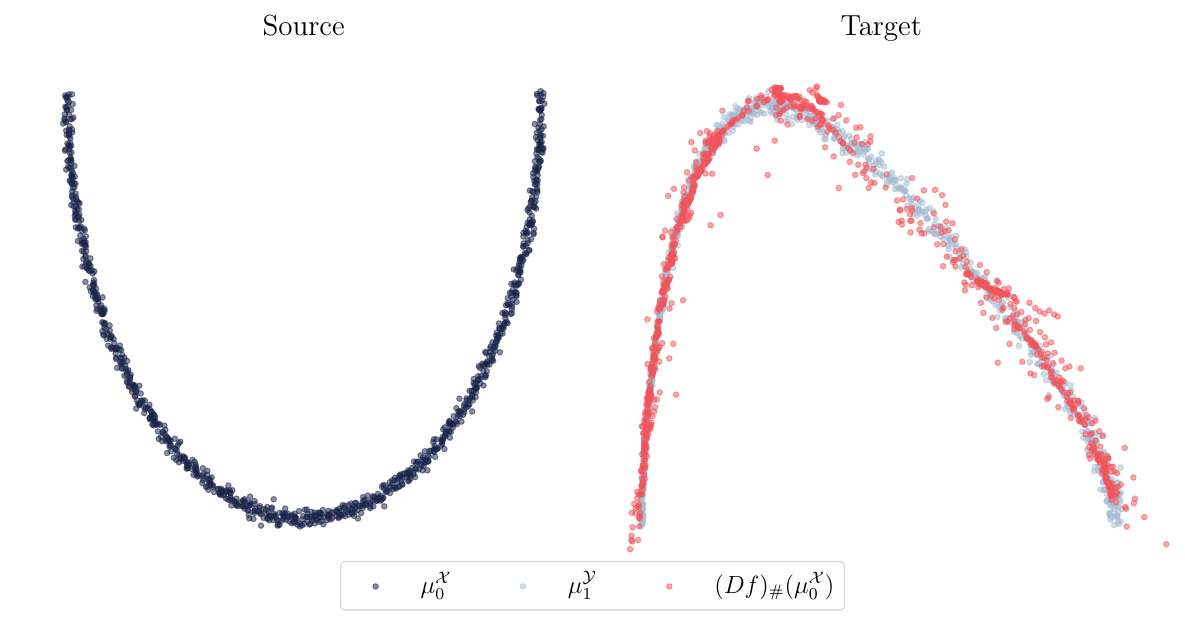}
  \end{center}
  \caption{\footnotesize $\Omega_i(x^i) = \exp(x^i)$}

  \label{fig:ddi_rout_2_ConjugateExtendedKL__a_1.0_}
  \end{subfigure}
  \begin{subfigure}[t]{0.45\textwidth}
  \begin{center}
    \includegraphics[width=\linewidth]{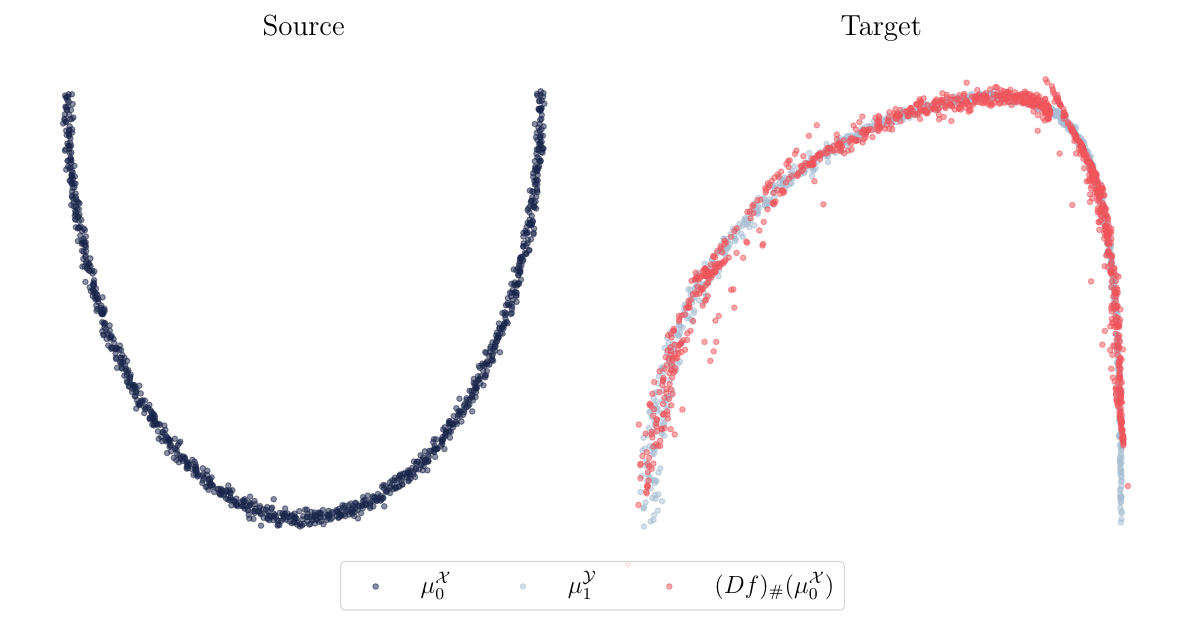}
  \end{center}
  \caption{\footnotesize $\Omega_i(x^i) = \exp(-x^i)$}
  \label{fig:ddi_rout_2_ConjugateExtendedKL__a_-1.0_}
  \end{subfigure}
  \begin{subfigure}[t]{0.45\textwidth}
  \begin{center}
    \includegraphics[width=\linewidth]{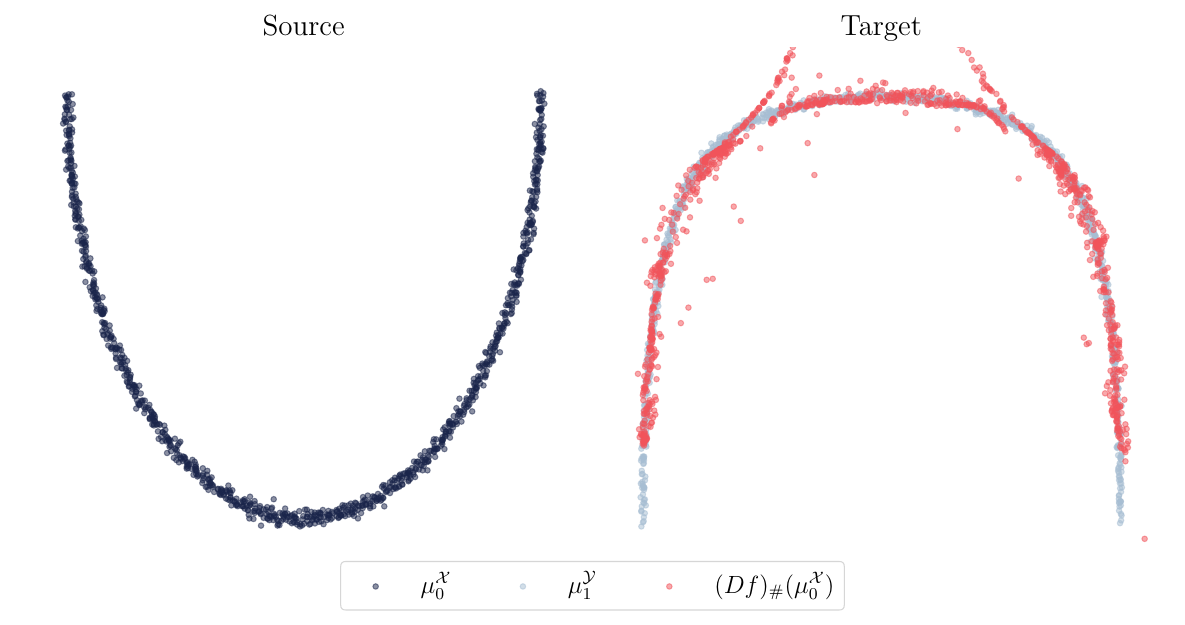}
  \end{center}
  \caption{\footnotesize $\Omega_i(x^i) = \frac{1}{2}\log\left( 1 + \exp(2x^i)
  \right) $}
  \label{fig:ddi_rout_2_ConjugateHNNTanh__beta_1.0_}
  \end{subfigure}
  \begin{subfigure}[t]{0.9\textwidth}
  \begin{center}
    \includegraphics[width=\linewidth]{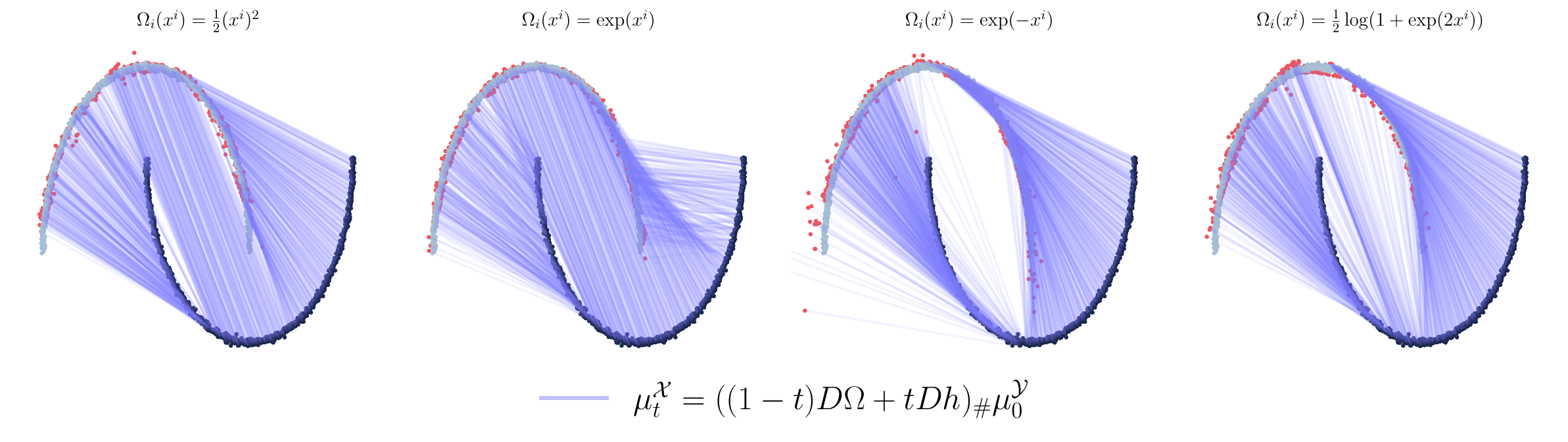}
  \end{center}
  \caption{\footnotesize Primal displacement interpolations.}
  \label{fig:rout_2_pdi}
  \end{subfigure}
  \begin{subfigure}[t]{0.9\textwidth}
  \begin{center}
    \includegraphics[width=\linewidth]{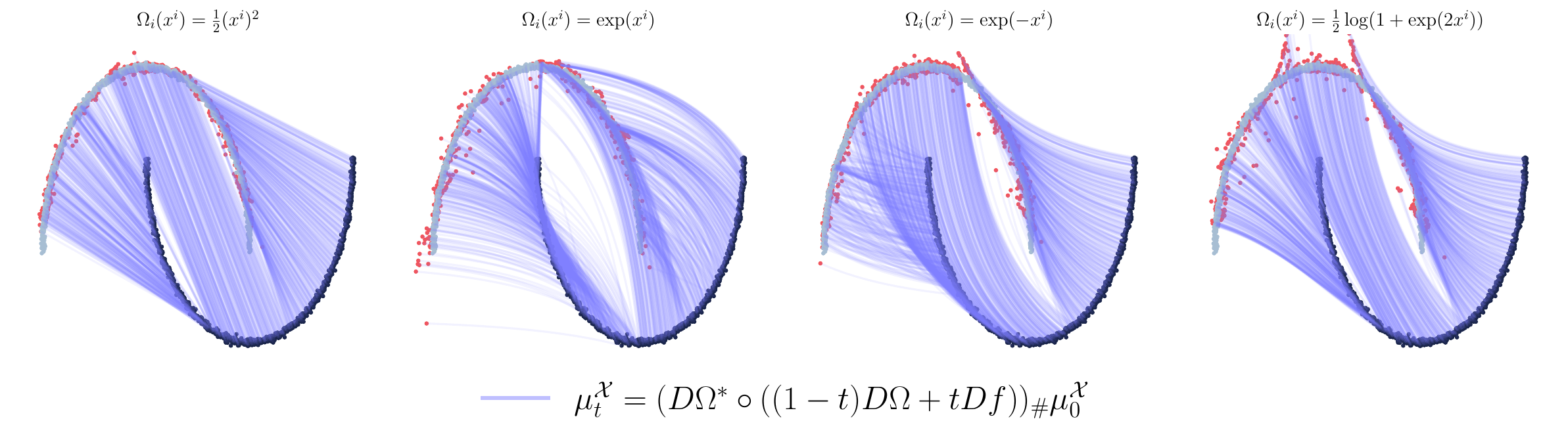}
  \end{center}
  \caption{\footnotesize Dual displacement interpolations.}
  \label{fig:rout_2_ddi}
  \end{subfigure}
  \caption{\footnotesize Bregman-Wasserstein OT in the context of Example \ref{eg:neural.OT}. 
Figures~\ref{fig:ddi_rout_2_Euclidean}-\ref{fig:ddi_rout_2_ConjugateHNNTanh__beta_1.0_}
plot the recovered Brenier map $Df$ from $\mu_0^{\X}$ to $\mu_1^{\Y}$ for
various choices of the separable Bregman generator
, while Figures~\ref{fig:rout_2_pdi} and~\ref{fig:rout_2_ddi}
trace (on the primal domain $\X$) the primal and dual displacement interpolations induced by the Brenier maps
$Dh$ and $Df$ for the corresponding $\Omega$.}%
  \label{fig:ddi_rout_2}
\end{figure}

Most existing work and applications focus on the Euclidean $2$-Wasserstein distance. However, it is not always the optimal choice. Proposition~\ref{prop:solving.BW.transport} enables the estimation of OT maps
corresponding to the Bregman-Wasserstein divergence using the same techniques as the $\mathscr{W}_2^2$. Crucially, by \eqref{eqn:BW.Euclidean.computation} it is straightforward to adapt any approach for estimating the $\mathscr{W}_2^2$ to our setting.
Given a Bregman generator $\Omega$, it is sufficient to solve the
$2$-Wasserstein problem between $\mu_0^{\X}$ and $\mu_1^{\Y}$ (or $\mu_0^{\Y}$
and $\mu_1^{\X}$) to recover the Bregman-Wasserstein OT map by composition
(see Equations~\eqref{eqn:primal.OT.map} and~\eqref{eqn:dual.OT.map}). The
primal and dual displacement interpolations can be computed similarly.

In this subsection, we empirically validate the use of the
Bregman-Wasserstein for transporting between synthetic datasets. For visualization purposes, we focus on two-dimensional settings, but the approach also generalizes to higher dimensions. Among the many possible approaches, we adopt the {\it ENOT solver}~\cite{buzun2024erf}, which demonstrated superior performance on several benchmark tasks. It relies on a form of expectile regularization to encourage optimality of the learned Kantorovich potentials and is {\it bidirectional} -- it jointly learns the
transport map and its inverse through dual potentials parameterized by neural networks.

\begin{example} \label{eg:neural.OT}
On the primal domain $\X := \mathbb{R}^2$, we consider two synthetic distributions $\mu_0^{\X}$ ($\smallsmile$-shaped) and $\mu_1^{\X}$ ($\smallfrown$-shaped) that are taken from ~\cite{rout2021gmw} and visualized in Figure~\ref{fig:ddi_rout_2}. We consider separable Bregman generators on $\X$ of the form $\Omega(x) \defeq \sum_{i}
\Omega_i(x^i)$, where each $\Omega_i$ is a univariate convex function. This corresponds to Example \ref{eg:Hopfield} where $\sigma_i = \Omega_i'$. In particular, $\Omega_i(x^i) = \frac{1}{2}(x^i)^2$ recovers the self-dual generator $\frac{1}{2}|x|^2$, which induces ($1/2$ times) the Euclidean $2$-Wasserstein distance (Example \ref{eg:reduces.to.Euclidean.W2}). For this and three other choices of the generator $\Omega_i$, we estimate the Brenier map $Df$ from $\mu_0^{\X}$ to $\mu_1^{\Y}$ using the ENOT solver (Figures~\ref{fig:ddi_rout_2_Euclidean}-\ref{fig:ddi_rout_2_ConjugateHNNTanh__beta_1.0_}). We also estimate the primal and dual displacement interpolations (Figures~\ref{fig:rout_2_pdi} and~\ref{fig:rout_2_ddi}).

Figure~\ref{fig:rout_2_pdi} illustrates the primal displacement interpolation
by estimating the Brenier map $Dh$ (see \eqref{eqn:primal.generalized.geodesic}) for various Bregman generators, and while
these correspond to straight lines in $\X$, the transport maps can evidently be diverse,
and distinct from the $\Wt^2$-transport. Furthermore, it is clear from
Figure~\ref{fig:rout_2_ddi} that the OT maps inducing the primal and dual
displacement interpolations are not identical or even similar, except for $\Wt^2$ (case A) which is self-dual. Additional details and examples can be found in Appendix \ref{sec:detail.neural.OT}.

\end{example}

\subsection{Bregman-Wasserstein barycenters} \label{sec:barycenters}
We extend Agueh and Carlier's notion of $\mathscr{W}_2^2$ barycenters \cite{AC11} to one which incorporates the Bregman-Wasserstein divergence and illustrate its advantages with an example. Then, we apply it to Bayesian statistics following the approach of \cite{backhoff-veraguas2022blw}.

Intuitively, a barycenter is a particular weighted average of a collection of elements. The simplest example is the barycenter of $K$ points in Euclidean space. Given $K$ elements $x_1,\dots,x_K \in \mathbb{R}^d$, and weights $\lambda_i \geq 0$ satisfying $\sum_{i = 1}^K \lambda_i = 1$, the {\it Euclidean barycenter} is the point
\begin{equation} \label{eqn:Euclidean.barycenter}
\bar{x}_{\lambda} = \argmin_{x \in \mathbb{R}^d} \sum_{i = 1}^K \lambda_i |x - x_i|^2.
\end{equation}

Let $\nu_1, \ldots, \nu_K \in \mathcal{P}_2(\mathbb{R}^d)$ be given and weights $\lambda_1, \ldots, \lambda_K$ be as above. Following \cite{AC11}, we say that a measure $\bar{\nu}_{\lambda}$ is a {\it Wasserstein barycenter} if
\begin{equation} \label{eqn:Wasserstein.barycenter}
\bar{\nu}_{\lambda} \in \argmin_{\nu \in \mathcal{P}_2(\mathbb{R}^d)} \sum_{i = 1}^K \lambda_i \mathscr{W}_2^2(\nu_i, \nu).
\end{equation}
Uniqueness of the Wasserstein barycenter, whilst true in reasonable circumstances, is no longer a given. The Wasserstein barycenter has been extended in many directions and applied in statistics and machine learning \cite{PC19}. %
Analogous constructions have been used for information-theoretic divergences. For example, Amari \cite{A07} characterized $\alpha$-integration as the barycenter with respect to the $\alpha$-divergence.

We formulate a natural extension of \eqref{eqn:Wasserstein.barycenter}. Let $M$ be a given Bregman manifold. Given $\nu_1, \ldots, \nu_K \in \mathcal{P}(\M)$ and weights $\lambda_1, \ldots, \lambda_K$, we say that $\bar{\nu} \in \mathcal{P}(\M)$ is a {\it Bregman-Wasserstein (right) barycenter} if
\begin{equation} \label{eq:breg-bary}
 \bar{\nu}_{\lambda} \in \argmin_{\nu \in \mathcal{P}(M)}\sum_{i = 1}^K \lambda_i\mathscr{B}(\nu_i,\nu),
\end{equation}
for given $\nu_i \in \mathcal{P}(M)$ and weights $\lambda_i$. Using the self-dual representation \eqref{eqn:BW.quadratic.case}, the {\it left barycenter} can be studied similarly. One may regard \eqref{eq:breg-bary} as model averaging under the Bregman-Wasserstein loss. In Section \ref{sec:Bayesian}, we extend this interpretation precisely, in a Bayesian context.

The Bregman-Wasserstein setting is particularly interesting, since the Bregman (right) barycenter is explicitly known to be the Euclidean weighted sum. More precisely, Banerjee et al.~\cite{BMDG05} showed that for $x_1, \ldots, x_K \in \X$, we have
\begin{equation} \label{eqn:Bregman.right.barycenter}
\argmin_{x \in \X} \sum_{i = 1}^K \lambda_i {\bf B}_{\Omega}(x_i, x) = \bar{x}_{\lambda} = \sum_{i = 1}^K \lambda_i x_i.
\end{equation}
In fact, this property characterizes the Bregman divergence; for precise statements, see \cite[Section VI]{BGW05}. Using this property, we obtain the following result.

\begin{theorem} [Bregman-Wasserstein barycenter] \label{thm:BW.barycenter.well.posed}
Let $\nu_1, \dots,\nu_K$ be compactly supported probability measures on $\M$, absolutely continuous with respect to the volume measure on $\M$.  Then there exists a right Bregman-Wasserstein barycenter.
\end{theorem}

Our approach is to use \eqref{eqn:Bregman.right.barycenter} to show that \eqref{eq:breg-bary} is equivalent to a suitable {\it multimarginal optimal transport problem} as in the Wasserstein barycenter \cite{AC11}. Then, we invoke the results of \cite{P11} to obtain existence, and a characterization of the Bregman-Wasserstein barycenter. The proof is given in Appendix \ref{sec:BW.barycenter.proof}, where we explain the multi-marginal formulation and provide a characterization of the barycenter. We also note that in applications the distributions are typically discrete, and hence compactly supported.

\begin{example} \label{eg:BW_barycenter}
Consider the Bregman manifold in Example \ref{eg:simplex.cgf} where the dual domain can be identified with the unit simplex and the dual generator is the negative Shannon entropy. For each column in Figure \ref{fig:bar_simplex}, we consider several synthetic marginal distributions $\nu_i$ that are arranged to vary roughly along a primal geodesic. Using the algorithm described in Appendix \ref{sec:details.barycenter} (in which a complete description of the setup can be found), we compute the left Bregman-Wasserstein barycenter, as well as the Wasserstein barycenter (with respect to the Euclidean geometry on the simplex). From the figure, it is clear that the Bregman-Wassrestein barycenter is more natural, as it concentrates on the primal geodesic, whereas the Wasserstein barycenter instead concentrates on the dual geodesic connecting the endpoint marginals (which have higher weights by design).
\end{example}

\begin{figure}[t!]
  \centering
  \includegraphics[width=0.99\textwidth]{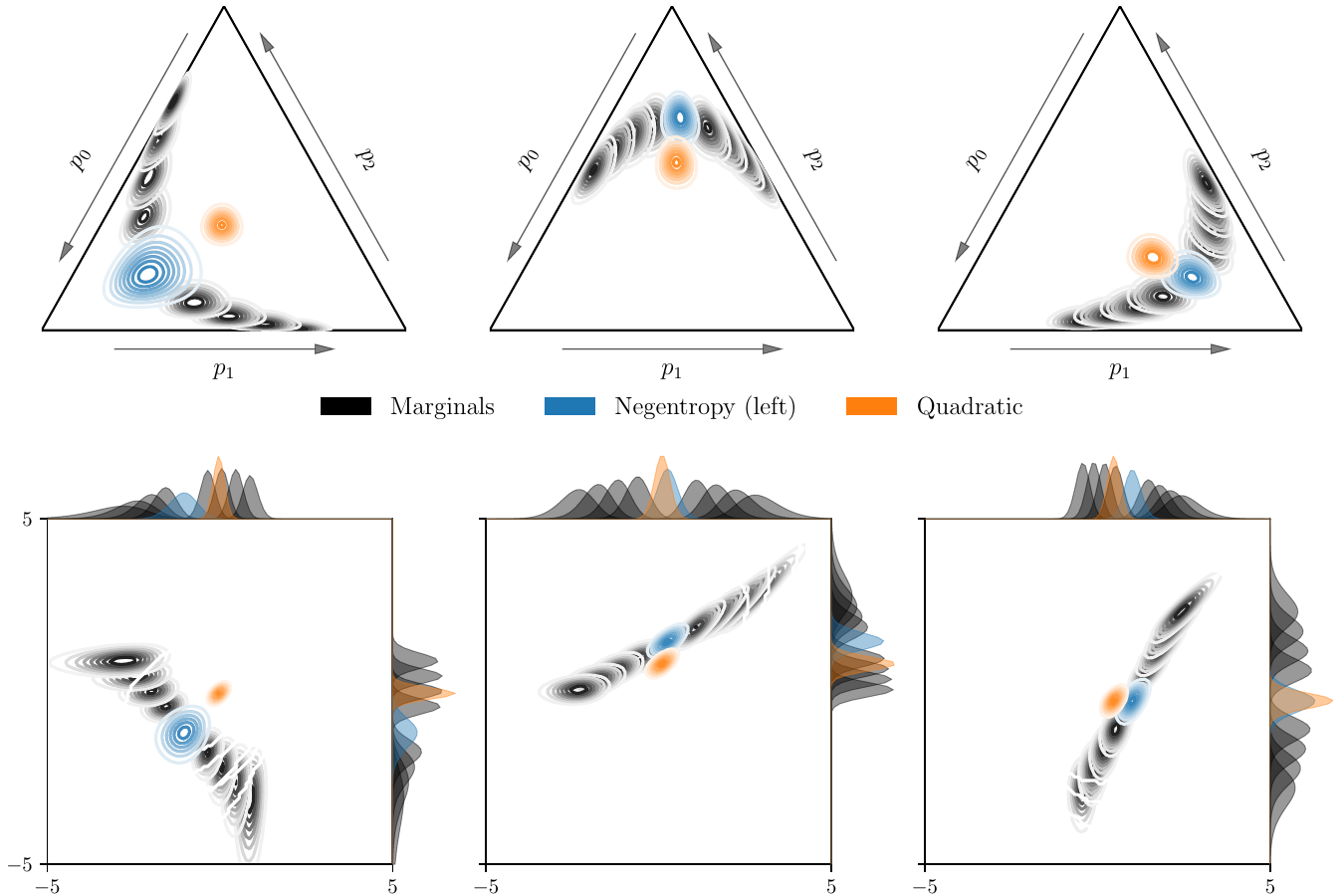}
  \caption{\footnotesize Comparing the (left) Bregman-Wasserstein barycenter (blue) and $2$-Wasserstein barycenter (orange) on the simplex, for the given marginals (gray). The bottom row illustrates the distributions on the primal domain $\X = \R^2$.}
\label{fig:bar_simplex}
\end{figure}

\subsubsection{Application in Bayesian statistics} \label{sec:Bayesian}
In \cite{backhoff-veraguas2022blw}, the Wasserstein barycenters \eqref{eqn:Wasserstein.barycenter} is analyzed through the lens of statistical decision theory (also see \cite{srivastava2018sbv}). Under such a perspective, it is possible to
interpret the barycenter as a {\it Bayes estimator}, i.e.~a distribution
minimizing the Bayes risk. We explain the general setup in \cite{backhoff-veraguas2022blw} and then adapt it to the Bregman-Wasserstein divergence.

Consider a manifold $\mathcal{M}$ (e.g.~$\mathbb{R}^d$) representing the state space of the data. In probabilistic modeling, a {\it model space} $\Ms\subseteq \sop(\M)$ can be thought of as a subset of the set of probability distributions. Under a Bayesian framework, given a dataset $\mathcal{D} = \left\{p_1, \ldots, p_n\right\}$, a {\it prior distribution} $P \in \sop(\sop(\M))$ induces a {\it posterior distribution} $P_{\mathcal{D}} \in \sop(\sop(\M))$ over the space of distributions. Given a {\it loss function} $\mathcal{L}: \sop(\M) \times \sop(\M) \to \mathbb{R}_{+}$ measuring the divergence between two models, the {\it Bayes risk} (expected loss under the posterior distribution) for a given model $\mu\in \Ms$, and the Bayes estimator are defined as
\begin{equation} \label{eqn:Bayes.estimator}
  \risk_{\mathcal{L}}(\mu\vert \D) \defeq \int_{\sop(\M)} \mathcal{L}(\mu, \nu) \dd P_{\D}(\nu) \quad \text{and} \quad
  \bar{\nu}_{\L} \defeq \argmin_{\nu \in \Ms} \risk_{\L}(\nu \vert \D).
\end{equation}
This general framework, which operates directly with distributions rather than parameters, can be used to recover classically known Bayesian
estimators, while also allowing for novel and interesting estimators by
considering more exotic loss functions. For instance, the maximum a posteriori (MAP) estimator is the Bayes estimator for the binary loss $\L(\mu,  \nu) = \indic{\mu \neq\nu}$, and the Bayes model average is optimal for both the mean-squared error and KL-divergence loss functions. Although classically used $f$-divergence loss functions are
popular due to their strong invariance properties, OT-based loss
functions capture the geometry of the underlying sample space $\M$ by computing a ``horizontal'' displacement cost (as opposed to a ``vertical'' interpolation that only depends on the density, and not the location of the support). Now, if the loss function $\L = \Wt^2$ is the squared $2$-Wasserstein distance, then $\bar{\nu}_{\L}$ is given by the usual (population) Wasaserstein barycenter.

While previous works~\cite{backhoff-veraguas2022blw,srivastava2018sbv}
mainly explored the symmetric Wasserstein loss,\footnote{Note that this requires specifying a metric on $\M$, such as the Euclidean metric with respect to a chosen coordinate system.} this may not always be the ideal choice. For example, in the context of an exponential family (Section \ref{sec:exp.family}), the Bregman divergence of the dual potential on the dual domain (which corresponds to the KL-divergence on the family, see \eqref{eqn:Bregman.as.KL}) is arguably more natural than the Euclidean distance. %
If we let $\L = \BWD$ be the Bregman-Wasserstein divergence corresponding to \eqref{eqn:Bregman.as.KL}, then the Bayes estimator is the Bregman-Wasserstein barycenter. Here we give an idealized example where we restrict ourselves to {\it Gaussian mixtures} that can be regarded as distributions on the Gaussian Bregman manifold. The Bregman-Wassestein barycenter is visibly and meaningfully different from other barycenters. The details can be found in Appendix  \ref{sec:details.barycenter}.

\begin{figure}[t!]
	\centering
	\hspace{-0.55cm}
	\includegraphics[scale = 0.5]{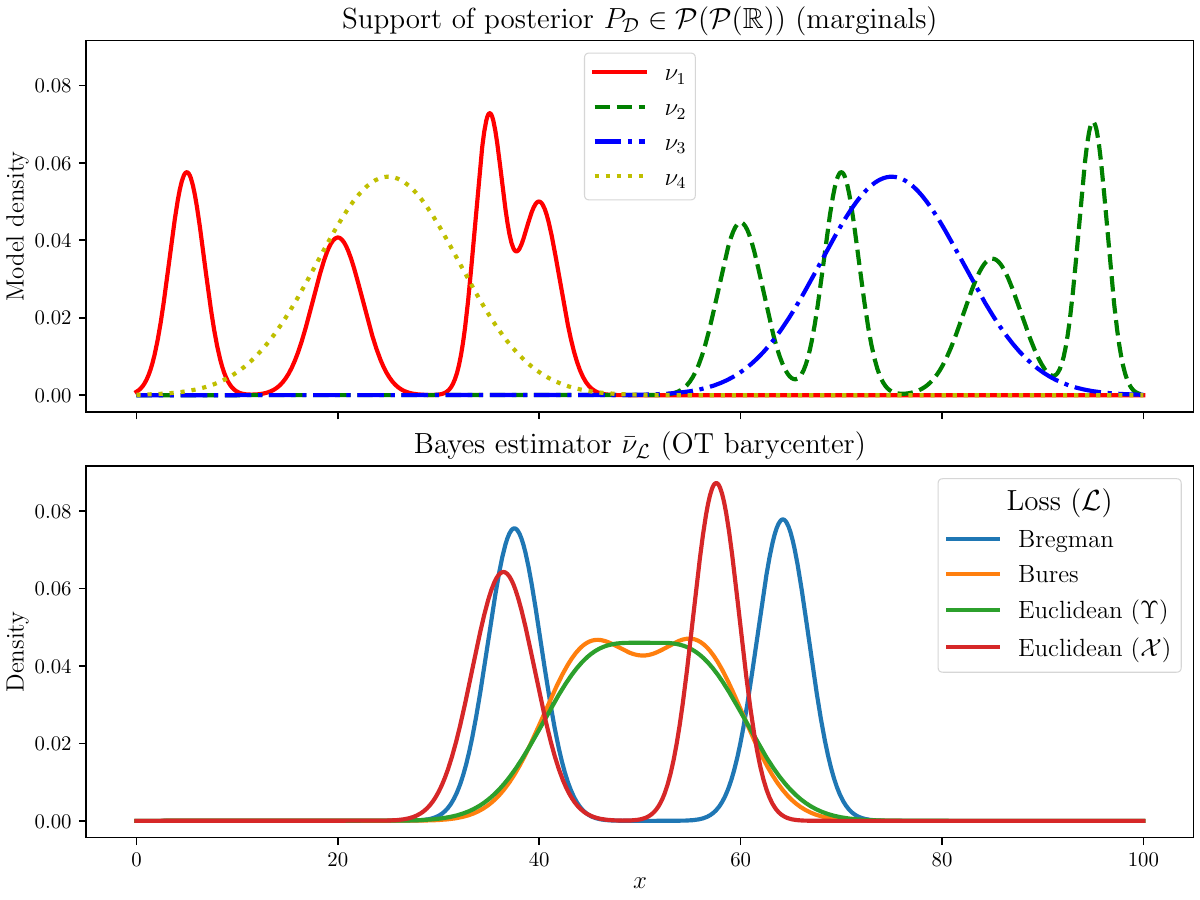}
    \vspace{-0.2cm}
	\caption{\footnotesize Illustration of the posterior (top row) and barycenters (second row) in the context of Example \ref{eg:GM.barycenters}.}
  \label{fig:GMM}
\end{figure}

\begin{example}[Gaussian mixtures] \label{eg:GM.barycenters}
Gaussian mixtures are ubiquitous in statistics and machine learning, since they lead to tractable formulations and are universal approximators of arbitrary distributions~\cite[Chapter~3]{goodfellow2016dl}. For concreteness, we consider Gaussian mixtures on $\mathbb{R}$. It is well known that the space of univariate Gaussian distributions is a $2$-dimensional exponential family, and we take this to be our Bregman manifold $\mathcal{M}$. A Gaussian mixture can then be identified by a probability distribution on $\mathcal{M}$. Given two Gaussian mixtures corresponding to $\mu, \nu \in \mathcal{P}(\mathcal{M})$, we take the Bregman loss function (on the subspace of Gaussian mixtures) to be $\mathscr{B}(\mu, \nu)$, as in equation~\eqref{eqn:BW.mixture.exp.family} of Example \ref{eg:BW.exp.family}.

Assume that given a dataset $\D$, the posterior distribution $P_{\D} \in
\mathcal{P}(\sop(\mathbb{R}))$ is the uniform distribution over four Gaussian
mixtures $\nu_i$, $i \in \left\{1, 2, 3, 4\right\}$, visualized in the top row
of Figure \ref{fig:GMM}. We consider each $\nu_i$ as a discrete distribution
over the space $\Upsilon := \mathbb{R} \times (0, \infty)$ of the mean and
variance parameters, which in turn can be identified with $\M$. We let the model
space $\Ms$ be Gaussian mixtures with two components. We equip $\M$ with four
costs: (i) the squared Euclidean distance on $\Upsilon$, (ii) the squared Euclidean distance on
the natural parameter space $\X$, (iii) the Bures-Wasserstein distance~\cite{BJL19}, and (iv) the KL-divergence~\cite[Section~9]{duchi2007dla}, which recovers the Bregman loss as
stated above, and in Equation~\eqref{eqn:BW.mixture.exp.family}. In the second
row of Figure \ref{fig:GMM}, we compare the Bregman-Wasserstein barycenter with
the OT barycenters induced by the Bures-Wasserstein and Euclidean distances. We
observe that the Bregman-Wasserstein barycenter is truly bimodal, while the
Bures/Euclidean Wasserstein barycenters are virtually unimodal. While using the
Euclidean cost in the natural parameter space also yields bi-modality, the two
components have unequal variance and tilt slightly towards the left, whereas the
Bregman is symmetric with similar variance for both components.
\end{example}

\subsection{Riemannian Wasserstein gradient flows} \label{sec:Wasserstein.gradient.flow}
In this subsection, we outline the application of the Bregman--Wasserstein
divergence to the approximation of Wasserstein gradient flows. Here we explain
the key results one can obtain; they are proved in \cite{RW24}.

Let $(\M,g)$ be a Riemannian manifold, which we will eventually take to be a
Bregman manifold. A function $\rho : [0,\infty) \times \M \rightarrow \mathbf{R}$ solves the
(Riemannian) Fokker--Planck equation with initial condition $\rho_0 \in \mathcal{P}(\M)$ and potential $\psi$ provided
\begin{equation}
\label{eq:fp}  \left\{
    \begin{array}{rll}
      \partial_t\rho&=  \Delta\rho + \divg(\rho \grad \psi)&\text{in}\quad(0,\infty) \times \M;\\
      \rho(0,\cdot) &= \rho_0(\cdot) &\text{on}\quad \M.
    \end{array}
  \right.
\end{equation}
The JKO scheme, introduced in the seminal work of Jordan, Kinderlehrer, and Otto
\cite{JKO98}, is an approximation scheme for the equation which reflects that
\eqref{eq:fp} is a gradient flow in the Wasserstein space. The scheme is as
follows, given $\rho_0$, recursively compute
\begin{equation}
  \label{eq:jko-def}\tag{JKO}
  \rho^\tau_{k+1} := \text{argmin}_{\rho \in \mathcal{P}(\M)} \int_{\M} \log \rho \, \dd \rho + \int_{\M}\psi \, d\rho
  + \frac{1}{2\tau}\mathscr{W}^2_2(\rho,\rho^\tau_{k}),
\end{equation}
and define the piecewise constant interpolation
\begin{equation}
  \label{eq:interpolant-def}
  \rho^\tau(t) := \rho^\tau_k \quad \text{for } t\in((k-1)\tau,k\tau].
\end{equation}
Jordan, Kinerlehrer, and Otto's key conclusion is that for each $t \in [0,\infty]$, $\rho^{\tau}(t)$ converges weakly in $L^{1}(\M)$ to a probability density $\rho(t,\cdot)$, which solves \eqref{eq:fp}. The significance of this result and its effect on optimal transport cannot be understated. However, the original work \cite{JKO98} took place on Euclidean space, and there the distance squared term defining the Wasserstein distance is the straightforward-to-compute Euclidean distance. When one extends Jordan, Kinderlehrer, and Otto's results to manifolds they are conceptually unchanged, practically, however, one faces the difficulty that explicit, tractable, expressions for the Riemannian geodesic distance may not exist depending on the Riemannian metric.

This is the case if we now consider $(\M,g)$ as a Bregman manifold.
A straightforward-to-compute expression for the Riemannian distance between two
points does not exist in general. Of course, we do have a simple expression for
the Bregman divergence. Thus, one may ask if solutions to the Fokker-Planck
equation may be approximated by a JKO-like scheme where the term $\mathscr{W}^2_2(\rho,\rho^\tau_{k})$ in \eqref{eq:jko-def} is replaced with the Bregman-Wasserstein divergence $\mathscr{B}(\rho,\rho^{\tau}_{k})$. That this is true is a related work \cite{RW24} by the second and third authors. More precisely, there we prove the following theorem.

\begin{theorem} \label{thm:modified-jko}
Let $\Omega$ be a regular Bregman generator on $\X$ and let $\M$ be the associated Bregman manifold with divergence ${\bf B}$. Assume either:
  \begin{enumerate}
  \item[(i)] $\tilde{\M}$ is some open and precompact subset of $\M$ with smooth boundary; or
  \item[(ii)] $\tilde{\M} = \M$ and $(\M,g)$, where $g_{ij}(x) = D_{ij}\Omega(x)$ under the Euclidean coordinates on $\X$, is a complete Riemannian manifold with Ricci curvature bounded below and there are constants $\lambda,\Lambda > 0$ such that $\lambda d^2(x_0,x_1) \leq B_{\Omega}(x_0,x_1) \leq \Lambda d^2(x_0,x_1)$ for $x_0, x_1 \in \X$.
  \end{enumerate}
  Finally, assume $\psi: \M \rightarrow \mathbf{R}_+$ is a smooth function satisfying $\Vert \nabla \psi \Vert \leq C(1+\psi)$. Define $\rho^{\tau}$ by \eqref{eq:interpolant-def} where $\rho_0$ is given and
\begin{equation}
  \label{eq:bw-jko-def}\tag{BW-JKO}
  \rho^\tau_{k+1} := \text{argmin}_{\rho \in \mathcal{P}(\M)} \int_{\M} \log \rho \, \dd \rho + \int_{\M}\psi \, d\rho
  + \frac{1}{\tau}\mathscr{B}(\rho,\rho^{\tau}_{k}).
\end{equation}
   Then there exists a measurable function $\rho: [0,\infty) \times \M \rightarrow\mathbf{R}_+$ such that for each $t \in [0,\infty)$ we have $\rho^\tau(t) \rightharpoonup \rho(t)$ weakly in $L^1(\M ; \dd \vol)$, and $\rho$ solves the Fokker-Planck equation \eqref{eq:fp}.
Moreover, in case (i) $\rho$ is the unique classical solution satisfying the Neumann boundary condition, and in case (ii) $\rho$ is the unique classical solution with finite drift (by which we mean $\int_\M \psi \dd \rho(t) < \infty$) and second moment.
\end{theorem}

\section{Conclusion and future directions} \label{sec:conclusion}
In this paper, we have shown that the Bregman-Wasserstein divergence induces a rich geometric structure on the space of probability measures, has numerous interesting properties, and can be readily implemented using existing OT algorithms. Additionally, we studied applications to Bregman-Wasserstein barycenters and modified JKO schemes for Wasserstein gradient flows. Our results suggest many directions for future research, some of which are highlighted below.
\begin{itemize}
\item Deeper study of the Bregman-Wasserstein geometry. In this paper, we follow the formal approach of Otto \cite{O01} and Lott \cite{L08} by restricting to sufficiently regular probability distributions and vector fields. An interesting direction of research is to develop our constructions in the generality of the work of Gigli \cite{Gigli12}, and study further properties of the parallel transports and curvature. We also believe that analogous theories can be developed when the cost function is a tractable divergence such as the {\it logarithmic divergence} studied in \cite{PW16, PW18, PW18b, W18, WZ22}.
\item Implementation of the Bregman-Wasserstein JKO scheme \eqref{eq:bw-jko-def}. A possible approach is to use {\it input-convex neural networks} \cite{amos2017input} which has been applied to implement \eqref{eq:jko-def} in the Euclidean case~\cite{mokrov2021large}. Other approaches, such as {\it normalizing flows} \cite{xu2023nfn}, may also be considered.  More generally, it is worthwhile to explore applications of the Bregman-Wasserstein divergence in optimization over spaces of probability distributions, especially generalizations of the mirror descent to the Wasserstein space \cite{bonet2024mirror, DKPS23}.
\item {\it Distributionally robust optimization (DRO)}. Consider a stochastic program of the form
\begin{equation} \label{eqn:stochastic.program}
\inf_{a \in A} \int_{\M} L(p, a) \dd \widehat{\mu}(p),
\end{equation}
where $\widehat{\mu} \in \mathcal{P}(\M)$ is a given reference distribution (typically estimated from data) on a state space $\M$, and $L$ is a loss function depending on the realization $p \in \M$ and the chosen action $a \in A$. The DRO ``recipe'' \cite{DSW24} {\it robustifies} \eqref{eqn:stochastic.program}  to the minimax problem
\begin{equation} \label{eqn:DRO}
\inf_{a \in A} \sup_{\mu \in \mathcal{N}(\widehat{\mu})} \int_{\M} L(p, a) \dd \mu(p),
\end{equation}
where $\mathcal{N}(\widehat{\mu})$ is a suitably defined {\it ambiguity set}, that is, a distributional neighborhood about $\widehat{\mu}$. The key idea is that $\mathcal{N}(\widehat{\mu})$ captures possible estimation errors or adversarial perturbations. In \eqref{eqn:DRO}, we aim to find an action which minimizes the worst-case expected loss. By considering \eqref{eqn:DRO} instead of \eqref{eqn:stochastic.program}, one obtains a more robust solution, which may possess desirable properties such as reduced sensitivity to outliers and robustness to adversarial perturbations. Existing OT-based ambiguity sets typically use  the Wasserstein distance~\cite{blanchet2019qdm,pflug2007aps}. Given the flexibility of the Bregman-Wasserstein divergence, it is natural to consider ambiguity sets of the form
   \begin{equation} \label{eqn:BW.ball}
     \mathcal{N}(\widehat{\mu}) = \left\{\mu \in \sop(\M)
       \ : \ \mathscr{B}(\mu,\widehat{\mu}) \leq \epsilon \right\}, \quad \epsilon > 0.
   \end{equation}
In fact, this was the original motivation of \cite{GHY17} to consider the Bregman-Wasserstein divergence. Also, see \cite{PVYY24} for a recent application of \eqref{eqn:BW.ball} in expected utility maximization. An interesting avenue is to validate the use of Bregman-Wasserstein ambiguity sets for applications such as portfolio optimization, adversarially robust classification and more. In particular, it may be advantageous to consider applications where penalizing    deviations asymmetrically is desirable.
\end{itemize}

\appendix
\section*{Appendix}
\section{Implementation details} \label{sec:implementation.detail}

\subsection{Neural OT} \label{sec:detail.neural.OT}
We present some additional results in Figure~\ref{fig:pdi_demo} for another
synthetic dataset in two dimensions. For both Figures~\ref{fig:ddi_rout_2}
and~\ref{fig:pdi_demo}, we show the Brenier mapping and displacement
interpolations between $1024$ randomly sampled points from $\mu_0$ and $\mu_1$
(in the specified coordinates).

When estimating the primal displacement interpolation, we run the ENOT solver to solve
$\Wt^2(\mu_0^\Y, \mu_1^\X)$, and estimate the neural Brenier maps $Dh$ and $Dh^* = (Dh)^{-1}$. Similarly,
for the dual displacement interpolation, we instead solve the $\Wt^2(\mu_0^\X, \mu_1^\Y)$ transport
to estimate the Brenier maps $Df$ and $Df^* = (Df)^{-1}$. The Brenier maps
in the primal displacement interpolation can be used to construct the OT map (and its inverse) for
$\mathscr{B}(\mu_1^\Y, \mu_0^\Y)$ as in Equation~\eqref{eqn:dual.OT.map}, while the dual displacement
interpolation is related to the
OT map for $\mathscr{B}(\mu_0^\X, \mu_1^\X$) in Equation~\eqref{eqn:primal.OT.map}.
Importantly, these OT maps are not inverses of each other, except in the special self-dual case
of the quadratic potential $\Omega = \Omega^* = |\cdot|^2/2$. In our implementation, we use
the same neural network configuration and hyperparameters as in~\cite[Table~5]{buzun2024erf}.

\begin{figure}[]
  \begin{subfigure}[t]{0.45\textwidth}
  \begin{center}
    \includegraphics[width=\linewidth]{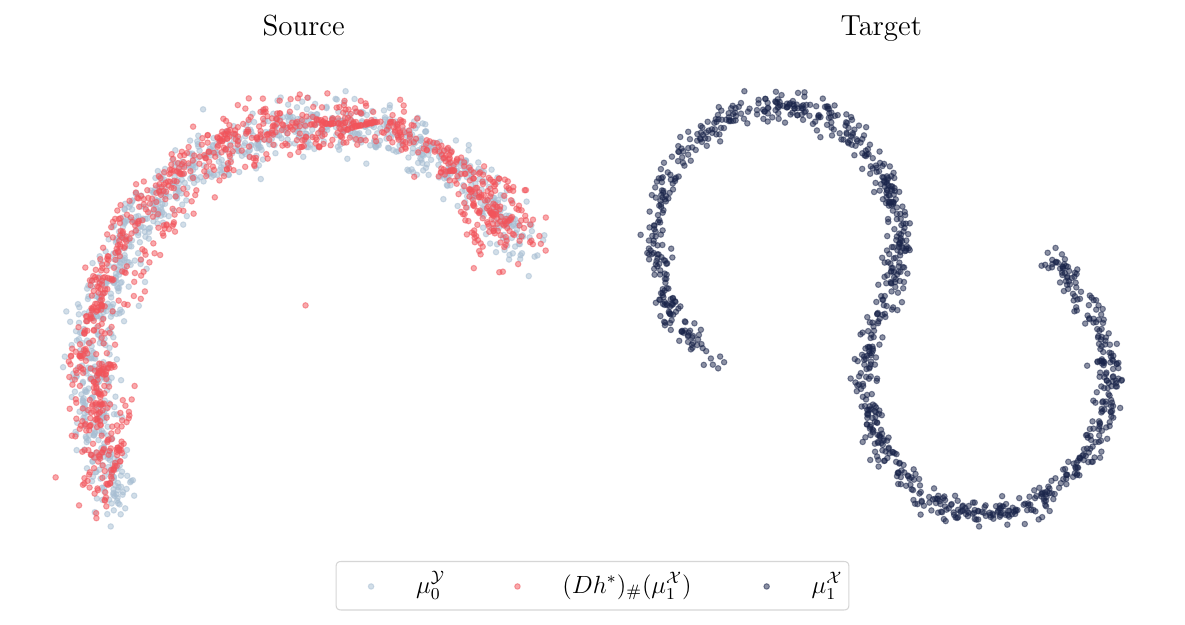}
  \end{center}
  \caption{$\Omega_i(x^i) = \frac{1}{2}(x^i)^2$}
  \label{fig:pdi_demo_Euclidean}
  \end{subfigure}
  \begin{subfigure}[t]{0.45\textwidth}
  \begin{center}
  \includegraphics[width=\linewidth]{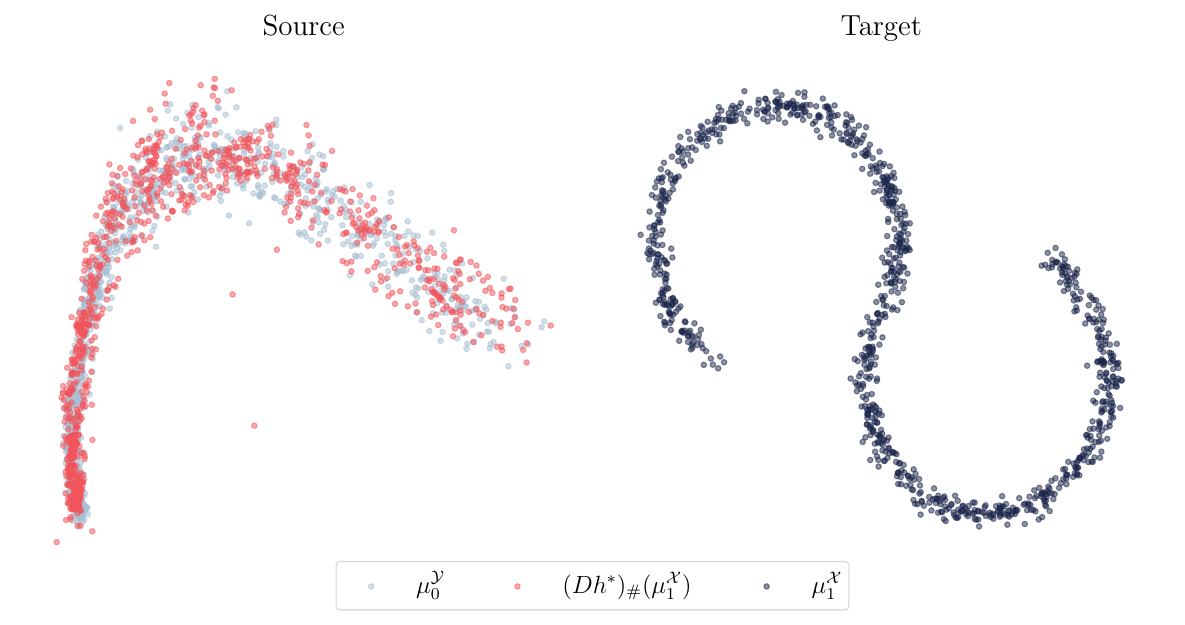}
  \end{center}
  \caption{$\Omega_i(x^i) = \exp(x^i)$}
  \label{fig:pdi_demo_ConjugateExtendedKL__a_1.0_}
  \end{subfigure}
  \begin{subfigure}[t]{0.45\textwidth}
  \begin{center}
    \includegraphics[width=\linewidth]{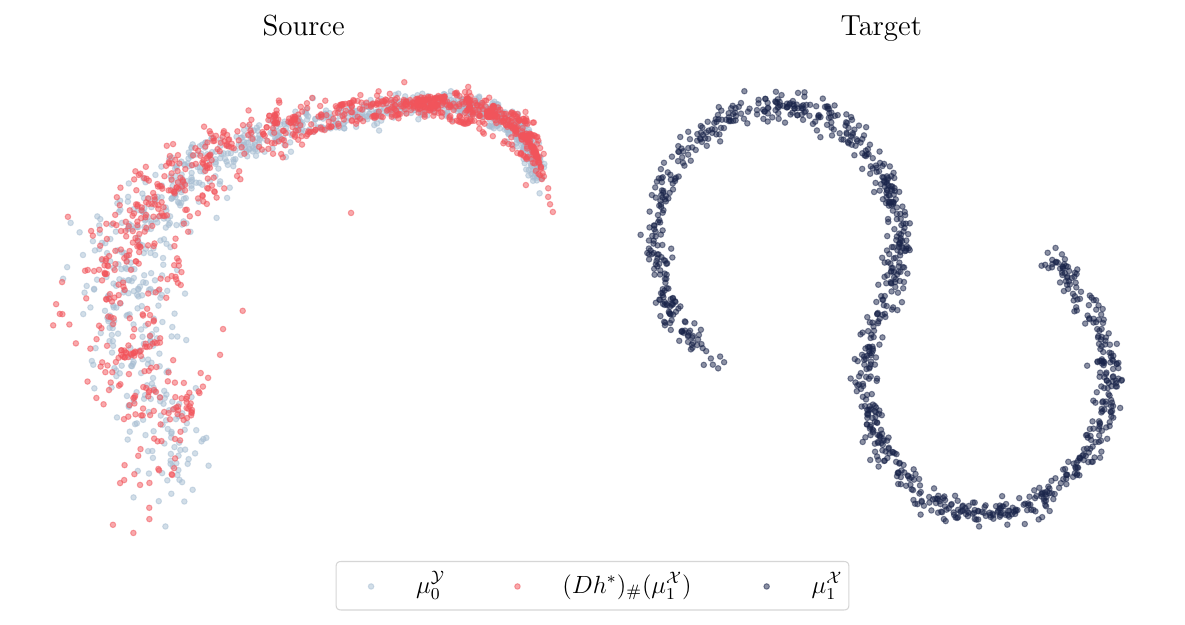}
  \end{center}
  \caption{$\Omega_i(x^i) = \exp(-x^i)$}
  \label{fig:pdi_demo_ConjugateExtendedKL__a_-1.0_}
  \end{subfigure}
  \begin{subfigure}[t]{0.45\textwidth}
  \begin{center}
    \includegraphics[width=\linewidth]{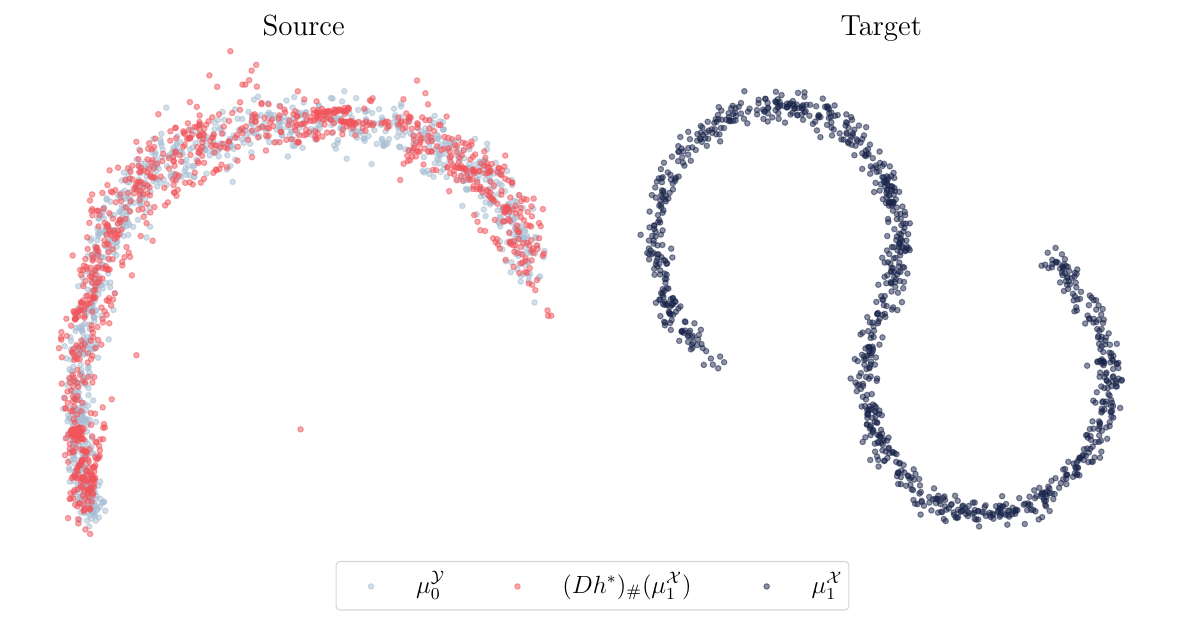}
  \end{center}
  \caption{$\Omega_i(x^i) = \frac{1}{2}\log\left( 1 + \exp(2x^i) \right) $}
  \label{fig:pdi_demo_ConjugateHNNTanh__beta_1.0_}
  \end{subfigure}
  \begin{subfigure}[t]{0.9\textwidth}
  \begin{center}
    \includegraphics[width=\linewidth]{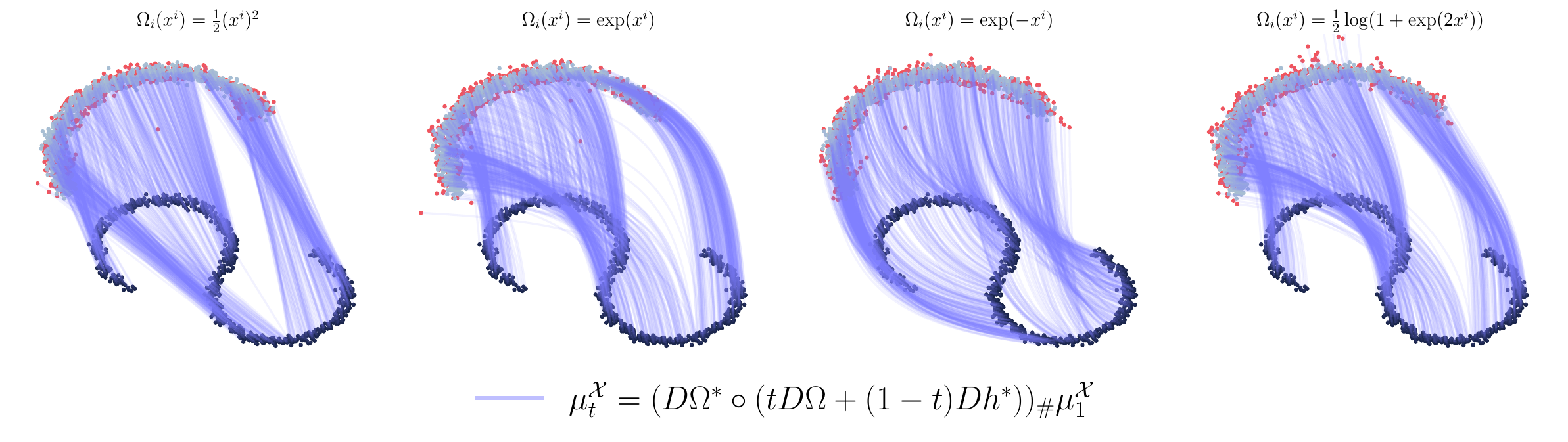}
  \end{center}
  \caption{Time-reversed primal diplacement interpolations.}
  \label{fig:pdi_demo_paths}
  \end{subfigure}
  \begin{subfigure}[t]{0.9\textwidth}
  \begin{center}
    \includegraphics[width=\linewidth]{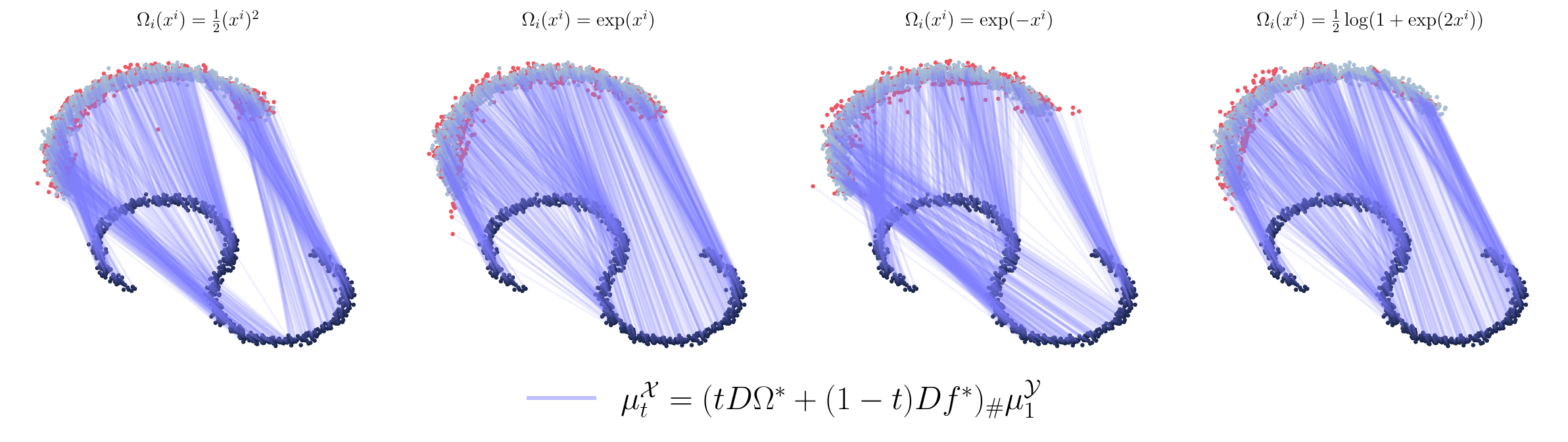}
  \end{center}
  \caption{Time-reversed dual diplacement interpolations.}
  \label{fig:ddi_demo_paths}
  \end{subfigure}
  \caption{\footnotesize Bregman-Wasserstein OT for the synthetic dataset considered
  in~\cite{UC23}.
Figures~\ref{fig:pdi_demo_Euclidean}-\ref{fig:pdi_demo_ConjugateHNNTanh__beta_1.0_}
plot the neural Brenier map $Dh^*$ between $\mu_1^\mathcal{X}$ and
$\mu_0^\mathcal{Y}$ for various choices of the separable Bregman generator
$\Omega(x) \defeq \sum_{i = 1}^d \Omega_i(x^i)$. 
Note that time reversal of a primal displacement interpolation is equivalent to
a dual displacement interpolation in the opposite direction, and vice-versa.}
  \label{fig:pdi_demo}
\end{figure}

\subsection{Bregman-Wasserstein barycenter} \label{sec:details.barycenter}
Given its well-posedness (established in Theorem \ref{thm:BW.barycenter.well.posed}), the Bregman-Wasserstein barycenter \eqref{eqn:Bregman.right.barycenter} can be readily computed by adapting existing methods. 
We use the {\it free-support method}~\cite{cuturi2014fcw} implemented in the~\texttt{OTT} Python package~\cite{cuturi2022ott} using the~\texttt{JAX} library~\cite{bradbury2018jct}. This method incorporates entropically regularized OT (for discrete measures) that can be solved efficiently by the Sinkhorn algorithm and can be parallelized to exploit GPU hardware. The algorithm proceeds by repeatedly iterating over two steps and works for any cost function $c$:
\begin{itemize}
  \item Computing the entropic transport plan using the Sinkhorn algorithm for the
    current estimate of the barycenter and each marginal (estimating $K$ such
    plans for $K$ marginals). The entropic OT map can then be approximated as
    the {\it barycentric projection}~\cite{deb2021reo} (with respect to cost $c$) of the entropic OT plan.
  \item The next barycenter estimate is obtained by applying the {\it barycentric map}
    to the transported particles obtained in the previous step. 
\end{itemize}
The barycentric projection relies on the barycentric maps for the underlying cost. From \eqref{eqn:Bregman.right.barycenter}, for the Bregman cost the left and right barycentric maps are respectively given by
    \begin{equation*}
    \begin{split}
        \bar{x}_{\lambda}^{\leftarrow} &\defeq \argmin_{x}  \sum_{i=1}^K \lambda_i \mathbf{B}_\Omega(x, x_i) = D\Omega^*\left(\sum_{i=1}^K \lambda_i D\Omega(x_i)\right) \\
        \bar{x}_{\lambda}^{\rightarrow} &\defeq \argmin_{x} \sum_{i=1}^K \lambda_i \mathbf{B}_\Omega(x_i, x) = \sum_{i=1}^K \lambda_i x_i
    \end{split}
    \end{equation*}
We use an entropic penalty of $10^{-2}$ for estimating the barycenter, and run
the Sinkhorn algorithm for a maximum of $50$ iterations.

For
Figure~\ref{fig:bar_simplex}, we compute the barycenter of $8$ marginal
distributions, each approximated using $250$ particles. Each column of the
figure corresponds to a different set of marginals that are constructed along
a primal geodesic (curved on the simplex) between the endpoints. More precisely, each marginal
is taken to be a Dirichlet distribution centered at evenly spaced points along
the geodesic, i.e. $\mu_i = \mathrm{Dir}\left( 100 p_i \right)$, where each $p_i$ is a point on the open unit simplex in $\mathbb{R}^3$. The barycenter is initialized by sampling from the flat Dirichlet with unit concentration. Additionally, the weights $\lambda = (\lambda_1, \ldots, \lambda_8)$ are skewed to give more weight to the marginals at the endpoints of the geodesic. Specifically, the weights are
constructed by normalizing the last row of Pascal's triangle:
 \begin{equation*}
   \lambda = \begin{bmatrix} \frac{35}{128} & \frac{21}{128} &\frac{7}{128}
   & \frac{1}{128}  & \frac{1}{128} & \frac{7}{128} & \frac{21}{128}
   & \frac{35}{128} \end{bmatrix}.
 \end{equation*}
 The contours for the resulting barycenter in each case is approximated by
 a Dirichlet distribution (on the simplex) centered at the mean of the particles, and concentration determined by the variance.

 For Figure~\ref{fig:GMM}, we assume that the posterior distribution $P_{\D}\in
 \sop(\sop(\mathbb{R}))$ is represented by an empirical measure supported uniformly on four ``models''
 $\left\{\nu_i\right\}_{1 \leq i \leq 4}$, and each (Gaussian mixture) model
 $\mu$ is composed of its Gaussian ``particles'' (components) that are identified by their mean and variance
 \begin{equation*}
   \upsilon \defeq (m, \sigma^2) \in \Upsilon \defeq \mathbb{R} \times (0, \infty).
 \end{equation*}
 Under this identification, we consider marginals given by
 \begin{equation*}
   \begin{split}
     \nu_1 &= \frac{1}{4}\left( \delta_{(5, 3)} + \delta_{(20, 6)}
     + \delta_{(35, 2)} + \delta_{(40, 4)} \right),\\
     \nu_2 &= \frac{1}{4}\left( \delta_{(85, 8)} + \delta_{(60, 5)}
       + \delta_{(70, 3)} + \delta_{(95, 2)} \right),\\
       \nu_3 &= \delta_{(75, 50)},\\
       \nu_4 &= \delta_{(25, 50)}.
   \end{split}
 \end{equation*}
 The barycenter is constrained to be a Gaussian mixture of two components (this defines the model class $\mathbbm{M}$), which is initialized by sampling ${m_i \sim \mathrm{Uniform}(0, 100)}$ and ${\sigma_i^2 \sim \mathrm{Uniform}(0.5, 5)}$.%

 We consider OT barycenters under four cost functions on $\Upsilon$. The ``Euclidean'' (in green) refers to the usual squared Euclidean cost in the mean-variance half-space, i.e.
 \begin{equation}
 \begin{split}
 c(\upsilon_1, \upsilon_2) &= (m_1 - m_2)^2 + (\sigma_1^2 - \sigma_2^2)^2,\quad (\text{on } \Upsilon)\\
 \end{split}
 \end{equation}
 The ``Bures'' (in orange) refers to the {\it Bures-Wasserstein distance} (see e.g.~\cite{BJL19}) between univariate Gaussians, and is given by
 \begin{equation*}
   c(\upsilon_1, \upsilon_2) = (m_1 - m_2)^2 + (\sigma_1 - \sigma_2)^2.
 \end{equation*}
 This is the same as the Euclidean cost in the mean-standard deviation space.
 The Bregman cost is induced by the log-partition function of the
 univariate Gaussian family defined on the natural parameter space $\Theta$ (see for instance \cite[Section 2.2.1]{A16}):
 \begin{equation*}
   \Omega: \Theta \defeq \mathbb{R} \times (-\infty, 0) \to \mathbb{R}, \quad \theta \mapsto
   \frac{\theta_1^2}{4\theta_2^2} - \frac{1}{2}\log\left( -2\theta_2 \right).
 \end{equation*}
 The natural parameters are related to the mean and variance by
 \begin{equation*}
   \theta(\upsilon) = \theta(m, \sigma^2) := \left( \frac{m}{\sigma^2},
   -\frac{1}{2\sigma^2}\right).
 \end{equation*}
 Therefore, by the ``Bregman'' cost, we really mean the cost function
 \begin{equation*}
   c(\upsilon_1, \upsilon_2) = \mathbf{B}_\Omega(\theta(\upsilon_1),
   \theta(\upsilon_2)) = \mathbf{H}(\mathcal{N}(m_2,\sigma_2^2) || \mathcal{N}(m_1,\sigma_1^2)).
 \end{equation*}
 Similarly, the Euclidean cost on the natural parameter space $\Theta$ is
 \begin{equation*}
 \widetilde{c}(\upsilon_1, \upsilon_2) = |\theta(\upsilon_1) - \theta(\upsilon_2)|^2
 \end{equation*}
 We compute and plot its right barycenter in Figure~\ref{fig:GMM}, which
 intuitively corresponds to averaging particles in the natural parameter space.

\section{Existence and characterization of Bregman-Wasserstein barycenter} \label{sec:BW.barycenter.proof}
In this appendix we provide the proof of Theorem \ref{thm:BW.barycenter.well.posed} by adapting and extending the results and techniques of Agueh and Carlier \cite{AC11} and Pass \cite{P11}.

We begin by recalling the Monge and Kantorovich formulations of multi-marginal optimal transport.  Given Polish spaces $\M_1, \dots, \M_K$, a Borel cost function $c:\M_1 \times\dots \times \M_K \rightarrow \mathbb{R}$, and measures $\nu_i \in \mathcal{P}(\M_i)$, the Monge problem is to minimize
\[
\mathcal{C}(G_2,\dots,G_K) = \int_{\M_1}c(x_1,G_2(x_1),\dots,G_K(x_1)) \dd \nu_1(x_1),
\]
over mappings $(G_2, \ldots, G_K): \M_1 \rightarrow \M_2 \times \cdots \times \M_K$ satisfying $(G_i)_{\#}\nu_1 = \nu_i$. The Kantorovich formulation is to minimize over plans: Let $\Pi(\nu_1,\dots,\nu_K)$ denote the subset of $\mathcal{P}(\M_1 \times \dots \times \M_K)$ whose $i^{\text{th}}$ marginal is $\nu_i$.
The Kantorovich problem  is to find $\pi \in \Pi(\nu_1 \times \dots \times \nu_K)$ minimizing
\[
\mathcal{K}(\pi) = \int_{\M_1 \times \dots \times \M_K}c(x_1,\dots,x_K) \dd \pi(x_1,\dots,x_K).
\]

By \eqref{eqn:BW.divergence.coordinates}, we may work under the primal representation (on $\X$) which is more notationally convenient because of \eqref{eqn:Bregman.right.barycenter}. In the following we fix positive weights $\lambda_1, \ldots, \lambda_K$ with $\sum_{i = 1}^K \lambda_i = 1$. Let $T = T_{\lambda}: \X^K \rightarrow \X$ be the barycentric map $T(x) = \sum_{i = 1}^K \lambda_i x_i$, where $x = (x_1, \ldots, x_K) \in \X^K$.

\begin{proposition} [Multi-marginal formulation]\label{prop:multimarginal-equiv}
Let $\nu_1,\dots,\nu_K \in \mathcal{P}(\M)$ be absolutely continuous with respect to the Riemannian volume $\dd \vol$.\footnote{This is the same as saying that $\nu_i^\X$ (resp.~$\nu_i^Y$) is absolutely continuous with respect to the Lebesgue measure on $\X$ (resp.~$Y$).} If $\gamma \in \Pi(\nu_1, \ldots, \nu_K)$ is a minimizer of the multimarginal Kantorovich problem
 \begin{equation} \label{eqn:BW.multimarginal}
 \inf_{\pi \in \Pi(\nu_1, \ldots, \nu_K) } \left\{ \mathcal{K}(\pi) := \int_{{\X}^K} \sum_{i=1}^K \lambda_i {\bf B}_{\Omega}(x_i,T(x)) \dd \pi^{\X}(x) \right\} < +\infty,
 \end{equation}
then $\overline{\nu} = (\iota^{-1} \circ T)_{\#} \gamma^{\X} \in \mathcal{P}(\M)$ is a Bregman-Wasserstein barycenter, i.e., a solution to \eqref{eq:breg-bary}. Moreover, the problems \eqref{eq:breg-bary} and \eqref{eqn:BW.multimarginal} attain the same minimizing value.
\end{proposition}
\begin{proof}
 We follow closely the argument in \cite[Proposition 4.2]{AC11}. We first establish
 \begin{equation}  \label{eq:bary-est1}
\inf_{\tilde{\nu} \in \mathcal{P}(\M)}\sum_{i=1}^K\lambda_i\mathscr{B}(\nu_i, \tilde{\nu}) \geq \mathcal{K}(\gamma),
 \end{equation}
  after which we show
  \begin{equation}
    \label{eq:bary-est2}
    \sum_{i=1}^K \lambda_i \mathscr{B}(\nu_i, \overline{\nu}) \leq \mathcal{K}(\gamma).
 \end{equation}
 Combining \eqref{eq:bary-est1} and \eqref{eq:bary-est2} yields the result.

 For \eqref{eq:bary-est1} we begin with some notation for disintegrating families of measures. Let $\tilde{\nu}^\X \in \mathcal{P}(\X)$ and $\tilde{\nu}^\X_i \in \Pi(\nu^\X_i,\tilde{\nu}^\X)$ be arbitrary. By the disintegration theorem \cite[Theorem 1.4.10]{FG21} each $\tilde{\nu}^\X_i$ may be decomposed as
 \[ \tilde{\nu}^\X_i = \tilde{\nu}_i^{\X,y} \otimes \tilde{\nu}^\X.\]
 Here the $\tilde{\nu}_i^{\X,y}$ are uniquely defined for $\tilde{\nu}^\X$ almost every $y$ by the requirement
 \begin{align*}
   \int_{\X^2}f(x_i,y) \dd \tilde{\nu}^\X_i(x_i,y) = \int_{\X} \int_{\X} f(x_i,y) \dd \tilde{\nu}_i^{\X,y}(x_i) \dd \tilde{\nu}^\X(y), \quad  f \in C_b(\X^2).
 \end{align*}
 Here $C_b(\X)$ denotes the space of bounded continuous real functions on $\X$. Thus to each choice of $\tilde{\nu}^\X_1 \in \Pi(\nu^\X_1,\tilde{\nu}^\X),\dots,\tilde{\nu}^\X_K \in \Pi(\nu^\X_K,\tilde{\nu}^\X)$ we may define
 \begin{align*}
   \nu^\X &:= \tilde{\nu}_1^{\X,y} \otimes \dots \otimes \tilde{\nu}_K^{\X,y} \otimes \tilde{\nu}^{\X},
 \end{align*}
 and then subsequently define $\sigma^\X$ as the canonical projection of $\nu^\X$ onto $\X^K$. That is, $\sigma^\X$ is defined by the requirement
 \[ \int_{\X^K}f(x_1,\dots,x_K) \dd \sigma^\X(x) = \int_{\X^{K+1}}f(x_1,\dots,x_K) \dd \nu^\X(x,y), \quad  f \in C_b(\X^K),\]
 and heuristically may be regarded as ``integrating out'' the $y$-variable. Note that $x \in \X^K$.

 To obtain \eqref{eq:bary-est1} have arbitrary $\tilde{\nu}^\X \in \mathcal{P}(\X)$ and $\tilde{\nu}_i^\X \in \Pi(\nu_i^\X,\tilde{\nu}^\X)$. The definition of the Bregman-Wasserstein divergence implies
 \begin{align*}
   \sum_{i=1}^K \lambda_i \mathscr{B}(\nu_i,\tilde{\nu}) &\geq \sum_{i=1}^K \lambda_i \int_{\X^2}{\bf B}_{\Omega}(x_i,y) \dd \tilde{\nu}^\X_i(x_i,y) \\
   &= \int_{\X^{K+1}}\sum_{i=1}^{K}\lambda_i{\bf B}_{\Omega}(x_i,y) \dd \nu^\X (x,y).
 \end{align*}
 Using that the Euclidean barycenter is also the Bregman barycenter and then integrating out the $y$ variable we obtain
 \begin{align*}
   \sum_{i=1}^K \lambda_i \mathscr{B}(\nu_i,\tilde{\nu}) &\geq \int_{\X^{m+1}}\sum_{i=1}^{K}\lambda_i{\bf B}_{\Omega}(x_i,T(x)) \dd \nu^\X (x,y)\\
   &=  \int_{\X^{K}}\sum_{i=1}^{K}\lambda_i{\bf B}_{\Omega}(x_i,T(x)) \dd \sigma^\X(x).
 \end{align*}
 Finally since $\sigma^\X$ has, by construction, $i^\text{th}$ marginal $\nu^\X_i$ and $\gamma$ is the assumed minimizer of the multimarginal Kantorovich problem we have
 \[  \sum_{i=1}^K \lambda_i \mathscr{B}(\nu_i,\tilde{\nu})  \geq  \int_{\X^{K}}\sum_{i=1}^{K}\lambda_i{\bf B}_{\Omega}(x_i,T(x)) \dd \gamma^\X = \mathcal{K}(\gamma). \]
 Taking the infimum over $\tilde{\nu}$ yields \eqref{eq:bary-est1}.

 For \eqref{eq:bary-est2} first note since $\gamma^\X$ has $i^{\text{th}}$-marginal $\nu^\X_i$ and $\overline{\nu}^\X$ is defined as $T_{\#}\gamma^\X$, its clear that $\gamma_i^\X := (\pi_i,T)_{\#}\gamma^\X$ is in $\Pi(\nu^\X_i,\overline{\nu}^\X)$. Thus by definition of the Bregman-Wasserstein divergence
 \begin{align*}
   \mathscr{B}(\nu_i,\overline{\nu}) &= \inf_{\pi \in \Pi(\nu_i^\X,\nu^\X)} \int_{\X^2} {\bf B}_{\Omega}(x,y) \dd \pi(x,y)\\
                      &\leq \int_{\X^2} {\bf B}_{\Omega}(x,y) \dd \gamma^\X_i(x,y)= \int_{\X^K} {\bf B}_{\Omega}(x_i,T(x)) \dd \gamma^\X.
 \end{align*}
Summation weighted by $\lambda_i$ implies
 \[ \sum_{i=1}^K \lambda_i \mathscr{B}(\nu_i,\overline{\nu}) \leq \int_{\X^K} \sum_{i=1}^K \lambda_i\mathbf{B}_{\Omega}(x_i,T(x)) \dd \gamma^\X,\]
 which is \eqref{eq:bary-est2}, thereby completing the proof.
 \end{proof}

Now we obtain the existence and a characterization of the Bregman-Wasserstein barycenter as a consequence of Pass's work on multi-marginal optimal transportation \cite[Theorem 3.1]{P11}\footnote{We require compactness in Theorem \ref{thm:BW.barycenter.well.posed} since this is assumed in Pass's work.}.

 \begin{proof}[Proof of Theorem \ref{thm:BW.barycenter.well.posed}]
   Proposition \ref{prop:multimarginal-equiv} asserts that if $\gamma$ is a solution to the Kantorovich multimarginal problem then $\nu := T_{\#}\gamma$ is a Bregman-Wasserstein barycenter. In this proof we just check the conditions to apply \cite[Theorem 3.1]{P11} are met. That theorem asserts that the solution to the Monge problem is of the form $(x_1,G_2(x_1),\dots,G_K(x_1))$ for suitable functions $G_i,$m and existence is guaranteed by \cite[Theorem 2.2]{P11}.

   Making direct reference to the notation of \cite[Theorem 3.1]{P11} we show the conditions (I)-(V) are met. Condition (V), that the first marginal does not charge small sets, is met by assumption. For the remaining conditions we work in the primal coordinates and let
   \[ c(x_1,\dots,x_K):=\sum_{i=1}^K\lambda_i{\bf B}_{\Omega}(x_i,T(x)),\]
   denote the cost function associated to the multimarginal problem. Next we verify (II), the property of being $(1,K)$ twisted, that is that
   \[x_K \mapsto D_{x_1}c(x_1,x_2,\dots,x_K),\]
   is injective for fixed $x_1\dots,x_{K-1}$.
  A direct calculation yields
   \begin{align}
\label{eq:first-deriv}   D_{x_k}c(x_1,\dots,x_K) = \lambda_k\big[D\Omega(x_k)-D\Omega(T(x))\big]
   \end{align}
   Thus the injectivity of $D_{x_1}c(x_1,\dots,x_K)$ in $x_K$ follows from the associated injectivity of $D\Omega$ and $x_K \mapsto T(x_1,\dots,x_{K-1},x_K)$.

   The next condition to verify is (I), $(1,m)$-non-degeneracy, that is that the quantity $D_{x_1 x_K}c(x_1,\dots,x_K)$ is a nondegenerate as a linear operator (equivalently $\det D_{x_1 x_K}c \neq 0$ ). Continuing from \eqref{eq:first-deriv} compute
   \begin{equation}
     \label{eq:second-deriv}
     D_{x_k x_l}c(x_1,\dots,x_K) =  \lambda_{k}(\delta_{kl}D^2\Omega(x_k) - \lambda_l D^2\Omega(T(x))),
   \end{equation}
   where $\delta_{kl}$ is the Kronecker delta. That \eqref{eq:second-deriv} has nonzero determinant for $k=1, l=K$ follows from $k \neq l$, $\lambda_k,\lambda_l \neq 0$ and the positive definiteness of $D^2\Omega$.
   Next we verify Pass's condition (III), that a particular tensor, $T$, is
   nonnegative definite. The tensor is defined as follows: Fix arbitrary $x
   = (x_1,\dots,x_K) \in \X^K$, let $i \in \{2,3,\dots,K-1\}$, and let $x^{(i)}
   = (x_1^{(i)},\ldots,x_{i-1}^{(i)},x_i,x_{i+1}^{(i)},\ldots,x_K^{(i)}) \in \X^K$ be another given point, arbitrary except for the fact that its $i^{\text{th}}$ component agrees with that of $x$. The tensor is
   \begin{align*}
     T:= - \sum_{j=2}^{K-1}&\sum_{\substack{i=2\\i\neq j}}^{K-1}D^2_{x_ix_j}c + \sum_{i,j=2}^{K-1}(D^2_{x_ix_K}c)(D^2_{x_1x_K}c)^{-1}(D^2_{x_1x_j}c)(x)\\
     &+\sum_{i=2}^{K-1} \left[\text{Hess}_{x_i}c(x^{(i)})-\text{Hess}_{x_i}c(x)\right].
   \end{align*}
   Now, the bounds of the summation forbid repeated indices, so \eqref{eq:second-deriv} implies
   \[ (D^2_{x_ix_K}c)(D^2_{x_1x_K}c)^{-1}(D^2_{x_1x_j}c)(x)  = -\lambda_{i}\lambda_jD^2\Omega (T(x)),\]
   which for $i\neq j$ equals $D^2_{x_ix_j}c(x)$ and cancels with the first summation. Thus
      \begin{align*}
     T =  -\sum_{i=2}^{K-1}\lambda_{i}^2D^2\Omega(T(x))+\sum_{i=2}^{m-1} \left[\text{Hess}_{x_i}c(x^{(i)})-\text{Hess}_{x_i}c(x)\right].
      \end{align*}
      From \eqref{eq:second-deriv}, we have
      \[ \text{Hess}_{x_i}c(x) = \lambda_{i}D^2\Omega(x_i) - \lambda_{i}^2D^2\Omega(T(x)).\]
      From which we compute that
      \[ T = -\sum_{i=2}^{K-2}\lambda_i^2D^2\Omega(T(x^{(i)})), \]
      which by the uniform convexity of $\Omega$ is non-negative definite: exactly Pass's requirement.

      The final condition to check, by which we conclude the proof of the theorem, is the (Euclidean) convexity of the set
      \begin{equation}
        \label{eq:pass-conv-condn}
        \{(x_2,\dots,x_{K-1}) ;\text{ there is } x_K \text{ such that } c_{x_1}(x_1,x_2,\dots,x_K) = p_1 \}.
       \end{equation}
       By again employing \eqref{eq:first-deriv} we see the convexity of this set follows from the convexity of $\X$. Indeed, \eqref{eq:pass-conv-condn} can be expressed as the set of $(x_2,\dots,x_{K-1})$ for which there is $x_K$ such that
       \[ \sum_{i=2}^{K}\lambda_ix_i = v := D\Omega^*(D\Omega(x_1) - p_1/\lambda_1) - \lambda_1x_1,\]
       and one sees this set is convex by arguing straight from the definition of convexity.
 \end{proof}

\section{Additional proofs regarding the generalized dualistic structure} \label{sec:appendix.proofs}

\begin{proof}[Proof of Theorem \ref{thm:BW.conjugate.connections}]
Fix $\rho$ and  $\mathcal{V}_u,\mathcal{V}_v,\mathcal{V}_w \in T_{\rho}\mathcal{P}^\infty(\M)$ where without loss of generality $\Pi_\rho(u) = u$, and similarly for $v,w$.
  (i) This is similar to \cite[Lemma 2 and (4.7)]{L08}; for completeness we include the details. We begin with a technical claim: that tangent vectors are characterized by their action when integrated against smooth test functions. Namely, for each smooth compactly supported $\phi$ on $M$ define $F_\phi:\mathcal{P}^\infty(M) \rightarrow \mathbb{R}$  by $F_\phi(\rho \dd \vol) := \int_{\M} \phi \rho \dd \vol$.
We prove if $\mathcal{V}_v,\mathcal{V}_w \in T_{\rho}\mathcal{P}^\infty(M)$ satisfy
  \[ \mathcal{V}_vF_\phi(\rho \dd \vol) = \mathcal{V}_w F_\phi(\rho \dd \vol),\]
  for all such $\phi$, then $\mathcal{V}_v = \mathcal{V}_w$. Indeed, by assumption $v-w$ is in the $L^2(\rho \dd \vol)$ closure of the space of gradients of smooth functions. Thus there is a sequence of such $\phi_n$ with $\grad \phi_n \rightarrow v - w$ in $L^2(\rho \dd \vol)$. Then by assumption
  \[ 0 = (\mathcal{V}_v-\mathcal{V}_w)F_{\phi_n}(\rho \dd \vol) = \int\phi_n \divg[\rho(w-v)] \dd \vol = \int \langle \grad \phi_n , v-w \rangle \rho \dd \vol.\]
  Taking $n \rightarrow \infty$ yields $v = w$ as elements of $L^2(\rho \dd \vol)$.

  Now we prove the torsion free identity. Fix an arbitrary $\phi$ and first note
  \begin{align}
\nonumber    &(\nnabla_{\mathcal{V}_v}\mathcal{V}_w )F_\phi(\rho \dd \vol) - (\nnabla_{\mathcal{V}_w}\mathcal{V}_v) F_\phi(\rho \dd \vol) \\
&= \nonumber -\int_{\M} \phi \divg(\rho \Pi_\rho(\nabla_v w - \nabla_w v)) \dd \vol\\
\label{eq:torsion-1}    &= \int_{\M} \langle \grad \phi, [v,w]\rangle \rho \dd \vol,
  \end{align}
  where the first equality holds by definition and linearity, and the second by the torsion free property of $\nabla$. On the other hand we compute
  \begin{align*}
    \mathcal{V}_v\mathcal{V}_w F_\phi(\rho \dd \vol) &= \frac{\dd}{\dd\epsilon_1}\Big\vert_{\epsilon_1 = 0}\frac{\dd}{\dd\epsilon_2}\Big\vert_{\epsilon_2 = 0}\int_{\M} \phi \times [\rho - \epsilon_1 \divg( \rho v) - \\
                           &\quad\quad \epsilon_2 \divg\big[\big(\rho - \epsilon_1 \divg(\rho v)\big)w\big]\dd \vol\\
                           &= -\frac{\dd}{\dd\epsilon_1}\Big\vert_{\epsilon_1 = 0} \int_{\M} \phi \times \divg\big[\big(\rho - \epsilon_1 \divg(\rho v)\big)w\big]\dd \vol\\
                           & = \frac{\dd}{\dd\epsilon_1}\Big\vert_{\epsilon_1 = 0} \int_{\M} \langle \grad \phi , \rho - \epsilon_1 \divg\big(\rho v\big)w \rangle \dd \vol \\
                           &= -\int_{\M} \langle \grad \phi ,w\rangle \divg(\rho v) \dd \vol \\
    &= \int_{\M} [ \langle\overline{\nabla}_v \grad \phi , w \rangle + \langle \grad \phi , \overline{\nabla}_v w\rangle ] \rho \dd \vol.
  \end{align*}
  Deriving the corresponding identity for $ \mathcal{V}_w\mathcal{V}_v F_\phi(\rho \dd \vol)$ and subtracting (recall symmetry of the Hessian of $\phi$) we obtain
  \[ \llbracket \mathcal{V}_v,\mathcal{V}_w \rrbracket F_\phi(\rho \dd \vol) = \int_{\M} \langle \grad \phi, \overline{\nabla}_v w - \overline{\nabla}_wv \rangle \rho \dd \vol.  \]
    Since the Levi-Civita connection $\overline{\nabla}$ is torsion free, comparison with \eqref{eq:torsion-1} completes the proof of part (i) for the primal connection. The proof is unchanged with the primal connection replaced by the dual connection.

    (ii) We argue from the definition of $\mathcal{V}_u\mathfrak{g}(\mathcal{V}_v,\mathcal{V}_w)$. Indeed,
    \begin{align*}
      \mathcal{V}_u[\mathfrak{g}(\mathcal{V}_v,\mathcal{V}_w)(\rho)] &= \frac{\dd}{\dd\epsilon}\Big\vert_{\epsilon = 0}\mathfrak{g}(\mathcal{V}_v,\mathcal{V}_w)[\rho - \epsilon \divg(\rho u)]\\
                                  &= \frac{\dd}{\dd\epsilon}\Big\vert_{\epsilon = 0}\int_{\M} \langle v, w \rangle (\rho  - \epsilon \divg(\rho u))\dd \vol\\
                                  &= -\int_{\M} \langle v , w \rangle \divg (\rho u)  \dd \vol.
    \end{align*}
    Now using the defining property of the Riemannian divergence and the corresponding conjugacy property on the base space we obtain
    \begin{align*}
          \mathcal{V}_u[\mathfrak{g}(\mathcal{V}_v,\mathcal{V}_w)(\rho \dd \vol)] &= \int_{\M} u(\langle v, w \rangle) \rho \dd \vol \\
                      &= \int_{\M} [\langle \nabla_u v, w \rangle + \langle v, \nabla^*_u w \rangle]  \rho \dd \vol \\
                      &= \mathfrak{g}(\nnabla_{\mathcal{V}_u}\mathcal{V}_v, \mathcal{V}_w) + \mathfrak{g}(\mathcal{V}_v, \nnabla^*_{\mathcal{V}_u}\mathcal{V}_w).
    \end{align*}

    (iii) In accordance with the remark at the beginning of the proof it suffices to prove the equality holds when tested against linear functionals. Let $\mathcal{V}_v,\mathcal{V}_w \in T_\rho\mathcal{P}^\infty(M)$ be given. Directly from the definitions we have
    \begin{align*}
      &\frac{1}{2}(\nnabla_{\mathcal{V}_v}\mathcal{V}_w + \nnabla^*_{\mathcal{V}_v}\mathcal{V}_w)F_\phi(\rho) \\
      &= -\frac{1}{2}\int_{\X} \phi \divg ( \rho \nabla_vw) \dd \vol-\frac{1}{2}\int_{\X} \phi \divg ( \rho \nabla^*_vw) \dd \vol\\
      &= -\int_{\X} \phi \divg \Big[\rho \frac{1}{2}(\nabla_vw + \nabla^*_vw)\Big] \dd \vol.
    \end{align*}
    Then by the corresponding result on the base space, and the definition of the Levi-Civita connection on $\mathcal{P}^\infty(M)$, \eqref{eq:key-ident}, we obtain
    \begin{align*}
      \frac{1}{2}(\nnabla_{\mathcal{V}_v}\mathcal{V}_w + \nnabla^*_{\mathcal{V}_v}\mathcal{V}_w)F_\phi(\rho ) &= -\int_{\X} \phi \divg (\rho \overline{\nabla}_vw) \dd \vol\\
      &= \overline{\nnabla}_{\mathcal{V}_v}\mathcal{V}_wF_\phi(\rho).
    \end{align*}
    \end{proof}

\begin{proof}[Proof of Proposition \ref{cor:christoffel}]
  This is a direct calculation. Indeed we have
  \begin{align*}
    \GGamma_{\mathcal{V}_v,\mathcal{V}_w,\mathcal{V}_u}(\rho) &= \llangle \nnabla_{\mathcal{V}_{v}}\mathcal{V}_{w} , \mathcal{V}_{u} \rrangle_{\rho}  =  \int_{\M} \langle\Pi_\rho(\nabla_{v}w),u\rangle \rho \dd \vol.
  \end{align*}
  We may assume $\Pi_\rho(u) = u$, so that by properties of projection operators
  \begin{align*}
     \GGamma_{\mathcal{V}_v,\mathcal{V}_w,\mathcal{V}_u}(\rho) = \int_{\M} \langle\Pi_\rho(\nabla_{v}w),\Pi_\rho(u)\rangle \rho \dd \vol =  \int_{\M} \langle\nabla_{v}w,\Pi_\rho(u)\rangle \rho \dd \vol.
  \end{align*}
 Using again $\Pi_\rho(u) = u$ completes the proof. The proof for the dual connection is the same.
\end{proof}

\begin{proof}[Proof of Theorem \ref{thm:cubic.expansion}]
  The second order expansion is given by Proposition \ref{prop:BW.metric}. Thus, the result follows by computing the third derivative of $\mathscr{B}(\mu_0,\mu_t)$ at $t = 0$.
  Since $\mu_t$ is assumed a dual displacement we have
  \[\mu_t^{\Y} = [(1-t)D\Omega + t Df]_{\#}\mu_0^{\X},\]
  for an appropriate convex function $f$ defined on $\Omega$. Then in accordance with Lemma \ref{lem:displacement.interpolation.fields} $\dot{\mu_0} = \mathcal{V}_ {\grad \phi}$ for $\phi:= (f-\Omega)\circ\iota$ which in the primal coordinates satisfies $\grad \phi = (f_j-\Omega_j)\partial_j$.  Using Proposition \ref{prop:solving.BW.transport} and the self-dual representation of Bregman divergence we have
  \[\mathscr{B}(\mu_0,\mu_t) = \int_{\X} \left\{ \Omega(x) + \Omega^*((1-t)D\Omega + t Df) - x\cdot((1-t)D\Omega + t Df) \right\}  \dd \mu_0^{\X}.\]
  We compute
  \begin{align}
\label{eq:third-comp1}  \frac{\dd^3}{\dd t^3}\Big\vert_{t=0}\mathscr{B}(\mu_0,\mu_t)  &= \int_{\X} D_{y_iy_jy_k}\Omega^*(D\Omega(x))(\grad \phi)^i(\grad \phi)^j(\grad \phi)^k \ \dd \mu_0^{\X}.
  \end{align}

  On the other hand, using Corollary \ref{cor:christoffel}, we have
  \begin{align}
   \nonumber \mathfrak{T}_{\mathcal{V}_\phi,\mathcal{V}_\phi,\mathcal{V}_\phi} &=  \GGamma^*_{\mathcal{V}_\phi,\mathcal{V}_\phi,\mathcal{V}_\phi} - \GGamma_{\mathcal{V}_\phi,\mathcal{V}_\phi,\mathcal{V}_\phi}\\
\nonumber    &= \int_{\X} \langle(\nabla^*_{\grad \phi} - \nabla_{\grad \phi})\grad \phi, \grad \phi \rangle \dd \mu_0^{\X}\\
\nonumber    &= \int_{\X} (\grad \phi)^i(\grad \phi)^j(\grad \phi)^k \langle \nabla^*_{\partial_i}\partial_j-\nabla_{\partial_i}\partial_j,\partial_k \rangle  \dd \mu_0^{\X}\\
\label{eq:third-comp2}    &= \int_{\X} (\grad \phi)^i(\grad \phi)^j(\grad \phi)^k \Omega^*_{y_iy_jy_k} \dd \mu_0^{\X}.
  \end{align}
  In the final line we have used the coordinate expression for the cubic tensor \cite[eq. 6.32]{A16}. Comparison of \eqref{eq:third-comp1} and \eqref{eq:third-comp2} concludes the proof.
\end{proof}

To prove Theorem \ref{thm:sectional.curvature}, we first derive an expression of the curvature tensor which is similar to \cite[Theorem 1]{L08} though a particular identity using the Levi-Civita connection must be replaced by one involving the primal and dual connections. Set $N_\rho = I - \Pi_\rho$ to be the projection onto the orthogonal component of the tangent space.%

\begin{lemma}[Expression of curvature tensor]\label{lem:curv-tensor}
For $\mathcal{V}_1,\dots,\mathcal{V}_4$ as above we have the following identity relating the primal curvature on $\mathcal{P}^\infty(M)$ to the curvature on $M$ and the projection operators:
  \begin{align}
    \label{eq:curvature-expansion}    \llangle \mathfrak{R} &(\mathcal{V}_{1},\mathcal{V}_{2})\mathcal{V}_{3},\mathcal{V}_{4}\rrangle = \int_{\M}\Big[\langle R (\Phi_1,\Phi_2)\Phi_3,\Phi_4\rangle\\
    \nonumber                                          &- \langle N_\rho(\nabla _{\Phi_2}\Phi_3),N_{\rho}(\nabla _{\Phi_4}\Phi_1)\rangle+ \langle N_\rho(\nabla _{\Phi_1}\Phi_3),N_\rho(\nabla _{\Phi_4}\Phi_2)\rangle \\
   \nonumber &-\langle N_p(\nabla _{\Phi_2}\Phi_1),N_\rho(\nabla _{\Phi_4}\Phi_3)\rangle+ \langle N_\rho(\nabla _{\Phi_1}\Phi_2),N_\rho(\nabla _{\Phi_4}\Phi_3)\rangle\Big] \rho \dd \vol.
  \end{align}
  An analogous expression holds for the dual curvature tensor $\mathfrak{R}^*$.
\end{lemma}
\begin{proof}[Proof of Lemma \ref{lem:curv-tensor}]
  The proof is a direct computation of the terms in \eqref{eq:curv-def}.
  First, by the conjugacy relation \eqref{eq:wass-conjugate} we have
  \begin{align}
\nonumber    \langle\nnabla _{\mathcal{V}_{1}}\nnabla _{\mathcal{V}_{2}}\mathcal{V}_{3},\mathcal{V}_{4} \rangle &= \mathcal{V}_{1}(\llangle \nnabla _{\mathcal{V}_{2}}\mathcal{V}_{3},\mathcal{V}_{4}\rrangle) - \llangle \nnabla _{\mathcal{V}_{2}}\mathcal{V}_{3},\nnabla^*_{\mathcal{V}_{1}}\mathcal{V}_{4}\rrangle\\
                                         \nonumber        &= -\int_{\M}\langle\nabla _{\Phi_2}\Phi_3,\Phi_4\rangle\text{div}(\rho\Phi_1)\dd \vol \\
\nonumber    &\quad \quad- \int_{\M}\langle\Pi_\rho(\nabla _{\Phi_2}\Phi_3),\Pi_\rho(\nabla^*_{\Phi_1}\Phi_4)\rangle \rho \dd \vol.
\end{align}
  We integrate by parts in the first integral (using the conjugacy of the primal and dual connections), and note that in the second integrand only one of the $\Pi_\rho$ is necessary, to obtain
  \begin{align}
  \nonumber    \langle\nnabla _{\mathcal{V}_{1}}\nnabla _{\mathcal{V}_{2}}\mathcal{V}_{3},\mathcal{V}_{4} \rangle &= \int_{\M}\big[\langle \nabla _{\Phi_1}\nabla _{\Phi_2}\Phi_3 , \Phi_4\rangle +\langle \nabla _{\Phi_2}\Phi_3, \nabla^*_{\Phi_1}\Phi_4\rangle\big] \rho \dd \vol\\
\nonumber  &\quad\quad -\int_{\M}\langle\Pi_\rho(\nabla _{\Phi_2}\Phi_3),\nabla^*_{\Phi_1}\Phi_4\rangle \rho \dd \vol\\
                                          &= \int_{\M}\langle\nabla _{\Phi_1}\nabla _{\Phi_2}\Phi_3,\Phi_4\rangle  +\langle N_\rho(\nabla _{\Phi_2}\Phi_3),N_\rho(\nabla^*_{\Phi_1}\Phi_4)\rangle \rho \dd \vol.   \label{eq:curv-t1}
  \end{align}
In the final line we used that for an orthogonal projection $\langle N_\rho(v),w\rangle = \langle N_\rho(v),N_\rho(w)\rangle$.  Clearly we also obtain
  \begin{align}
    \label{eq:curv-t2}
    \langle\nnabla _{\mathcal{V}_{2}}\nnabla _{\mathcal{V}_{1}}\mathcal{V}_{3},\mathcal{V}_{4} \rangle&=  \int_{\M}\langle\nabla _{\Phi_2}\nabla _{\Phi_1}\Phi_3,\Phi_4\rangle  +\langle N_\rho(\nabla _{\Phi_1}\Phi_3),N_\rho(\nabla^*_{\Phi_2}\Phi_4)\rangle \rho \dd \vol.
  \end{align}

Finally, employing the symmetry of the Riemannian Hessian we compute
  \begin{align}
    \nonumber \langle\nnabla _{\llbracket\mathcal{V}_{1},\mathcal{V}_{2}\rrbracket}\mathcal{V}_{3},\mathcal{V}_{4}\rangle &= \int_{\M}\langle\nabla _{\Pi_\rho([\Phi_1,\Phi_2])}\Phi_3,\Phi_4\rangle \rho \dd \vol\\
    \nonumber &= \int_{\M} \langle\nabla _{\Phi_4}\Phi_3,\Pi_\rho(\nabla_{\Phi_1}\Phi_2 - \nabla_{\Phi_2}\Phi_1)\rangle \rho \dd \vol\\
    \label{eq:curv-t3}   &= \int_{\M}\langle\nabla _{[\Phi_1,\Phi_2]}\Phi_3,\Phi_4\rangle- \langle N_\rho(\nabla_{\Phi_1}\Phi_2 - \nabla_{\Phi_2}\Phi_1),\nabla _{\Phi_4}\Phi_3\rangle \rho \dd \vol.
                             \end{align}
                             Substituting \eqref{eq:curv-t1}, \eqref{eq:curv-t2}, and  \eqref{eq:curv-t3} into \eqref{eq:curv-def} proves the Theorem in conjunction with the following identity \eqref{eq:raise-lower}. We claim
                             \begin{align}
                              N_\rho(\nabla _{\Phi_2}\Phi_1) = - N_\rho(\nabla^*_{\Phi_1}\Phi_2). \label{eq:raise-lower}
                             \end{align}
                             Note for the Levi--Civita connection this reduces to the anti-symmetry property $N_\rho(\overline{\nabla} _{\Phi_2}\Phi_1) = - N_\rho(\overline{\nabla}_{\Phi_1}\Phi_2)$. To derive \eqref{eq:raise-lower} note
                             \begin{equation}
                               \label{eq:grad-ident}
                               \grad \langle\Phi_1,\Phi_2\rangle = \overline{\nabla}_{\Phi_2}\Phi_1 + \overline{\nabla}_{\Phi_1}\Phi_2 = \nabla _{\Phi_2}\Phi_1 + \nabla^*_{\Phi_1}\Phi_2.
                            \end{equation}
  Here the second equality follows by a direct calculation in coordinates using the Christoffel symbols for the Levi-Civita and dual connections. Then because the right hand side of \eqref{eq:grad-ident} is the gradient of a smooth function it vanishes under the normal projection and we obtain \eqref{eq:raise-lower}.                            \end{proof}

\begin{proof}[Proof of Theorem \ref{thm:sectional.curvature}]
The proof for $d = 1$ is a straightforward consequence of \eqref{eq:curvature-expansion} and that in one dimension $\Pi_\rho = I$. Indeed, this follows because in dimension one each smooth vector field is the gradient of a smooth function. Subsequently $N_\rho = 0$, which combined with dual flatness ($R = R^{*} = 0$) and \eqref{eq:curvature-expansion} establishes part (i).

Next, to prove the result for $d \geq 2$, we claim,
\begin{equation}
  \label{eq:curvature-identity}
  \llangle \mathfrak{R}^*(\mathcal{V}_{1},\mathcal{V}_{2})\mathcal{V}_{3},\mathcal{V}_{4}\rrangle = -\llangle \mathfrak{R} (\mathcal{V}_{1},\mathcal{V}_{2})\mathcal{V}_{4},\mathcal{V}_{3}\rrangle.
\end{equation}
This identity is a consequence of \eqref{eq:curvature-expansion}, the corresponding identity on $M$, and \eqref{eq:raise-lower}. Indeed by substituting $\langle R^*(\Phi_1,\Phi_2)\Phi_3,\Phi_4 \rangle = -\langle R(\Phi_1,\Phi_2)\Phi_4 ,\Phi_3 \rangle$ (proved in \cite[Proposition 8.14]{CU14}) into \eqref{eq:curvature-expansion} and using \eqref{eq:raise-lower} to exchange $\nabla$ for $\nabla^*$ we obtain
\begin{align*}
   \llangle \mathfrak{R} &(\mathcal{V}_{1},\mathcal{V}_{2})\mathcal{V}_{3},\mathcal{V}_{4}\rrangle = \int_{\M}\Big[-\langle R^* (\Phi_1,\Phi_2)\Phi_4,\Phi_3\rangle\\
    \nonumber                                          &- \langle N_\rho(\nabla^* _{\Phi_3}\Phi_2),N_{\rho}(\nabla^* _{\Phi_1}\Phi_4)\rangle+ \langle N_\rho(\nabla^* _{\Phi_3}\Phi_1),N_\rho(\nabla^* _{\phi_2}\Phi_4)\rangle \\
  \nonumber &-\langle N_p(\nabla^* _{\Phi_1}\Phi_2),N_\rho(\nabla^* _{\Phi_3}\Phi_4)\rangle+ \langle N_\rho(\nabla^* _{\Phi_2}\Phi_1),N_\rho(\nabla^* _{\Phi_3}\Phi_4)\rangle\Big] \rho \dd \vol \\
  &=-\llangle \mathfrak{R}^* (\mathcal{V}_{1},\mathcal{V}_{2})\mathcal{V}_{4},\mathcal{V}_{3}\rrangle
\end{align*}
Here the final inequality is a consequence of the dual version of Lemma \ref{lem:curv-tensor}.

Choosing $\mathcal{V}_3=\mathcal{V}_2,\mathcal{V}_4=\mathcal{V}_1$ and employing \eqref{eq:curvature-identity} we have
\begin{align*}
  \mathfrak{K}^*(\mathcal{V}_1,\mathcal{V}_2) &=\frac{\llangle \mathfrak{R}^*(\mathcal{V}_{1},\mathcal{V}_{2})\mathcal{V}_{2},\mathcal{V}_{1} \rrangle}{\llangle \mathcal{V}_1,\mathcal{V}_1 \rrangle  \ \llangle \mathcal{V}_2,\mathcal{V}_2\rrangle  - \llangle \mathcal{V}_1,\mathcal{V}_2 \rrangle^2}= \frac{-\llangle \mathfrak{R}(\mathcal{V}_{1},\mathcal{V}_{2})\mathcal{V}_{1},\mathcal{V}_{2} \rrangle}{\llangle \mathcal{V}_1,\mathcal{V}_1 \rrangle  \ \llangle \mathcal{V}_2,\mathcal{V}_2\rrangle  - \llangle \mathcal{V}_1,\mathcal{V}_2 \rrangle^2},
\end{align*}
where the second to last equality uses skew symmetry in the first and second argument.
\end{proof}

\section*{Acknowledgment}
We thank Robert McCann and Soumik Pal for helpful discussions. Some preliminary results were presented in the session ``Optimal Transport and Applications" of the PRIMA Congress 2022, and we thank the organizers and participants for their comments. Cale Rankin is supported by the Fields Institute for Research in Mathematical Sciences. He thanks the institute and the organizers of the Thematic Program on Nonsmooth Riemannian and Lorentzian Geometry. Leonard Wong's research is partially supported by an NSERC Discovery Grant (RGPIN-2019-04419) and a Connaught New Researcher Award. %
\bibliographystyle{abbrv}
\bibliography{geometry.ref,applications}

\end{document}